\documentclass[11pt]{amsart}
\usepackage{amssymb}
\usepackage{amsfonts}
\usepackage{amsmath}
\usepackage{graphicx}
\usepackage{xcolor}
\usepackage{amsmath}
\usepackage{mathrsfs}
\usepackage{stmaryrd}
\usepackage{epsfig,color}
\usepackage{blindtext}
\usepackage{enumerate}
\usepackage{hyperref}
\usepackage{url}
\usepackage{bbm}
\usepackage{nicefrac,mathtools}
\usepackage{bm}   
\usepackage{tabularx}
\DeclareGraphicsExtensions{.pdf,.jpeg,.png}
\usepackage{epstopdf}
\usepackage{cancel} 
\usepackage[normalem]{ulem} 
\usepackage{verbatim} 
\usepackage{scalerel}
\usepackage{enumitem}

\usepackage{color}
\usepackage[msc-links, lite]{amsrefs}
\usepackage{geometry}
\usepackage{stackengine,scalerel}

\usepackage{tikz-cd}

\newcommand\ttimes{\mathbin{\ThisStyle{\ensurestackMath{%
  \stackengine{-1\LMpt}{\SavedStyle\times}
  {\SavedStyle_{\hstretch{.9}{\mkern1mu\sim}}}{O}{c}{F}{T}{S}}}}}

\newtheorem{theorem}{Theorem}[section]
\newtheorem{thmx}{Theorem}
 
\newtheorem*{problem}{Problem}

\newtheorem{proposition}[theorem]{Proposition}
\newtheorem{lemma}[theorem]{Lemma}
\newtheorem{corollary}[theorem]{Corollary}

\newtheorem{claim}[]{Claim}

\theoremstyle{definition}
\newtheorem{definition}[theorem]{Definition}
\theoremstyle{remark}
\newtheorem{remark}[theorem]{Remark}

\numberwithin{equation}{section}

\newcommand{\dv}{\mathrm{div}}

\newcommand{\lb}[1]{\langle#1\rangle}

\newcommand{\mf}{\mathbf}
\newcommand{\mb}{\mathbb}
\newcommand{\mc}{\mathcal}
\newcommand{\ms}{\mathscr}
\newcommand{\mk}{\mathfrak}
\newcommand{\mr}{\mathrm}
\newcommand{\oli}{\overline}

\newcommand{\wti}{\widetilde}
\newcommand{\res}{\scaleobj{1.75}{\llcorner}}

\newcommand{\A}{\mathcal A}
\newcommand{\C}{\mathcal C}

\newcommand{\Vol}{\mathrm{Vol}}
\newcommand{\Area}{\mathrm{Area}}
\newcommand{\Id}{\mathrm{Id}}

\newcommand{\dist}{\operatorname{dist}}

\newcommand{\lc}{\scalebox{1.8}{$\llcorner$}}

\newcommand{\rom}[1]{\expandafter\romannumeral #1}
\newcommand{\Rom}[1]{\uppercase\expandafter{\romannumeral #1}}

\DeclareMathOperator{\Ric}{Ric}
\DeclareMathOperator{\Hess}{Hess}
\DeclareMathOperator{\Diff}{Diff}
\DeclareMathOperator{\spt}{spt}
\DeclareMathOperator{\interior}{int}
\DeclareMathOperator{\closure}{Clos}

\DeclareMathOperator{\Graph}{Graph}

\DeclareMathOperator{\VC}{\mc V\C}

\title{Existence of four minimal spheres in $S^3$ with a bumpy metric}

\author{Zhichao Wang}
\address{Shanghai Center for Mathematical Science, 2005 Songhu Road, Fudan University, Shanghai, 200438, China}
\email{zhichao@fudan.edu.cn}

\author{Xin Zhou}
\address{Department of Mathematics, 531 Malott Hall, Cornell University, Ithaca, NY 14853, USA}
\email{xinzhou@cornell.edu}

\begin{document}

\begin{abstract}
We prove that in the three dimensional sphere with a bumpy metric or a metric with positive Ricci curvature, there exist at least four distinct embedded minimal two-spheres. This confirms a conjecture of S. T. Yau in 1982 for bumpy metrics and metrics with positive Ricci curvature. The proof relies on a multiplicity one theorem for the Simon-Smith min-max theory. 
\end{abstract}

\maketitle

\setcounter{section}{-1}


\section{Introduction}
In his famous 1982 Problem Section \cite{Yau82}, S. T. Yau posed the following problem. 
\begin{problem}[\cite{Yau82}*{Problem 89}]
Prove that there are four distinct embedded minimal spheres in any manifold diffeomorphic to $S^3$.  
\end{problem}

In this paper, we provide a solution of this problem when the metric is bumpy or has positive Ricci curvature. Note that a metric $g$ on a given closed manifold $M$ is called bumpy if every closed embedded minimal hypersurface is non-degenerate. White \cite{Whi91} proved that the set of bumpy metrics is generic in the Baire sense. Our first main result is as follows, and we refer to Theorem \ref{thm:4 minimal spheres in bumpy} for a more general statement and the proof therein.

\begin{thmx}
\label{thm:theoremA}
Assume that $g$ is a bumpy metric or a metric with positive Ricci curvature on $S^3$. Then there exist at least four distinct embedded minimal two-spheres in $(S^3, g)$. 
\end{thmx}

Around the time when Yau first posed this problem, Simon-Smith \cite{Smith82} proved the existence of at least one embedded minimal sphere in $S^3$ with an arbitrary metric, using a variant of the min-max theory for minimal hypersurfaces developed by Almgren \cites{Alm62,Alm65} and Pitts \cite{Pi}; see also Schoen-Simon \cite{SS} and Colding-De Lellis \cite{Colding-DeLellis03}. Later on, White \cite{Whi91}, using degree methods, proved the existence of at least two embedded minimal spheres when the metric has positive Ricci curvature, and at least four embedded minimal spheres when the metric is sufficiently close to the round metric. Recently, Haslhofer-Ketover \cite{HK19} proved the existence of at least two embedded minimal spheres for bumpy metrics, by combining the Simon-Smith min-max theory with the mean curvature flow; a key ingredient of their proof is the Catenoid estimates first appeared in Ketover-Marques-Neves \cite{KMN16}. We also remark that branched immersed minimal spheres were obtained by Sacks-Ulenbeck \cite{Sacks-Uhlenbeck81} using min-max theory for harmonic maps; see also Colding-Minicozzi \cite{Colding-Minicozzi08b}. 

The motivation of Yau's conjecture is tightly related to the topology of the space of embedded spheres in $S^3$. By Hatcher's proof of the Smale conjecture \cite{Hat83}*{Appendix (14)}, the space of embedded spheres in $S^3$ deformation retracts to the space of great spheres, which is homeomorphic to $\mathbb{RP}^3$, so the area functional when restricted to this space, should have four nontrivial critical points, that is, embedded minimal spheres. One can simply apply the Simon-Smith min-max theory to the four naturally defined homotopy classes of sweepouts associated with the $\mb{RP}^3$-structure; see Section \ref{subsec:four sweepouts} for detailed discussions. However, the major challenge is that the min-max theory may produce minimal spheres counted with integer multiplicity, so we may not obtain new minimal spheres when applying to new sweepouts. As the major advancement of this paper, we prove a new multiplicity one theorem for the Simon-Smith min-max theory. We refer to Theorem \ref{thm:classical multiplicity one} for the detailed statement and the proof therein.

\begin{thmx}
\label{thm:a rough multiplicity one theorem}
Let $(M, g)$ be a closed, orientable, three dimensional Riemannian manifold. Then the min-max varifold associated with any homotopy class of smooth sweepouts of a fixed orientable genus-$\mk g_0$ surface $\Sigma_0$ is supported on a pairwise disjoint collection of connected, closed, embedded, minimal surface $\{\Gamma_j\}_{j=1}^N$ with integer multiplicities $\{m_1, \cdots, m_N\}$, so that
\begin{itemize}
    \item if $\Gamma_j$ is unstable and two-sided, then $m_j = 1$, and
    \item if $\Gamma_j$ is one-sided, then its connected double cover is stable.
\end{itemize}
Moreover the weighted total genus of $\Sigma_0$ is bounded by $\mk g_0$. 
\end{thmx}
\begin{remark}
We remark that, as compared with the Multiplicity One Theorem \cite{Zhou19}*{Theorem C} for the Almgren-Pitts theory where all non-degenerate components have multiplicity one, Theorem \ref{thm:a rough multiplicity one theorem} only shows that unstable components have multiplicity one. The proof of Theorem \ref{thm:theoremA}, when stable minimal spheres exist, follows by combining a variant of Theorem \ref{thm:a rough multiplicity one theorem} with A. Song's cylindrical manifold construction \cite{Song18}. 
\end{remark}


We have observed tremendous development of the Almgren-Pitts min-max theory since the celebrated resolution of the Willmore Conjecture by Marques-Neves \cite{MN14}. In particular, Yau's another famous problem \cite{Yau82}*{Problem 88} on the existence of infinitely many closed minimal surfaces was solved by combining the works of Marques-Neves \cite{MN17} and Song \cite{Song18}. A key ingredient in this program was the volume spectrum first introduced by Gromov \cite{Gro88} and later by Guth \cite{Guth09}. The Weyl law for the volume spectrum, proved by Liokumovich-Marques-Neves \cite{LMN16}, has led to surprising density and equidistribution results for closed minimal hypersurfaces, by Irie-Marques-Neves \cite{IMN17} and Marques-Neves-Song \cite{MNS17} respectively; see also \cite{Song-Zhou21}. After the resolution of the Multiplicity One Conjecture by the second-named author \cite{Zhou19} (see also \cite{CM20}), Marques-Neves finished their program on establishing a Morse theory for the area functional \cite{MN18}; see also \cites{Marques-Neves16, Marques-Montezuma-Neves20}.  We refer to the survey articles \cites{Marques-Neves21, Zhou22} for more detailed history on this exciting field. 
After we finished the work, there was a nice application of our multiplicity one theorem by Liokumovich-Ketover in their proof of Smale's conjecture for lens spaces \cite{Liokomovich-Ketover23}.

\subsection{Sketch of the proof}
We first describe the proof for Theorem \ref{thm:a rough multiplicity one theorem}. We follow similar strategy as the proof of the Multiplicity One Conjecture in the Almgren-Pitts setting \cite{Zhou19}, that is, to use prescribing mean curvature functionals $\A^h$ \eqref{eq:Ah functional} to approximate the area functional. However, there are several major challenges, mainly caused by the fact that the local $\A^h$-isotopy minimizing problem  has only $C^{1,1}$-solutions, as proved recently by Sarnataro-Stryker \cite{SS23}, and the solution may have a large portion of self-touching set, where the mean curvature vanishes. The challenges and new ideas invoked are summarized as follows.
\begin{itemize}
    \item A notion of strong $\A^h$-stationarity is introduced to prevent the $\A^h$-min-max solutions to degenerate to integer multiple of minimal surfaces in our special scenario. 
    
    \item A new notion of critical set is introduced to include pairs of varifolds and Caccioppoli sets, and a stronger tightening process is developed to show that all pairs in the min-max critical sets are $\A^h$-stationary. 
    
    \item A new scheme of proving $C^{1,1}$ regularity of the $\A^h$-min-max solutions is established, using chains of replacements, without invoking unique continuation. 
    
    \item A new argument for compactness is used, without proving the optimal Morse index bounds, to show the $\A^{\varepsilon h}$-min-max solutions converge in $C^{1,1}_{loc}$ to embedded minimal surfaces, when $\varepsilon\to 0$. The genus bound of the limit minimal surfaces also follows easily from our construction of the approximating $\A^{\varepsilon h}$-min-max solutions.
    
    \item A special prescribing function $h$ can be chosen, which is $L^2$-orthogonal to the first eigenfunctions of all possible limit minimal surfaces (which is a finite set by bumpiness), so that when combined with the strong $\A^{\epsilon h}$-stationarity, one can show that the limit minimal surface with multiplicity greater than one must be stable. 
\end{itemize}

Now we dip into some details of these new ideas. The $\A^h$-functional is defined for a pair of $C^{1,1}$-surface $\Sigma$ and a Caccioppoli set $\Omega$ enclosed by $\Sigma$ \eqref{eq:Ah version2}. A $C^{1,1}$-pair $(\Sigma, \Omega)$ which is merely stationary for the $\A^h$-functional could be just an even multiple of a closed minimal surface with $\Omega = \emptyset$ or $M$. A pair $(\Sigma, \Omega)$ is called strongly $\A^h$-stationary (see Definition \ref{def:strong one-sided stationarity}) if near a touching point, moving the top/bottom sheet away from other sheets increases the $\A^h$-functional up to the first order, or simply the top/bottom sheet solves the obstacle problem for $\A^h$ w.r.t. other sheets. Therefore, if a sequence of such pairs $\{(\Sigma_k, \Omega_k)\}_{k\in \mb N}$ converges to a minimal surface $\Sigma_\infty$ with multiplicity $m\geq 2$, and if the prescribing functions $h_k$ change sign along $\Sigma_\infty$, we know that $\Sigma_k$ cannot be an $m$-sheeted minimal surface. Otherwise, by Corollary \ref{cor:description of Omega near touching point}, $\Omega$ coincides with $M$ near points where $h>0$ and with $\emptyset$ near points where $h<0$, which is not possible.  

We set up our min-max problem using smooth sweepouts of surfaces of a fixed genus in the relative setting following \cites{Colding-DeLellis03, Zhou19}. We then extend the critical sets (Definition \ref{def:minimizing seq, min-max seq, and critical set}) so as to include pairs $(V, \Omega)$ in our newly defined $\VC$-space (Section \ref{ss:Ah functional and VC space}). As the main advantage to get back the $\Omega$-factor in the critical set, it makes sense to say $\A^h$-stationarity for critical pairs. In particular, we re-design the tightening process in Section \ref{SS:tightening} to show that every critical pair is $\A^h$-stationary. This is an improvement as compared with the previous CMC/PMC min-max theory \cites{ZZ17, ZZ18}, where it was only shown that the critical varifolds have uniformly bounded first variation.

We then introduce the notion of $\A^h$-almost minimizing using smooth isotopies, and use the combinatorial argument originally due to Almgren-Pitts to show the existence of a min-max pair, which is $\A^h$-almost minimizing in every small annuli. As a by-product of this step, we can show that there exists an integer $L=L(m)$ depending only on the dimension of the parameter space, such that for any $L(m)$-number of admissible collection of concentric annuli, the min-max pair is stable in at least one of them. This will play the role of Morse index upper bound when proving the desired compactness results. 

To prove the $C^{1,1}$-regularity and strong $\A^h$-stationarity for the min-max pairs, we do induction on density (which are integers). Denote by $\mc S(V, \leq m-1)$ and $\mc S(V, m)$ the subsets of the support $\spt\|V\|$ with density less than $m$ or equal to $m$ respectively. Suppose we have proved the regularity in $\mc S(V, \leq m-1)$. Fix a $q\in\mc S(V, m)$ and a small neighborhood $U_0$; we cover $\closure(U_0)\cap \mc S(V, m)$ by balls of a small radius $\bm r$, and then take successive $\A^h$-replacements over these small balls. By a gluing procedure similar to \cite{ZZ17}*{\S 6}, we can obtain a strongly $\A^h$-stationary and stable $C^{1,1}$ replacement in $U_0$. Letting the radii $\bm r\to 0$, the union of these $\bm r$-balls will converge to $\closure(U_0)\cap \mc S(V, m)$, and by the compactness theory for stable $C^{1,1}$ surfaces, these replacements will converge to a pair $(V^*, \Omega^*)$ which is $C^{1,1}$ and strongly $\A^h$-stationary and stable in $U_0$. Note that $\spt\|V^*\|\subset \spt\|V\|$. To show that $V^*$ is identical to $V$ in $U_0$, we can first choose $U_0$ small enough so that the volume ratio of $V$ for some fixed radius $s\gg \bm r$ centered at any point in $\closure(U_0)\cap \mc S(V, m)$ is close to $m$. Since $V^*$ and $V$ have the same mass in any open set containing $\closure(U_0)$, we can show the density of $V^*$ at any point in $\closure(U_0)\cap \mc S(V, m)$ is less than or equal to $m$ (using the monotonicity formula and the fact that $V^*$ is an integral varifold), and hence $V^* = V$ and the $C^{1,1}$-regularity and strong $\A^h$-stationarity are proved in $\mc S(V, m)$. Note that we do not need any unique continuation type result in this procedure.

Next, we consider those $C^{1,1}$ min-max pairs $\{(\Sigma_k, \Omega_k)\}_{k\in\mb N}$ associated with the $\A^{\varepsilon_k h}$-functionals for some sequence $\varepsilon_k\to 0$. The $C^{1,1}_{loc}$ convergence (away from a finite set) to a smoothly embedded minimal surface $\Sigma_\infty$ follows from the fact that each $(\Sigma_k, \Omega_k)$ is stable in at least one of any $L(m)$-admissible concentric annuli. We can further choose a diagonal sequence (of closed embedded surfaces of a given genus) converging to $\Sigma_\infty$ which are $(\A^{\varepsilon_{k(j)} h}, \epsilon_j, \delta_j)$-almost minimizing in small annuli ($\epsilon_j, \delta_j\to 0$). By choosing $h$ to vanish along $\Sigma_\infty$ except for finitely many sub-disks, this diagonal subsequence is $(\epsilon_j, \delta_j)$-almost minimizing for the area functional away from these sub-disks, and hence the curve lifting argument in \cites{Smith82, DeLellis-Pellandini10, Ketover13} can be applied so as to show the desired genus bound for $\Sigma_\infty$.

To prove Theorem \ref{thm:a rough multiplicity one theorem}, if the convergence $\Sigma_k\to\Sigma_\infty$ has multiplicity greater than one, we can construct a nontrivial nonnegative supsolution to the Jacobi operator $L_{\Sigma_\infty}$ of $\Sigma_\infty$. First by our choice of the prescription function $h$, we know that $\Sigma_k$ cannot be an integer multiple of minimal surfaces, so we can consider the height difference between the top and bottom sheets. Unlike in the proof of the Almgren-Pitts Multiplicity One Theorem \cite{Zhou19}, where both the top and bottom sheets satisfy the prescribing curvature equation, the mean curvature in our current setting may vanish in a large subset. Nevertheless, the key observation is that the height difference, which is nontrivial, will converge after normalization to a nontrivial weak supersolution $\varphi_\infty\geq 0$ of 
\[ L_{\Sigma_\infty}\varphi_\infty \geq 2\cdot c\cdot  h|_{\Sigma_\infty} \varphi_\infty,  \]
for some constant $c\geq 0$. Since we can take $h|_{\Sigma_\infty}$ to be $L^2$-orthogonal to the first eigenfunction $\phi_1$ of $L_{\Sigma_\infty}$, applying the integration-by-part formula will show the stability of $\Sigma_\infty$, that is, $\lambda_1(L_{\Sigma_\infty})\geq 0$. 

All the above arguments provide an outline of proof for Theorem \ref{thm:a rough multiplicity one theorem} for relative min-max. One can extend this to classical min-max (for free homotopy of sweepouts) using the double-cover-lifting argument as in \cite{Zhou19}.

\medskip
We now describe how to prove Theorem \ref{thm:theoremA} using Theorem \ref{thm:a rough multiplicity one theorem}. If $(S^3, g)$ does not contain any stable minimal spheres, Theorem \ref{thm:a rough multiplicity one theorem} applied to the four natural families of sweepouts of embedded spheres gives at least four distinct embedded minimal spheres with multiplicity one.  When $(S^3, g)$ admits a non-degenerate stable minimal sphere $S^2$, we can cut $(S^3, g)$ along this sphere to obtain a three-ball $(B^3, g)$ with a stable minimal boundary. We then glue the boundary $S^2$ with a cylindrical end modeled by $S^2\times [0, \infty)$ as Song \cite{Song18}. The Simon-Smith min-max theory when applied to compact approximations of this noncompact  Lipschitz manifold produces embedded minimal spheres in $(B^3, g)$ in the same way as \cite{Song18}. One can check that there are two family of sweepouts of embedded spheres in these compact approximations with uniformly bounded width, using the method in \cite{HK19}. We can prove a variant of Theorem \ref{thm:a rough multiplicity one theorem} in this non-compact setting, so as to produce at least two distinct embedded minimal spheres in $(B^3, g)$. Since there are two such three balls after cutting, we prove the existence of at least five embedded minimal spheres in this case.

\subsection{Outline of the paper}

We collect preliminary materials in Section \ref{sec:preliminaries}; then we set up the relative min-max problem and establish the tightening process in Section \ref{sec:min-max and tightening}. In Section \ref{sec:almost minimizing}, we introduce several concepts related to the almost minimizing property and prove the existence of almost minimizing pairs. Section \ref{sec:regularity of min-max pairs} is the first essential part of this paper, where we prove the $C^{1,1}$ regularity of min-max pairs.  In Section \ref{sec:passing to limit}, we prove the convergence of min-max pairs to minimal surfaces for a sequence $\{\varepsilon_k h\}_{k\in \mb N}$ with $\lim \varepsilon_k =0$, as well as genus bound for the limit minimal surface. Section \ref{sec:existence of supersolution}, another essential part of this paper, is devoted to the construction of supersolutions. In Section \ref{sec:multiplicity one}, we prove Theorem \ref{thm:a rough multiplicity one theorem}, and in Section \ref{sec:existence of minimal spheres}, we prove Theorem \ref{thm:theoremA}.

\subsection*{Acknowledgement} The authors would like to thank Professor Richard Schoen and Professor Gang Tian for their interest in this work. Z. W. would like to thank Professor Jingyi Chen and Professor Ailana Fraser for their support and encouragement. X. Z. is supported by NSF grant DMS-1945178, and an Alfred P. Sloan Research Fellowship. The authors would also like to thank Robert Haslhofer and Daniel Ketover for pointing out a gap in their original proof of \cite[Theorem 5.2]{HK19} regarding the Lusternik-Schnirelmann inequality and showing their erratum. 

\section{Preliminaries}\label{sec:preliminaries}
In this part, we collect all necessary preliminary materials. After introducing basic notations, we will introduce the $\VC(M)$ space as the closure of the natural embedding $\C(M) \to \mc V(M)\times \C(M)$ under the product metric in Section \ref{ss:Ah functional and VC space}. Then we will define $C^{1,1}$-almost embedded surfaces and boundaries, the $\A^h$-functional and its associated stationarity in Section \ref{SS:C11 almost embedded surface}, and then a crucial notion of strong $\A^h$-stationarity and its corollaries in Section \ref{ss:strong Ah stationarity}. Lastly, we will recall stable compactness for $\A^h$-stationary boundaries in Section \ref{ss:stability and compactness}, and the $\A^h$-isotopy minimizing problem in Section \ref{ss:isotopy minimizing problem}. 

\subsection*{Notations}
We will not specify the ambient manifold to the three-sphere before the last section. 
\begin{itemize}
    \item $(M^3, g)$ denotes a closed, oriented, 3-dimensional Riemannian manifold isometrically embedded in some $\mb R^L$, and $U\subset M$ an open subset ($U$ may be equal to $M$).
    \item $\mr{An}(p; s, r)$ ($p\in M$, $0<s<r$) denotes an annulus given by $B(p, r)\setminus \closure(B(p, s))$. 
    \item $h\in C^{\infty}(M)$ denotes a smooth mean curvature prescribing function.
    \item $\mathcal C(M)$ or $\mathcal C(U)$ denotes the space of sets $\Omega\subset M$ or $\Omega\subset U\subset M$ with finite perimeter (Caccioppoli set); see \cite{Si}*{\S 14}.
    \item $\mc V(M)$ or $\mc V(U)$ denotes the space of $2$-varifolds in $M$ or $U$.
    \item $\mathfrak X(U)$ denotes the space of smooth vector fields compactly supported in $U$.
    \item $\Diff_0(M)$ denotes the connected component of the diffeomorphism group of $M$ containing identity, and $\mk{Is}(U)$ denotes the set of isotopies of $M$ supported in $U$.
    \item A collection of connected $C^1$-embedded surfaces $\{\Gamma^i\}_{i=1}^\ell\subset U$ with $\partial\Gamma^i\cap U=\emptyset$ is said to be ordered, denoted by 
    \[\Gamma^1\leq \cdots\leq \Gamma^\ell,\] if for each $i$, $\Gamma^i$ separates $U$ into two connected components $U^i_+, U^i_-$, ($U\setminus \Gamma^i = U^i_+\sqcup U^i_-$), such that $\Gamma^j\subset \closure(U^i_-)$ for $j=1, \cdots, i-1$, and $\Gamma^j\subset \closure(U^i_+)$ for $j=i+1, \cdots, \ell$.
\end{itemize}

\subsection{{$\mc A^h$}-functional and $\mc{VC}$-space}\label{ss:Ah functional and VC space}

The {\em prescribing mean curvature functional} associated with $h\in C^{\infty}(M)$ in \cite{ZZ18}*{(0.2)} naturally extends to all pairs $(V, \Omega)\in \mc V(M)\times \mc C(M)$ as:
\begin{equation}\label{eq:Ah functional}
\A^h(V, \Omega) = \|V\|(M) - \int_\Omega h \,\mr d\mc H^3.
\end{equation}

We can naturally define push-forward by diffeomorphisms in $\mc V(M)\times \C(M)$. Note that given $(V, \Omega)\in \mc V(M)\times \C(M)$ and $F: (-\epsilon, \epsilon)\times M\to M$ a smooth map with $F^t\in \Diff_0(M)$, then $t\mapsto \A^h\big(F^t_\#(V, \Omega)\big)$ is a smooth function. As a result, we can define $\A^h$-stationarity for pairs in $\mc V(M)\times \C(M)$.

\begin{definition}[$\A^h$-stationary pairs]
A pair $(V, \Omega)\in \mc V(M)\times \C(M)$ is \textit{$\A^h$-stationary in $U$} if for any $X\in\mathfrak X(U)$ with $\phi^t$ the associated flow, 
\begin{equation}
\label{eq:1st variation of Ah}
\begin{split}
    \delta \mathcal A^h_{V, \Omega}(X)  :&= \frac{d}{dt}\Big{|}_{t=0}\A^h\big(\phi^t_\#(V, \Omega)\big)\\
    &= \int_{G_2(M)} \dv_S X(x) \,\mr dV(x, S) - \int_{\partial\Omega}\lb{X, \nu_{\partial\Omega}} h \,\mr d\mu_{\partial\Omega} = 0.
\end{split} 
\end{equation}
An $\A^h$-stationary pair $(V, \Omega)$ is \textit{$\A^h$-stable in U}, if for any $X\in\mathfrak X(U)$,
\begin{equation}\label{eq:2nd variation of Ah}
    \delta^2 \A^h_{V, \Omega}(X, X) : =  \frac{d^2}{dt^2}\Big{|}_{t=0} \A^h\big( \phi^t_\#(V, \Omega) \big) \geq 0.
\end{equation}
\end{definition}

\begin{remark}
    Note that $\A^h$ and its variations $\delta\A^h$, $\delta^2\A^h$ are also naturally defined in $\mc V(U)\times \C(U)$.
\end{remark}

We are mainly interested in a subspace of $\mc V(M)\times \C(M)$ which arises as the completion under weak topology of the ``diagonals" $\Delta(M)=\{ (|\partial\Omega|, \Omega)\in \mc V(M)\times \C(M): \Omega\in \C(M)\}$.

\begin{definition}
Motivated by Almgren's VZ-space \cite{Alm65}, we have the following.
\begin{enumerate}
    \item The \textit{$\VC$-space}, denoted by $\VC(M)$, is the space of all pairs $(V, \Omega)\in \mc V(M)\times\C(M)$ such that there is a sequence $\{\Omega_k\}\subset \C(M)$ with $|\partial\Omega_k|\to V$ in $\mc V(M)$ and $\Omega_k\to \Omega$ in $\C(M)$.

    \item Given two pairs $(V, \Omega)$ and $(V', \Omega')$ in $\VC(M)$, the \textit{$\ms F$-distance between them} is
    \begin{equation*}
        \ms F\big( (V, \Omega), (V', \Omega') \big):= \mf F(V, V') + \mc F(\Omega, \Omega'),
    \end{equation*}
    where $\mf F$ and $\mc F$ are respectively the varifold $\mf F$-metric and the flat metric. 
\end{enumerate}
\end{definition}

The next lemma follows from the lower semi-continuity of measure in weak convergence.
\begin{lemma}
For every $(V,\Omega)\in \VC(M)$, we have that    $\spt(\partial\Omega)\subset \spt(\|V\|)$, and $\|\partial\Omega\|\leq \|V\|$ as measures.
\end{lemma}

Then it is clear that we have the following.
\begin{lemma}\label{lem:Ah stationary pair in VC has bounded 1st variation}
Suppose that $(V, \Omega)\in \VC(M)$ is an $\A^h$-stationary in $U$. Denote $c=\sup_{x\in M}|h(x)|$. Then $V$ has $c$-bounded first variation in $U$.
\end{lemma}

We also have the following.
\begin{lemma}
    Given any $L>0$, the space
    \begin{equation}\label{eq:AL}
        A^L = \{(V, \Omega)\in \VC(M): \|V\|(M)\leq L \}
    \end{equation}
    is a compact metric space under the $\ms F$-metric.
\end{lemma}

It is also clear that for fixed $X\in \mk X(M)$, the map $(V, \Omega)\to \delta \A^h_{V, \Omega}(X)$ is continuous under the $\ms F$-metric, so we have the following.
\begin{lemma}
    The set 
    \begin{equation}\label{eq:AL0}
        A^L_0=\{(V, \Omega)\in A^L: (V, \Omega) \text{ is $\A^h$-stationary}\}
    \end{equation} 
    is a compact subset of $A^L$ under the $\ms F$-metric.
\end{lemma}

\subsection{$C^{1,1}$ almost embedded $h$-surfaces}\label{SS:C11 almost embedded surface}

\begin{definition}[$C^{1,1}$ almost embedding]
\label{def:C11 almost embedding}
We say that a $C^{1, 1}$ immersed surface $\phi: \Sigma\to U$ with $\phi(\partial \Sigma)\cap U=\emptyset$ is a {\em $C^{1,1}$ almost embedded surface} in $U$, if at any point $p\in \phi(\Sigma)$ near which $\phi$ is not an embedding, there exists a neighborhood $W\subset U$ of $p$, such that
\begin{itemize}
\item $\Sigma\cap\phi^{-1}(W)$ is a disjoint union of connected components $\sqcup_{i=1}^\ell \Gamma^i$;
\item $\phi: \Gamma^i \to W$ is a $C^{1, 1}$ embedding for each $i$;
\item for each $i$, any other component $\phi(\Gamma^j)$, ($j\neq i$), lies on one side of $\phi(\Gamma^i)$ in $W$.
\end{itemize}
We will denote $\phi(\Sigma)$ by $\Sigma$ and $\phi(\Gamma^i)$ by $\Gamma^i$ in appropriate context. The subset of points in $\Sigma$ where $\Sigma$ is not embedded will be called the {\em touching set}, and denoted by $\mathcal S(\Sigma)$. The {\em regular set} $\Sigma\setminus\mathcal S(\Sigma)$ will be denoted by $\mathcal R(\Sigma)$.
\end{definition}

\begin{remark}
Note that the touching set $\mathcal S(\Sigma)$ is a relatively closed subset of $\Sigma$, and regular set $\mathcal R(\Sigma)$ is relatively open in $\Sigma$. 
\end{remark}

\begin{definition}[$C^{1,1}$ boundary]\label{def:c11 boundary}
We say that a $C^{1,1}$ almost embedded surface $\phi: \Sigma\to  U$ is a {\em $C^{1,1}$ (almost embedded) boundary} in $U$, if $\Sigma$ is oriented, and there exists $\Omega\in \C(U)$, such that 
\begin{equation}\label{eq:Sigma and Omega}
\phi_{\#}(\llbracket\Sigma\rrbracket) = \partial\Omega \text{ as 2-currents in $U$};
\end{equation}
here $\llbracket\Sigma\rrbracket$ denotes the fundamental class of $\Sigma$.
\end{definition}

\begin{lemma}\label{lem:choice of outer normal}
Let $(\Sigma, \Omega)$ be a $C^{1,1}$-boundary in $U$. Then there exists a natural choice of unit normal $\nu_\Sigma$ of $\Sigma$ (as an immersed surface), such that if $\Omega\notin\{\emptyset, U\}$, then $\nu_\Sigma$ coincides with $\nu_{\partial\Omega}$ along $\spt(\partial\Omega)$. Moreover, if $\Sigma$ decomposes to ordered sheets $\Gamma^1\leq \cdots\leq \Gamma^\ell$ in any open subset $W\subset U$, then $\nu_\Sigma$ must alternate orientations along $\{\Gamma^i\}$.   
\end{lemma}
\begin{proof}
The orientation of $\Sigma$ induces a choice of unit normal $\nu_\Sigma$. If a connected component $\Sigma_0$ of $\Sigma$ (as an immersed surface) intersects  $\spt(\partial\Omega)$, we may possibly flip $\nu_\Sigma$ to $-\nu_\Sigma$ to let $\nu_\Sigma = \nu_{\partial\Omega}$ along $\spt(\partial\Omega)$. If this connected component $\Sigma_0$ has an integer multiplicity (as a subset of $U$), we can order these sheets by keeping $\nu_\Sigma$ in the first sheet and flipping $\nu_\Sigma$ alternatively for other sheets.  If a connected component $\Sigma_0$ does not intersect $\spt(\partial\Omega)$, it must have an even multiplicity by \eqref{eq:Sigma and Omega}, and we can arbitrarily order them by flipping $\nu_\Sigma$ alternatively. Thus $\nu_\Sigma$ has been chosen. 

We now check that the orientations of ordered sheet decomposition $\{\Gamma^i\}$ must alternate. Write $\nu$ for $\nu_\Sigma$. Assume for contradiction that $\nu|_{\Gamma^{i+1}}$ and $\nu|_{\Gamma^i}$ point to the same direction. If $\Gamma^{i+1}$ is not identical to $\Gamma^i$, this violates the assumption \eqref{eq:Sigma and Omega}. Assume now $\Gamma^{i+1}=\Gamma^i$. To show that $\nu$ alternates, we need to enlarge the open subset $W$ and track the connected components containing $\Gamma^{i+1}=\Gamma^i$ until either one sheet separates from the other, or we find a connected component of $\Sigma$ with multiplicity greater than one. The first case follows from the former discussion, and the later case follows from our choice of $\nu_\Sigma$ above. 
\end{proof}

\
The $\A^h$-functional is naturally defined on a $C^{1, 1}$-boundary as follows:
\begin{equation}\label{eq:Ah version2}
\mathcal A^h(\Sigma, \Omega) = \mc H^2 (\Sigma) - \int_{\Omega} h \,\mr d\mc H^3.
\end{equation}

\begin{lemma}\label{lem:enhanced 1st variation of Ah}
Let $(\Sigma, \Omega)$ be a $C^{1,1}$-boundary in $U$. For any $X\in \mk X(U)$, the first variation is
\begin{equation}
\label{eq:1st variation of Ah2}
\delta \mathcal A^h_{\Sigma, \Omega}(X) = \int_\Sigma \dv_\Sigma X -\lb{X, \nu_\Sigma} h \,\mr d\mc H^2.  
\end{equation}
\end{lemma}
\begin{proof}
Note that by \eqref{eq:1st variation of Ah}, the integral of $\lb{X, \nu_\Sigma}$ in $\delta\A^h$ is only defined on $\partial\Omega$. Nevertheless, we can use the $C^{1,1}$-structure to rewrite it over $\Sigma$. Indeed, we only need to check \eqref{eq:1st variation of Ah2} locally. Given any $p\in\Sigma$, there is a neighborhood $W\subset U$ of $p$, such that $\Sigma$ decomposes into ordered sheets $\Gamma^1\leq \cdots\leq \Gamma^\ell$. For each $\Gamma^i$, we choose $\Omega^i$ to be the connected component of $W\setminus\Gamma^i$ such that $\nu_\Sigma = \nu_{\partial\Omega^i}$. Then by the Constancy Theorem \cite{Si}*{26.27}, we know that $\llbracket \Omega\res W \rrbracket - \sum_{i=1}^\ell \llbracket\Omega^i\rrbracket = m\llbracket W \rrbracket$ for some integer $m\in \mb Z$. Therefore, for any $X\in \mk X(W)$, we know $\delta\A^h_{\Sigma, \Omega}(X) = \sum_{i=1}^\ell \delta \A^h_{\Gamma^i, \Omega^i}(X)$, which is exactly \eqref{eq:1st variation of Ah2}.
\end{proof}

\begin{definition}[$C^{1,1}$ $h$-boundary]
\label{def:C11 h-boundary}
A $C^{1,1}$-boundary $(\Sigma, \Omega)$ in $U$ is called a {\em $C^{1,1}$ (almost embedded) $h$-boundary} in $U$, 
if $(\Sigma, \Omega)$ is $\A^h$-stationary in $U$; 
that is, for any $X\in\mathfrak X(U)$, $\delta \mathcal A^h_{\Sigma, \Omega}(X) = 0$.
\end{definition}

\begin{lemma}\label{lem:regular part is PMC}
Assume that $(\Sigma, \Omega)$ is a $C^{1,1}$ $h$-boundary. Then the regular set $\mathcal R(\Sigma)$ is smoothly embedded, and its mean curvature $H$ (w.r.t. the unit normal $\nu_\Sigma$) is prescribed by $h$; that is,
\[ H = h|_{\Sigma}, \quad \text{ on } \mathcal R(\Sigma). \]
\end{lemma}
\begin{proof}
It follows from the first variation formula \eqref{eq:1st variation of Ah2} and standard elliptic regularity theory. 
\end{proof}

\subsection{Strong $\A^h$-stationarity}
\label{ss:strong Ah stationarity}
Near a touching point, while the above notion says that the sheets as a union is stationary for $\A^h$, a relatively stronger notion says that the top and bottom sheets are stationary for $\A^h$ w.r.t. deformations pointing away from all other sheets. This is the following {\em strongly $\A^h$-stationary property}.

\begin{definition}[Strong $\A^h$-stationarity]\label{def:strong one-sided stationarity}
A $C^{1, 1}$ $h$-boundary $(\Sigma, \Omega)$ is said to be {\em strongly $\A^h$-stationary in $U$}, if the following holds: 

For every $p\in \mc S(\Sigma)\cap U$, that is, $\ell:= \Theta^2(\Sigma, p)\geq 2$, 
there exists a small neighborhood $W\subset U$ of $p$, and decomposition $\Sigma\cap W = \cup_{i=1}^\ell \Gamma^i$ into $\ell\geq 2$ connected disks with a natural ordering $\Gamma^1\leq \cdots \leq \Gamma^\ell$. Denote by $W^1$ and $W^\ell$ the bottom and top components of $W\setminus \Sigma$. We require for $i=1, \ell$ and all $X\in \mk X(W)$ pointing into $W^i$ along $\Gamma^i$,
\begin{equation}\label{eq:1-sided stationary1:new}
\begin{split}
\delta \A^h_{\Gamma^i, W^i} (X) \geq 0, &\quad \text{ when $W^i\subset \Omega$}, \\
\delta \A^h_{\Gamma^i, W\setminus W^i}(X) \geq 0, &\quad \text{ when $W^i\cap \Omega=\emptyset$}.
\end{split}
\end{equation}
\end{definition}

\begin{remark}
This notion simply means that moving the top/bottom sheet of $\Sigma$ near $p$ away from all other sheets increases the $\A^h$-functional up to the first order. If a $C^{1,1}$-boundary $(\Sigma, \Omega)$ is the limit of an isotopic $\A^h$-minimizing sequence of embedded surfaces, then $(\Sigma, \Omega)$ is strongly $\A^h$-stationary; see Theorem \ref{thm:interior regularity}.
\end{remark}

The strongly $\A^h$-stationary property can deduce more information of $\mc S(\Sigma)$ as follows. 

\begin{lemma}\label{lem:singular part of strong stationary h-boundary}
As above, let $(\Sigma, \Omega)$ be a strongly $\A^h$-stationary, $C^{1,1}$ $h$-boundary in $U$. 
\begin{enumerate}[label=\roman*)]
\item \label{item:minimal around odd pt}If $p\in\mc S(\Sigma)$ and $\Theta^2(\Sigma, p)$ is odd, then there exists a neighborhood $W$ of $p$, such that $\Sigma\cap W$ is a minimal surface with multiplicity $\Theta^2(\Sigma, p)$, and $h(x)=0$ for all $x\in \Sigma\cap W$.
\item The generalized mean curvature of $\Sigma$ as an immersion vanishes $\mc H^2$-almost everywhere in $\mc S(\Sigma)$.
\end{enumerate}
\end{lemma}
\begin{remark}
Item \ref{item:minimal around odd pt} implies that if $h\neq 0$ in a neighborhood of $p\in \mc S(\Sigma)$, then $\Theta^2(\Sigma, p)$ is an even number.    
\end{remark}

\begin{proof}
We continue to use notations in Definition \ref{def:strong one-sided stationarity}. Let $p\in \mc S(\Sigma)$ with $\Theta^2(\Sigma,p)=\ell\geq 2$. Assume that the ordered sheets $\Gamma^1\leq \cdots \leq \Gamma^\ell$ in $W\subset U$ are graphs of $C^{1,1}$-functions $u^1\leq \cdots \leq u^\ell$ over a common domain $\mc W\subset \mb R^2$ with 
\[ u^1(0)=\cdots =u^\ell(0)=0.  \]
We first prove the first half of Item i), which describes the structure of $\mc S(\Sigma)$ with odd density.
\begin{claim}\label{claim:odd density is embedded}
If $p\in \mc S(\Sigma)$ and $\ell$ is odd, then $\Sigma\cap W$ is an embedded disk with multiplicity $\ell$.
\end{claim}
\begin{proof}[Proof of Claim \ref{claim:odd density is embedded}]
By Lemma \ref{lem:choice of outer normal}, the orientations of $\Gamma^i$ must alternate, and hence only one of $W^1, W^\ell$ lies in $\Omega$ (as $\ell$ is odd). Assume without loss of generality $W^1\subset \Omega$ and $W^\ell\cap \Omega = \emptyset$. Choose the unit normal vector fields $\nu^1, \nu^\ell$ along $\Gamma^1, \Gamma^\ell$ pointing away from $\Omega$ respectively. By \eqref{eq:1-sided stationary1:new}, we have
\[ \int_{\Gamma^1} \dv_{\Gamma^1} X - h\cdot \lb{X, \nu^1} \geq 0, \quad \text{ for all $X\in \mk X(W)$ with $\lb{X, \nu^1}_{\Gamma^1}\leq 0$}, \]
\[ \int_{\Gamma^\ell} \dv_{\Gamma^\ell} X - h\cdot \lb{X, \nu^\ell} \geq 0, \quad \text{ for all $X\in \mk X(W)$ with $\lb{X, \nu^\ell}_{\Gamma^\ell}\geq 0$}. \]
This implies that the generalized mean curvatures (w.r.t. $\nu^1, \nu^\ell$ respectively) satisfy:
\begin{equation}\label{eq:mean curvature inequalities1:new}
 H_{\Gamma^1}\leq h|_{\Gamma^1} \quad \text{ and }\quad H_{\Gamma^\ell}\geq h|_{\Gamma^\ell}.   
\end{equation}
Note that $\nu^1, \nu^\ell$ all point upward as $\ell$ is odd. Subtracting the two inequalities in \eqref{eq:mean curvature inequalities1:new}, the height difference $\varphi = u^\ell-u^1 \in C^{1,1}(\mc W)$ satisfies a differential inequality almost everywhere:
\[ L_{\mc W} \varphi \geq h(x, u^\ell(x)) - h(x, u^1(x)) = c(x)\varphi(x), \]
where $L_{\mc W}$ is a positive elliptic operator on $\mc W\subset \mb R^2$. Since $\varphi \geq 0$ and $\varphi = 0$ somewhere, by the Harnack estimates for strong solutions \cite{GT}*{Theorem 9.22}, we must have $\varphi \equiv 0$. 
This proves Claim \ref{claim:odd density is embedded}.
\end{proof}

We now determine the generalized mean curvature of each slice on $\mc S(\Sigma)$. By basic function theory applying to the functions $u^1\leq \cdots \leq u^\ell$, we know that the Hessians $\{\Hess u^i\}$ are identical almost everywhere along $\{u^1=\cdots=u^\ell\}$, and so the generalized mean curvature $H^i$ of $\Gamma^i$ (w.r.t. a common unit normal) are identical almost everywhere along $\Gamma^1\cap\cdots\cap\Gamma^\ell$. We next show that these generalized mean curvatures are zero almost everywhere in $\Gamma^1\cap\cdots\cap\Gamma^\ell$.  

Since $(\Sigma,\Omega)$ is $\mc A^h$-stationary and $\Gamma^i$ is a $C^{1, 1}$-surface, we have by \eqref{eq:1st variation of Ah2},
\begin{equation}\label{eq:form wrt g:new}
\sum_{i=1}^\ell\int_{\Gamma^i}\dv X\,\mr d\mc H^2 = \sum_{i=1}^\ell\int_{\Gamma^i}(-1)^{i-1}h(x)\langle X,\nu^i\rangle\,\mr d\mc H^2(x), \quad \forall X\in \mk X(W),
\end{equation}
where $\nu^i$ denotes the upward-pointing unit normal of $\Gamma^i$. Also the generalized mean curvature $H^i$ of $\Gamma^i$ (w.r.t. $\nu^i$) satisfies 
\[ \int_{\Gamma^i}\dv X\,\mr d\mc H^2=\int_{\Gamma^i}H^i\langle X,\nu^i\rangle \,\mr d\mc H^2, \quad \forall X\in \mk X(W).  \]
This together with \eqref{eq:form wrt g:new} gives that 
\begin{equation}\label{eq:h and H:new} 
\sum_{i=1}^\ell H^i(x) = \sum_{i=1}^\ell(-1)^{i-1}h(x),
\end{equation}
for $\mc H^2$-a.e. $x\in \Gamma^1\cap\cdots\cap\Gamma^\ell$. 
Recall that for $\mc H^2$-a.e. $x\in \Gamma^1\cap\cdots\cap\Gamma^\ell$, we have
\[  H^1(x)= \cdots =H^\ell(x),  \quad \text{and } \quad \nu^1(x) = \cdots = \nu^\ell(x).\]

\medskip
\begin{itemize}
\item {\noindent\em If $\ell$ is even}, we also know that the sum of the right hand side of \eqref{eq:h and H:new} vanishes $\mc H^2$-a.e. on $\Gamma^1\cap\cdots\cap\Gamma^\ell$. 
By \eqref{eq:h and H:new} again, we have that for $\mc H^2$-a.e. $x\in \Gamma^1\cap\cdots\cap\Gamma^\ell$,
\begin{gather*}
H^1(x)= \cdots = H^\ell (x) = 0.
\end{gather*} 

\item {\noindent\em If $\ell$ is odd}, then by Claim \ref{claim:odd density is embedded}, 
We have $\Gamma^1=\cdots=\Gamma^\ell =: \oli\Gamma$ in $W$. Then \eqref{eq:h and H:new} gives that for $\mc H^2$-a.e. $x\in \oli\Gamma$,
\[ H^1(x) = \cdots = H^\ell(x) = \frac{1}{\ell} h(x).\] 
Together with \eqref{eq:mean curvature inequalities1:new}, we have that for $\mc H^2$-a.e. $x\in\oli\Gamma$,
\[  h(x)\geq H^1(x)=\frac{1}{\ell} h(x)=H^\ell(x)\geq h(x).  \]
It follows that for $\mc H^2$-a.e. $x\in \oli\Gamma$,
\[    H^1(x)=\cdots=H^\ell(x)=h(x)=0.\]
\end{itemize}
This finishes the proof of $\mc H^2$-a.e. vanishing of $H$ on $\{\Theta^2(\Sigma,p)=\ell\}$ for $\ell\geq 2$, and hence on all $\mc S(\Sigma)$. 
\end{proof}

By combining Lemma \ref{lem:regular part is PMC} and Lemma \ref{lem:singular part of strong stationary h-boundary}, we have the following characterization of the mean curvature of strongly $\mc A^h$-stationary $h$-boundaries. 

\begin{corollary}\label{cor:mean curvature charaterization}
As above, let $\nu$ be the unit outer normal of $\Sigma$ induced by $\Omega$ by Lemma \ref{lem:choice of outer normal}. 
Then the generalized mean curvature $H$ of $\Sigma$ w.r.t. $\nu$ satisfies:
\[
H(p) = \left\{\begin{array}{ll}
     h(p) & \text{ when } p\in\mc R(\Sigma)\cap U  \\
     0 & \text{ for $\mc H^2$-a.e. } p\in \mc S(\Sigma)\cap U 
\end{array}\right ..
\]
\end{corollary}

We can also deduce the following important corollary of strong $\A^h$-stationarity, which asserts when the top/bottom sheets can have touching subsets.

\begin{proposition}\label{prop:mean curvature compared with pmc function}
As above, given $p\in \Sigma\cap U$, assume that $\Sigma$ decomposes into ordered sheets $\Gamma^1\leq \cdots \leq \Gamma^\ell$ in a neighborhood $W\subset U$ of $p$. Then the following holds (note that all generalized mean curvatures are defined w.r.t. $\nu$ induced from $\Omega$ by Lemma \ref{lem:choice of outer normal}):
\begin{enumerate}
    \item\label{item:case h>0} Assume $h>0$ in $W$. 
    \begin{enumerate}[label=\roman*)]
        \item\label{item:h>0:top} If $\Omega$ does not contain the region above $\Gamma^\ell$, then $\Gamma^\ell$ belongs to $\mc R(\Sigma)$ and $H^\ell = h|_{\Gamma^\ell}$. 

        \item If $\Omega$ contains the region above $\Gamma^\ell$, then $\Gamma^\ell$ may contain a subset of $\mc S(\Sigma)$, and in this case $H^\ell \leq  h|_{\Gamma^\ell}$. 
    \end{enumerate}

    \item\label{item:case h<0} Assume $h<0$ in $W$.
    \begin{enumerate}[label=\roman*)]
        \item If $\Omega$ contains the region above $\Gamma^\ell$, then $\Gamma^\ell$ belongs to $\mc R(\Sigma)$ and $H^\ell = h|_{\Gamma^\ell}$.

        \item If $\Omega$ does not contain the region above $\Gamma^\ell$, then $\Gamma^\ell$ may contain a subset of $\mc S(\Sigma)$, and in this case $H^\ell \geq h|_{\Gamma^\ell}$. 
    \end{enumerate}
\end{enumerate}
\end{proposition}

\begin{remark}
    Since we may flip the ordering, all the above statements for $\Gamma^\ell$ have corresponding statements for $\Gamma^1$. For instance, in Case \eqref{item:case h>0}(i), if $\ell$ is an even number, then $\Omega$ does not contain the region below $\Gamma^1$, so $\Gamma^1$ belongs to $\mc R(\Sigma)$ and $H^1 = h|_{\Gamma^1}$; if $\ell$ is odd, $\Omega$ contains the region below $\Gamma^1$, so $\Gamma^1$ may contain a subset of $\mc S(\Sigma)$, and in this case we only have $H^1 \leq h|_{\Gamma^1}$.
\end{remark}

\begin{proof}
    By possibly simultaneously flipping $(\Omega, h)$ to $(M\setminus \Omega, -h)$, we only need to prove Case \ref{item:case h>0}. 
    For Case \ref{item:case h>0}(\rom{1}), if for contrary $\Gamma^\ell\cap \mc S(\Sigma)\neq \emptyset$, then by the assumption of orientations and Corollary \ref{cor:mean curvature charaterization}, we must have
    \[ H^\ell \leq h|_{\Gamma^\ell}, \text{ and } H^\ell =0 < h|_{\Gamma^\ell} \text{ along } \Gamma^\ell\cap \mc S(\Sigma). \]
    However, under these assumptions, for any $X\in \mk X(W)$ with $\lb{X, \nu}> 0$ along $\Gamma^\ell$, (note that $\nu$ points into $W^\ell$ using notations in Definition \ref{def:strong one-sided stationarity}),
    we have
    \[ \delta \A^h_{\Gamma^\ell, W^\ell}(X) = \int_{\Gamma^\ell} (H^\ell-h) \lb{X, \nu} d\mc H^2 <0, \]
    which contradicts with the strongly $\A^h$-stationary assumption.

    Case \ref{item:case h>0}(\rom{2}) follows directly from Corollary \ref{cor:mean curvature charaterization} using the above argument. 
\end{proof}

Finally, we have the following direct corollary which forbids a strongly $\A^h$-stationary, $C^{1,1,}$ $h$-boundary to collapse to an even multiple of minimal surfaces in certain situation. 

\begin{corollary}\label{cor:description of Omega near touching point}
As above, assume that $p$ lies in the interior of $\mc S(\Sigma)$ and $h\neq 0$ near $p$. By Lemma \ref{lem:singular part of strong stationary h-boundary}, there exists some neighborhood $W\subset U$ of $p$, such that $\Sigma\cap W = m [\Gamma]$ for some $m \in 2\mb N$ and some minimal surface $\Gamma$. Then 
\begin{enumerate}[label=\roman*)]
    \item if $h>0$ in $W$, then $\Omega \cap W = W$;
    \item if $h<0$ in $W$, then $\Omega \cap W = \emptyset$.
\end{enumerate}
\end{corollary}

\subsection{Stability and compactness}
\label{ss:stability and compactness}
In this subsection, we recall the compactness of stable $C^{1,1}$ $h$-boundaries in \cite{SS23}*{\S 16 and \S 17}, which are natural generalizations of \cites{SSY, Schoen83, ZZ18} to the $C^{1,1}$-PMC setting. We will further show that the strongly $\A^h$-stationary property is preserved under suitable notion of convergence.

\begin{definition}[stable $C^{1,1}$ $h$-boundary]
Let $(\Sigma, \Omega)$ be a $C^{1,1}$ $h$-boundary in an open set $U\subset M$. $(\Sigma,\Omega)$ is \textit{stable} in $U$ if for any $X\in \mk X(U)$ with $\phi^t$ the associated flow (see also \eqref{eq:2nd variation of Ah}),
\[ \frac{d^2}{dt^2}\mc A^h\big(\phi^t(\Sigma, \Omega)\big)\geq 0. \]
If in addition $(\Sigma,\Omega)$ is strongly $\mc A^h$-stationary, by direct calculation, this is equivalent to,
\[ \int_\Sigma |\nabla ^\perp X^\perp|^2 -\mr{Ric}(X^\perp,X^\perp)-|A^\Sigma|^2|X^\perp|^2\,\mr d\mc H^2\geq \int_{\partial\Omega}\langle X^\perp,\nabla h\rangle \langle X,\nu\rangle \,\mr d\mc H^2, \]
where $X^\perp$ is the normal part of $X$ w.r.t. $\Sigma$, $\mr{Ric}$ is the Ricci curvature of $(M, g)$, and $A^\Sigma$ is the second fundamental form of $\Sigma$ (as an immersion) w.r.t. the unit outward normal $\nu$.
\end{definition}

\begin{proposition}
\label{prop:compactness}
Let $h_j, h\in C^2(M)$ be such that $\|h_j-h\|_{C^2}\to 0$. Let $\{(\Sigma_j,\Omega_j)\}_{j\in\mb N}$ be a sequence of stable $C^{1,1}$ $h_j$-boundary in $U$ satisfying $\mc H^2(\Sigma_j\cap U)\leq \Lambda$ for some $\Lambda>0$. Then there exists a stable $C^{1,1}$ $h$-boundary $(\Sigma, \Omega)$, so that $(\Sigma_j,\Omega_j)$ subsequently converges to $(\Sigma,\Omega)$ in the following sense:
\begin{enumerate}
    \item $\Sigma_j$ converges to $\Sigma$ in $U$ as varifolds and also in the sense of $C^{1,\alpha}_{loc}$ for all $\alpha\in(0,1)$; 
    \item $\Omega_j$ converges to $\Omega$ as currents in $\C(U)$. 
\end{enumerate}
Furthermore, we also have the following.
\begin{enumerate}[label=(\roman*)]
    \item  If $h\equiv 0$, then $\Sigma_j$ converges to $\Sigma$ in $U$ in the $C^{1,1}_{loc}$ topology.
    \item If $(\Sigma_j,\Omega_j)$ is strongly $\mc A^{h_j}$-stationary in $U$, then $(\Sigma,\Omega)$ is strongly $\mc A^h$-stationary in $U$.
\end{enumerate}

\end{proposition}
\begin{proof}
The subsequential convergence (1) and (2) are essentially proved by Sarnataro-Stryker \cite{SS23}*{Lemma 16.3 and Theorem 17.3}. 
Note that they require relatively stronger assumption (their Theorem 1.1) on the regularity of $(\Sigma_j, \Omega_j)$ to 
derive the stability inequality \cite{SS23}*{(17.1)} and show it can be passed to limit under $C^{1,\alpha}_{loc}$-convergence. 
This part can be replaced by the following fact: for a fixed $X\in \mk X(U)$, the maps $(\Sigma, \Omega, h) \mapsto \delta\A^h_{\Sigma, \Omega}(X)$ and $(\Sigma, \Omega, h) \mapsto \delta^2\A^h_{\Sigma, \Omega}(X, X)$ are both continuous w.r.t. the product topology in $\mc V(M)\times \C(M)\times C^2(M)$. Therefore, the blowup limit in their proof of Theorem 17.3 is stationary and stable. All other parts therein work well under our stable $C^{1,1}$ $h_j$-boundary assumptions. By this fact, we also know that $(\Sigma, \Omega)$ is a stable $C^{1,1}$ $h$-boundary.

Assume that $h\equiv 0$, then $\Sigma$ is a smooth minimal surface, and $\Sigma_j$ can be written as ordered $C^{1,1}$ graphs over $\Sigma$ for all large $j$. Moreover, the $C^{1,1}_{loc}$ norm is bounded by the $C^{1,\alpha}$ norm which converges to $0$ by Proposition \ref{prop:C11 estimates}; see also \cite{SS23}*{Corollary 11.2}. Thus we conclude Item (\rom{1}) -- the $C^{1,1}_{loc}$ convergence. 

It remains to prove Item (\rom{2}) -- the strong $\mc A^h$-stationarity. Fix  $p\in\Sigma\cap U$ with $\ell:=\Theta^2(\Sigma,p)\geq 2$. Let $W\subset U$ be a neighborhood of $p$ so that $\Sigma$ has a decomposition $\Sigma\cap W=\cup_{i=1}^\ell \Gamma^i$ into $\ell$  ordered sheets
\[ \Gamma^1\leq \cdots\leq \Gamma^\ell. \]
Denote by $W^1$ and $W^\ell$ the bottom and top components of $W\setminus\Sigma$ (as in Definition \ref{def:strong one-sided stationarity}). 
Without loss of generality, we assume that $W^1\subset \Omega$. Take an arbitrary $X\in \mk X(W)$ pointing into $W^1$ along $\Gamma^1$.

Recall that $\Sigma_j$ converges to $\Sigma$ in the sense of $C^{1,\alpha}$. Thus for all sufficiently large $j$, $\Sigma_j$ has a decomposition $\Sigma_j=\cup_{i=1}^\ell \Gamma_j^i$ into $\ell$ ordered sheets
\[\Gamma^1_j\leq \cdots\leq \Gamma^\ell_j. \] 
Denote by $W^1_j$ and $W^\ell_j$ the bottom and top components of $W\setminus\Sigma_j$. Since $W_j^1\to W^1$ and $\Omega_j\to \Omega$ in $\C(W)$, we must have $W_j^1\subset \Omega_j$.

Fix an open subset $W'\subset\subset W$ with $\spt(X)\subset W'$. Since $\Gamma^1_j$ converges to $\Gamma^1$ in $C^{1, \alpha}_{loc}(U)$, we may write $\Gamma^1_j$ as a $C^{1,1}$-graph $u_j$ over $\Gamma^1\cap W'$ for all $j$ large, and $\|u_j\|_{C^{1,\alpha}}\to 0$ as $j\to \infty$. For each such $j$, we can find a $C^{1,1}$-homeomorphism $\phi_j: W'\to W'$, such that $\phi_j(\Gamma^1\cap W') = \Gamma_j^1\cap W'$, $\phi_j(W^1\cap W') = W_j^1\cap W'$, and $\phi_j$ converges to the identity map in $C^{1, \alpha}$. Note that $(\phi_j)_*X$ must point into $W_j^1$ along $\Gamma_j^1\cap W'$. By the strong $\mc A^h$-stationarity of $\Sigma_j$, we have
\[  \int_{\Gamma_j^1} \dv_{\Gamma_j^1} \big((\phi_j)_*X\big) + h_j \lb{(\phi_j)_*X, \nu_j} \geq 0, \]
where $\nu_j$ is the unit normal of $\Gamma_j^1$ pointing into $W_j^1$. Then by the $C^{1,\alpha}$-convergence $\Gamma^1_j \to \Gamma^1$ in $W'$, we conclude by taking $j\to\infty$ that,
\[ \int_{\Gamma^1}\dv_{\Gamma^1}(X) + h \lb{X, \nu} \geq 0, \]
where $\nu$ be the unit normal of $\Gamma^1$ pointing into $W^1$. The desired inequality for $\Gamma^\ell$ can be proved by the same argument. Hence Proposition \ref{prop:compactness} is proved.
\end{proof}

\subsection{Isotopy minimizing problem}
\label{ss:isotopy minimizing problem}
In this part, we recall the regularity result for $\A^h$-isotopic minimizing problem covered in \cite{SS23}*{Theorem 1.1}, which generalized \cites{AS79, Meeks-Simon-Yau82} to the PMC setting. 
Let $\bm r_0=\bm r_0(M,g,\sup |h|)>0$ (see \cite{SS23}*{\S 14}) be a sufficiently small constant and $U\subset B_{
\bm r_0}(p)\subset M$ be an open set. Let $\mc R\in \C(U)$ be such that  $\Sigma:=\partial \mc R\cap U$ is a smoothly embedded surface.

Let $\{\phi_k\}\subset \mk{Is}(U)$ be a sequence of isotopies, such that 
\[ \lim_{k\to\infty} \A^h\big(\phi_k(\Sigma, \mc R)\big) = \inf\{\A^h\big(\phi(\Sigma, \mc R)\big): \phi \in \mk{Is}(U)\}. \]
Then up to a subsequence, we can assume that there is a pair $(V, \Omega)\in \VC(U)$ such that 
\[ (V, \Omega) = \lim_{k\to\infty}\big(\phi_k(\Sigma), \phi_k(\mc R)\big) \quad \text{under the $\ms F$-metric}. \]
In the following of this section, we use $(\Sigma_k,\Omega_k)$ to denote $(\phi_k(\Sigma),\phi_k(\mc R))$.
\begin{theorem}\label{thm:interior regularity}
As above, $(V, \Omega)$ is a strongly $\A^h$-stationary and stable $C^{1,1}$ $h$-boundary in $U$.
\end{theorem}
\begin{proof}
All the above conclusions besides strong $\A^h$-stationarity were already proved by \cite{SS23}*{Theorem 1.1}. We will prove the strong $\A^h$-stationarity by assuming that $\Sigma\cap U$ is a union of disjoint disks. Then the general case follows from the $\gamma$-reduction process in \cite{SS23}*{\S 13}, which is a generalization of \cite{Meeks-Simon-Yau82}*{\S 3}. 

 Fix $p\in \spt\|V\|$ with $\Theta^2(\|V\|, p) = \ell \geq 2$. Then there exists $r=r(p)>0$, such that $V$ decomposes into ordered sheets in $B_{2r}(p)$
\[ \Gamma^1\leq  \cdots\leq \Gamma^\ell,\] 
where each $\Gamma^i$ is a $C^{1,1}$-graph over a small disk in the tangent plane of $V$ at $p$. Without loss of generality, we assume that $\Omega$ does not intersect the region above $\Gamma^\ell$ in $B_r(p)$.

By the Replacement Lemma \cite{SS23}*{Lemma 8.2}, one can assume that  $\Sigma_k\res B_{2r}(p)$ consists of finitely many pairwise disjoint disks $\Gamma_k^1,\cdots,\Gamma_k^{m_k}$ with $\partial\Gamma_k^i\cap B_{2r}(p)=\emptyset$, and 
\[
\lim_{k\to\infty}[\Gamma_k^i]\res B_{r}(p)=[\Gamma^i]\res B_r(p), \quad \text{for } i=1,\cdots,\ell.
\]
Now let $X\in \mk X(B_r(p))$ be such that $\langle X,\nu\rangle\geq 0$, where $\nu$ is the upward normal to $\Gamma_\ell$. Denote by $\{\phi^t\}_{t\in[0,1]}\subset \mr{Diff}_0\big(B_r(p)\big)$ the flow generated by $X$. Fix $t\in[0,1]$. Then by replacing $\Gamma_k^\ell$ with $\phi^t(\Gamma_k^\ell)$, we obtain an immersed surface $\hat\Sigma_k$ (which may have self-intersections since we only moved one sheet). Applying resolution of overlaps lemma \cite{SS23}*{Lemma 7.3} to $\hat\Sigma_k$, there exists a smoothly embedded surface $\wti\Sigma_k$ with $\wti\Sigma_k=\partial\wti\Omega_k$ in $U$ for some $\wti\Omega_k\in \C(U)$, such that $(\wti\Sigma_k, \wti\Omega_k)$ can be obtained from $(\Sigma_k, \Omega_k)$ through some $\hat\phi_k\in \mk{Is}\big(B_r(p)\big)$, and
\[ \mf F([\Sigma_k],[\wti \Sigma_k])+ \mf M\big(\llbracket \wti \Omega_k\rrbracket-\llbracket\Omega_k\rrbracket-T(\phi^t(\Gamma_k^\ell),\Gamma_k^\ell)\big) <\frac{1}{k};\]
here $\llbracket \wti \Omega_k\rrbracket$, $\llbracket\Omega_k\rrbracket$ and $T(A,B)$ are 3-currents (mod 2) so that $\partial T(A,B)=\llbracket A\rrbracket-\llbracket B\rrbracket$ for any two disks $A,B$ with $\partial A=\partial B$. Note that  $T\big(\phi^t(\Gamma_k^\ell),\Gamma_k^\ell\big)$ converges to $T\big(\phi^t(\Gamma^\ell),\Gamma^\ell\big)$, and the interior of $T\big(\phi^t(\Gamma^\ell),\Gamma^\ell\big)$ does not intersect $\Omega$ (the limit of $\Omega_k$), since $X$ points upward. Then we conclude that, as $k\to\infty$, 
\begin{equation}\label{eq:convergence of each component} 
\mc H^3\big(\Omega_k \cap T(\phi^t(\Gamma_k^\ell),\Gamma_k^\ell)\big) \to 0, \text{ and hence } \wti\Omega_k \to \Omega\cup T\big(\phi^t(\Gamma^\ell),\Gamma^\ell\big) \text{ in $\C(U)$}.
\end{equation}
Observe that $\wti \Sigma_k$ is a slight perturbation of $\hat\Sigma_k = (\Sigma_k\setminus \Gamma_k^\ell)\sqcup\phi^t(\Gamma_k^\ell)$; this yields that
\[
\Big|\mc A^h(\wti \Sigma_k,\wti \Omega_k)- \big(\mc H^2(\Sigma_k\setminus\Gamma_k^\ell) + \mc H^2(\phi^t(\Gamma_k^\ell))-\int_{\wti\Omega_k}h\,\mr d\mc H^3\big) \Big| \to 0, \text{ as } k\to\infty.
\]
This together with \eqref{eq:convergence of each component} gives 
\begin{equation}\label{eq:limit of wti pair}
\lim_{k\to\infty}\mc A^h(\wti\Sigma_k,\wti\Omega_k)= \|V\|(U) - \mc H^2(\Gamma^\ell)+\mc H^2(\phi^t(\Gamma^\ell))-\int_{\Omega}h\,\mr d\mc H^3 - \int_{T\big(\phi^t(\Gamma^\ell),\Gamma^\ell\big)}h\,\mr d\mc H^3.
\end{equation}
Since $(V,\Omega)$ is $\mc A^h$-minimizing, then 
\[
\mc A^h(\wti\Sigma_k,\wti\Omega_k)\geq \mc A^h(V,\Omega) = \|V\|(U) - \int_\Omega h\,\mr d\mc H^3.
\]
Combining with \ref{eq:limit of wti pair}, we conclude that 
\[
\mc H^2(\phi^t(\Gamma^\ell))\geq \mc H^2(\Gamma^\ell) + \int_{T\big(\phi^t(\Gamma^\ell),\Gamma^\ell\big)}h\,\mr d\mc H^3.
\]
By taking the derivative w.r.t. to $t$, this gives the desired inequality of the strong $\mc A^h$-stationarity for $\Gamma^\ell$. The same argument will also give the desired inequality for $\Gamma^1$. This completes the proof of Theorem \ref{thm:interior regularity}.
\end{proof}


\section{Min-max and tightening}\label{sec:min-max and tightening}

In this section, we will set up the relative min-max problem for the $\A^h$-functional in the space of separating surfaces in Section \ref{ss:min-max problem}. We will also establish the pull-tight process and prove the tightening theorem in Section \ref{SS:tightening}.

\subsection{Min-max problem}\label{ss:min-max problem} 
Fix a connected closed surface $\Sigma_0$ of genus $\mk g_0$. A smooth embedding $\phi: \Sigma_0\to M$ is said to be \textit{separating} if $M\setminus \phi(\Sigma_0)$ is the disjoint union of two nonempty domains $\Omega^1, \Omega^2$ enjoying a common smooth boundary $\phi(\Sigma_0)$. We will write the image as $\Sigma = \phi(\Sigma_0)$. When we write a pair $(\Sigma, \Omega)$, where $\Omega$ is an arbitrary choice of $\{\Omega^1, \Omega^2\}$, we assume that $\Sigma$ carries the orientation induced by the outer normal $\nu$ of $\Omega$, and say that \textit{$\Omega$ is bounded by $\Sigma$}, or \textit{$\Sigma$ bounds $\Omega$}. 

We denote
\begin{equation}\label{eq:ms E}
\ms E = \left\{(\Sigma, \Omega): \text{ $\Sigma$ is a separating embedding of $\Sigma_0$ which bounds $\Omega$}\right\}, \end{equation}
endowed with oriented smooth topology in the usual sense, that is, $(\Sigma_j, \Omega_j)$ converges to $(\Sigma_\infty, \Omega_\infty)$ if $\Sigma_j$ converges in the smooth topology to $\Sigma_\infty$ and $\Omega_j$ converges to $\Omega_\infty$ in $\C(M)$.

Let $X$ be a finite dimensional cubical complex, and $Z\subset X$ be a subcomplex. Let $\Phi_0: X\to \ms E$ be a continuous map. We let $\Pi$ be the set of all continuous maps $\Phi: X\to \ms E$ which is homotopic to $\Phi_0$ relative to $\Phi_0|_Z: Z\to \ms E$. We call such a $\Phi$ an \textit{$(X, Z)$-sweepout}, or simply a \textit{sweepout}.

\begin{definition}
    Given $(X, Z)$ and $\Phi_0$ as above, $\Pi$ is called the \textit{$(X, Z)$-homotopy class} of $\Phi_0$. 
\end{definition}

We can now set up the relative min-max problem for the $\A^h$-functional as usual.
\begin{definition}\label{def:min-max sequence}
    The \textit{$h$-width} of $\Pi$ is defined by:
    \[ \mf L^h = \mf L^h(\Pi) = \inf_{\Phi\in \Pi} \sup_{x\in X} \A^h\big(\Phi(x)\big). \]
\end{definition}

\begin{definition}\label{def:minimizing seq, min-max seq, and critical set}
    A sequence $\{\Phi_i\}_{i\in\mb N}\subset \Pi$ is called a \textit{minimizing sequence} if
    \[ \mf L^h(\Phi_i): = \sup_{x\in X}\A^h\big(\Phi_i(x)\big) \to \mf L^h, \text{ when } i\to\infty. \]
    A subsequence $\{\Phi_{i_j}(x_j): x_j\in X\}_{j\in \mb N}$ is called a \textit{min-max (sub)sequence} if 
    \[ \A^h\big( \Phi_{i_j}(x_j) \big) \to \mf L^h, \text{ when } j \to\infty. \]
    The \textit{critical set} of a minimizing sequence $\{\Phi_i\}$ is defined by
    \[ \mf C(\{\Phi_i\})=\left\{(V, \Omega)\in\VC(M)\left|\,
    \begin{aligned}   
        & \exists \text{ a min-max subsequence }\{\Phi_{i_j}(x_j)\} \text{ such}\\
        & \text{that } \ms F\big(\Phi_{i_j}(x_j), (V, \Omega)\big) \to  0 \text{ as } j\to\infty
    \end{aligned}\right\}\right..
    \]
\end{definition}

We have the following min-max theorem, and the proof will be given in Section \ref{SS:proof of pmc min-max}.

\begin{theorem}[PMC Min-Max Theorem]\label{thm:pmc min-max theorem}
    With all notions as above, suppose
    \begin{equation}\label{eq:width nontrivial1}
        \mf L^h(\Pi)> \max\left\{\max_{x\in Z}\A^h\big(\Phi_0(x)\big), 0\right\}.
    \end{equation}
    Then there exist a minimizing sequence $\{\Phi_i\}\subset \Pi$, and 
    a strongly $\A^h$-stationary, $C^{1,1}$ $h$-boundary $(\Sigma, \Omega)$ lying in the critical set $\mf C(\{\Phi_i\})$, such that
    \[ \A^h(\Sigma, \Omega) = \mf L^h(\Pi). \]
\end{theorem}
\begin{remark}
The strong $\A^h$-stationarity is an essential part of the regularity result. This will play a crucial role in our new multiplicity one theorem.
\end{remark}

\subsection{Tightening}\label{SS:tightening}
Take $L= \mf L^h + \sup_M |h(p)| \cdot \Vol(M) + 1$. Recall that $A^L$ and $A^L_0$ are defined in \eqref{eq:AL} and\eqref{eq:AL0} respectively. Fix a compact subset $B\subset \ms E\cap A^L$, which we usually take to be $B=\Phi_0(Z)$. Following the procedure in \cite{ZZ17}*{Section 4} (see also \cite{Zhou19}*{\S 1.2}), we will describe the tightening process in four steps.

\medskip
{\noindent\em Step 1: Annular decomposition}. 

\medskip
Consider the concentric annuli around $A^L_0\cup B$ under the $\ms F$-metric,
\begin{equation}\label{eq:A_j}
\begin{split}
    & A_0 = A^L_0\cup B, \\
    & A_1 = \{(V, \Omega)\in A^L: \ms F\big((V, \Omega), A_0 \big)\geq \frac{1}{2} \}, \\
    & A_j = \{(V, \Omega)\in A^L: \frac{1}{2^j} \leq \ms F\big((V, \Omega), A_0 \big)\leq \frac{1}{2^{j-1}}\},\quad j\in\mb N,\, j\geq 2.
\end{split}
\end{equation}
By a straightforward contradiction argument using the compactness of $A_j$, we can find some $c_j>0$ depending only on $j$, such that for any $(V, \Omega)\in A_j$, there exists $\mc X_{V, \Omega}\in \mk X(M)$, such that
\begin{equation*}
    \|\mc X_{V, \Omega}\|_{C^1(M)}\leq 1, \quad \delta A^h_{V, \Omega}(\mc X_{V, \Omega})\leq -c_j<0.
\end{equation*}

\medskip
{\noindent\em Step 2: A map from $A^L$ to $\mk X(M)$}. 

\medskip
We will construct a map $\mc X: A^L \to \mk X(M)$ which is continuous under the $C^1$ topology on $\mk X(M)$. In this part, we will use $\ms B_r(V, \Omega)$ to denote the open ball in $(\VC(M), \ms F)$ centered at $(V, \Omega)$ with radius $r>0$.

As mentioned in Section \ref{ss:Ah functional and VC space}, for a fixed $\mc X\in \mk X(M)$, the map $(V, \Omega)\mapsto \delta \A^h_{V, \Omega}(\mc X)$ is continuous under the $\ms F$-metric. Therefore, for any $(V, \Omega)\in A_j$, there exists $0<r_{V, \Omega}<\frac{1}{2^{j+1}}$, such that for any $(V', \Omega')\in \ms B_{r_{V, \Omega}}(V, \Omega)$, we have
\begin{equation}\label{eq:first variational upper bound in small balls}
    \delta \A^h_{V', \Omega'}(\mc X_{V, \Omega}) \leq \frac{1}{2} \delta \A^h_{V, \Omega}(\mc X_{V, \Omega}) \leq -\frac{1}{2}c_j<0.
\end{equation}
Now $\{\ms B_{r_{V, \Omega}/2}(V, \Omega): (V, \Omega)\in A_j\}$ forms an open covering of $A_j$. By the compactness of $A_j$, we can find a finite subset $\{\ms B_{r_{j, i}}(V_{j, i}, \Omega_{j, i}): (V_{j, i}, \Omega_{j, i})\in A_j, 1\leq i \leq q_j\}$ where $r_{j, i}$ is the radius associated with $(V_{j, i}, \Omega_{j, i})$, such that,
\begin{enumerate}[label=\roman*)]
    \item the balls $\wti{\ms B}_{j, i}$ (with half radii of $\ms B_{j, i}$) covers $A_j$;

    \item the balls $\ms B_{j, i}$ are disjoint from $A_k$ for $|k-j|\geq 2$, (this can be easily achieved by possibly shrinking $r_{V, \Omega}$).
\end{enumerate}
Here and in the following we use $\ms B_{j, i}$, $\wti{\ms B}_{j, i}$, and $\mc X_{j, i}$ to denote $\ms B_{r_{j, i}}(V_{j, i}, \Omega_{j, i})$, $\ms B_{r_{j, i}/2}(V_{j, i}, \Omega_{j, i})$, and $\mc X_{V_{j, i}, \Omega_{j, i}}$ respectively.

Now we construct a partition of unity $\{\varphi_{j, i}: j\in \mb N, 1\leq i\leq q_j\}$ sub-coordinate to the covering $\{\wti{\ms B}_{j, i}\}$ by 
\[ \varphi_{j, i}(V, \Omega) = \frac{\psi_{j, i}(V, \Omega)}{\sum \{\psi_{p, q}(V, \Omega): p\in \mb N, 1\leq q\leq q_p \}},\]
where $\psi_{j, i}(V, \Omega) = \ms F\big((V, \Omega), A^L\setminus \wti{\ms B}_{j, i}\big)$.

We define the desired map $\mc X: A^L \to \mk X(M)$ by 
\begin{equation}\label{eq: X(V, G)}
    \mc X(V, \Omega) = \ms F\big((V, \Omega), A_0 \big) \sum_{j\in \mb N, 1\leq i\leq q_j} \varphi_{j, i}(V, \Omega) \mc X_{j, i}.
\end{equation}
Since in a sufficiently small neighborhood of any $(V, \Omega)\in A^L$, the above sum has only finitely many summands, we know
\begin{lemma}\label{lem:continuous maps to vector fields}
    The map constructed above is continuous under the $C^1$ topology on $\mk X(M)$. Moreover, the restriction $\mc X: A^L\setminus (A^L_0\cup B) \to \mk X(M)$ is continuous under the smooth topology on $\mk X(M)$.
\end{lemma}

\medskip
{\noindent\em Step 3: A map from $A^L$ to the space of isotopies}. 

\medskip
We will associate each $(V, \Omega)\in\VC(M)$ with an isotopy of $M$ in a continuous manner in the same way as \cite{ZZ17}*{Section 4.3}. The isotopies will be flows $\{\Phi_{V, \Omega}(t)\}_{t\geq 0}\in \Diff_0(M)$ associated with $\mc X_{V, \Omega}$ for each $(V, \Omega)$. For our purpose, we will need to specify how far (some $T_{V, \Omega}>0$) we can flow along each $\mc X_{V, \Omega}$, that is, $t\in [0, T_{V, \Omega}]$. Nevertheless, by Lemma \ref{lem:continuous maps to vector fields}, the family of vector fields obtained in this way are only continuous in the $C^1$-topology. We will carefully smooth out these families to make them continuous in the smooth topology at the end.

Given $(V, \Omega)\in \VC(M)$, write $(V_t, \Omega_t) = \Phi_{V, \Omega}(t)_\#(V, \Omega) \in \VC(M)$. We will show that the $\A^h$-values $\{\A^h(V_t, \Omega_t)\}$ can be deformed down by a fixed amount depending only on $\ms F\big((V, \Omega), A_0\big)$. To show this, given any $(V, \Omega)\in A_j$, there are only finitely many balls $\wti{\ms B}_{k, i}$ that contains $(V, \Omega)$ by our construction, so we let $\rho_{V, \Omega}$ be the smallest radii of those balls $\wti{\ms B}_{k, i}$ with $(V, \Omega)\in \wti{\ms B}_{k, i}$. Since for each $j$, only balls in the collection $\{\wti{\ms B}_{k, i}: j-1\leq k\leq j+1, 1\leq i\leq q_k\}$ may intersect $A_j$ nontrivially, we know that $\rho_{V, \Omega}\geq r_j >0$ for some $r_j$ depending only on $j$. By \eqref{eq:A_j}, \eqref{eq:first variational upper bound in small balls} and the definition of $\mc X(V, \Omega)$ in \eqref{eq: X(V, G)}, we have for any $(V', \Omega')\in \ms B_{\rho_{_{V, \Omega}}}(V, \Omega)$ that
\[ 
\begin{aligned}
    \delta \A^h_{V', \Omega'}\big(\mc X(V, \Omega)\big) & \leq \ms F\big((V, \Omega), A_0 \big) \cdot (-\frac{1}{2})\cdot \min\{c_{j-1}, c_j, c_{j+1}\} \\
    & \leq -\frac{1}{2^{j+1}} \min\{c_{j-1}, c_j, c_{j+1}\}. 
\end{aligned} 
\]
Therefore, we can find two continuous functions $g, \rho: (0, \infty) \to (0, \infty)$ such that $\lim_{t\to 0}g(t)=0$, $\lim_{t\to 0}\rho(t)=0$, and for any $(V', \Omega')\in A^L$,
\begin{equation}\label{eq:first variation upper bound2}
    \delta\A^h_{V', \Omega'}\big(\mc X(V, \Omega)\big) \leq -g\big(\ms F((V, \Omega), A_0)\big), \quad \text{if } \ms F\big((V', \Omega'), (V,  \Omega)\big) \leq \rho\big(\ms F((V, \Omega), A_0)\big).
\end{equation}

Next, we will construct a continuous time function $T: (0, \infty)\to (0, \infty)$, such that $T(t)\to  0$ as $t\to 0$, and for any $(V, \Omega)\in A^L$, denoting $\gamma = \ms F\big((V, \Omega), A_0\big)$,
\begin{itemize}
    \item $(V_t, \Omega_t)$ (obtained by deformations using isotopies $\Phi_{V,\Omega}(t)$) belongs to $\ms B_{\rho(\gamma)}(V, \Omega)$ for all $t\in [0, T(\gamma)]$. 
\end{itemize}
To check this, for any $(V, \Omega)\in A_j$, denoting $\rho = \rho\big(\ms F((V, \Omega), A_0)\big)$, there exists $T_{V, \Omega}>0$, such that $(V_t, \Omega_t)$ belongs to $\ms B_{\rho}(V, \Omega)$ for all $t\in [0, T_{V, \Omega}]$. By the compactness of $A_j$ and the continuity $\big(t, (V, \Omega)\big) \mapsto (V_t, \Omega_t)$, we can choose $T_{V, \Omega}$ such that $T_{V, \Omega}\geq T_j>0$ for all $(V, \Omega)\in A_j$ and for some $T_j$ depending only on $j$. The desired function $T(\gamma)$ can be obtained by interpolation between $T_j$'s. 

\medskip
In summary, for any $(V, \Omega)\in A^L\setminus A_0$, denoting $\gamma = \ms F\big((V, \Omega), A_0\big)>0$, we can define
\begin{equation}\label{eq:Psi_V,Omega}
    \Psi_{V, \Omega}(t, \cdot) = \Phi_{V, \Omega}\big( T(\gamma)t, \cdot \big),\quad \text{ for } t\in[0, 1],
\end{equation}
and $\mc L: (0, \infty)\to (0, \infty)$, with $\mc L(\gamma)= T(\gamma)g(\gamma)$; then $\lim_{\gamma\to 0}T(\gamma) = 0$. We can deform any $(V, \Omega)\in A^L\setminus A_0$ through a continuous family $\{(V_t, \Omega_t) = \Psi_{V, \Omega}(t)_\#(V, \Omega): t\in [0, 1]\}\subset \ms B_{\rho(\gamma)}(V, \Omega)$, such that, by \eqref{eq:first variation upper bound2},
\begin{equation}\label{eq:Ah deformation down}
    \begin{split}
    \A^h(V_1, \Omega_1) - \A^h(V, \Omega) & \leq \int_0^{T(\gamma)} \delta \A^h_{V_t, \Omega_t}\big(\mc X(V, \Omega)\big) \,\mr dt \leq -T(\gamma)\cdot g(\gamma)\\
    & = -\mc L(\gamma) <0.
    \end{split}
\end{equation}

\medskip
{\noindent\em Step 4: Smoothing out families of vector fields}. 

\medskip
We will use the construction above to prove the following \textit{pull-tight} result.

\begin{theorem}[Pull-tight]\label{thm:pull tight}
    Let $\Pi$ be an $(X, Z)$-homotopy class generated by some continuous $\Phi_0: X \to \ms E$ relative to $\Phi_0|_Z$. Given a minimizing sequence $\{\Phi^*_i\}_{i\in\mb N}\subset \Pi$ associated with $\A^h$, there exists another minimizing sequence $\{\Phi_i\}_{i\in \mb N}\subset \Pi$, such that $\mf C(\{\Phi_i\})\subset \mf C(\{\Phi_i^*\})$ and every element $(V, \Omega)\in \mf C(\{\Phi_i\})$ is either $\A^h$-stationary, or belongs to $B = \Phi_0(Z)\subset \ms E$. 
\end{theorem}
\begin{proof}
    We can assume without loss of generality that $\mf L^h(\Phi^*_i)\leq \mf L^h+1$ for all $i\in\mb N$, so clearly we have $\Phi^*_i(x)\in A^L$ for all $i\in \mb N$ and $x\in X$. 
    
    For each $\Phi^*_i: X\to \ms E$, we can associate it with a family of vector fields:
    \[ \mc X_i: X \to \mk X(M), \text{ such that } \mc X_i(x) = \mc X(\Phi^*_i(x)); \]
    then this map is continuous under the $C^1$-topology on $\mk X(M)$ by Lemma \ref{lem:continuous maps to vector fields}. Moreover, by our construction, 
    \[ \mc X_i(x) = 0, \quad \text{ for any } x\in Z. \]
    Define $\Psi_i: X \to \mk{Is}(M)$ such that $\Psi_i(x) = \Psi_{\Phi^*_i(x)}$ via \eqref{eq:Psi_V,Omega} if $\Phi^*_i(x) \notin A^L_0\cup B$, and $\Psi_i(x) = \Id$ if $\Phi^*_i(x) \in A^L_0\cup B$. Note that $x \mapsto \Psi_i(x)$ is only continuous under the $C^1$-topology on $\mk{Is}(M)$. Write $\wti\Phi_i(x) = \Psi_i(1, \Phi^*_i(x))$. Using \eqref{eq:Ah deformation down}, we have that
    \[ \A^h\big(\wti\Phi_i(x)\big) - \A^h\big(\Phi^*_i(x)\big) \leq -\mc L \big(\ms F(\Phi^*_i(x), A^L_0\cup B) \big). \]
    
    For each $i\in \mb N$, we can smooth out $\mc X_i$ to some $\wti{\mc X}_i : X \to \mk X(M)$ which is continuous under the smooth topology, and such that $\wti{\mc X}_i(x) = 0$ for any $x\in Z$, and $\|\mc X_i - \wti{\mc X}_i\|_{C^1}\leq \frac{1}{i}$. Note that by \eqref{eq:1st variation of Ah},
    \[ |\delta\A^h_{V, \Omega}(\mc X) - \delta \A^h_{V, \Omega}(\wti{\mc X})| \leq C\big(\|V\|(M)+\|\partial\Omega\|(M)\big)\cdot \|\mc X - \wti{\mc X}\|_{C^1} \leq C'\|\mc X - \wti{\mc X}\|_{C^1}, \]
    for some universal constant $C, C'>0$ independent of the choice $(V, \Omega)\in A^L$. 
    Now define $\wti{\Psi}_i : X\to \mk{Is}(M)$ using $\wti{\mc X}_i$ instead of $\mc X_i$ in \eqref{eq:Psi_V,Omega}, then $x\mapsto \wti\Psi_i(x)$ is continuous under the smooth topology on $\mk{Is}(M)$, and $\wti{\Psi}_i |_Z \equiv \Id$. Writing $\Phi_i(x)=\wti{\Psi}_i(1, \Phi^*_i(x))$, then $\Phi_i$ is homotopic to $\Phi^*_i$ in $\ms E$ relative to $\Phi_0|_Z$, so $\Phi_i$ belongs to $\Pi$. Now by \eqref{eq:Ah deformation down} we have that
    \begin{equation}\label{eq:Ah comparison for tightened sweepout}
        \A^h\big(\Phi_i(x)\big) - \A^h\big(\Phi^*_i(x)\big) \leq - \mc L \big(\ms F(\Phi^*_i(x), A^L_0\cup B) \big) + \frac{C''}{i},
    \end{equation}
    for some universal $C''>0$.

    Suppose that $\{\Phi_{i_j}(x_j)\}$ is a min-max subsequence, then $\A^h\big(\Phi_{i_j}(x_j)\big)\to \mf L^h$. By \eqref{eq:Ah comparison for tightened sweepout} and the fact that $\{\Phi^*_i\}$ is a minimizing sequence, we know that $\{\Phi^*_{i_j}(x_j)\}$ is also a min-max subsequence. Then the left hand side of \eqref{eq:Ah comparison for tightened sweepout}, when applied to the two min-max subsequences, will converge to $0$, and hence
    \[ \ms F(\Phi^*_{i_j}(x_j), A^L_0\cup B) \to 0, \text{ when } j\to \infty.  \]
    By the definition of $\mc X_i$, this implies that $\mc X_{i_j}(x_j)$ converges to $0$ in the $C^1$-topology, and so is $\wti{\mc X}_{i_j}(x_j)$. Hence we have $\ms F\big(\Phi^*_{i_j}(x_j), \Phi_{i_j}(x_j)\big) \to 0$ as $j\to \infty$, and this proves that $\mf C(\{\Phi_i\})\subset \mf C(\{\Phi^*_i\})$. Moreover, we also have
    \[ \ms F(\Phi_{i_j}(x_j), A^L_0\cup B) \to 0, \text{ when } j\to \infty. \]
    This implies that elements in $\mf C(\{\Phi_i\})$ is either $\A^h$-stationary (lying in $A^L_0$) or belongs to $B = \Phi_0(Z)$.
\end{proof}

\section{Almost minimizing}\label{sec:almost minimizing}

In this part, we adapt the almost minimizing property to the $\A^h$-functional using embedded separating surfaces; see Section \ref{ss:definitions for almost minimizing}. As the main difference compared with \cite{Colding-DeLellis03} where they need successive replacements in annuli, we need the existence of  chains of replacements in open subsets; see Definition \ref{def:replacement chain}. We then prove the existence $\A^h$-almost minimizing pairs using a combinatorial arguments originally due to Almgren-Pitts in Section \ref{ss:existence of almost minimizing pairs}.

\subsection{Definitions}\label{ss:definitions for almost minimizing}

\begin{definition}[c.f. \cite{Colding-DeLellis03}*{Definition 3.2}]\label{def:ep pmc almost minimizing}
Given $\epsilon, \delta >0$, an open set $U\subset M$, and an embedded separating surface $(\Sigma, \Omega)\in \ms E$, we say that {\em $(\Sigma, \Omega)$ is $(\A^h, \epsilon, \delta)$-almost minimizing in $U$} if there does not exist any isotopy $\psi\in \mk{Is}(U)$, such that
\[ \A^h(\psi(t, \Sigma, \Omega)) \leq \A^h(\Sigma, \Omega) + \delta \quad \text{ for all } t\in[0,1];\]
\[  \A^h(\psi(1, \Sigma, \Omega)) \leq \A^h(\Sigma, \Omega) - \epsilon. \]
\end{definition}

\begin{definition}[$\A^h$-almost minimizing pairs]\label{def:pmc almost minimizing}
Given an open subset $U\subset M$, a pair $(V, \Omega)\in \VC(M)$, and a sequence of embedded separating surfaces $\{(\Sigma_j, \Omega_j)\}_{j\in \mb N}\subset \ms E$. 
We say that $(V, \Omega)$ is {\em $\A^h$-almost minimizing w.r.t. $\{(\Sigma_j, \Omega_j)\}$ in $U$}, if there exist $\epsilon_j\to 0$ and $\delta_j\to 0$, such that
\begin{itemize}
\item $(\Sigma_j, \Omega_j) \to (V, \Omega)$ in the $\ms F$-metric as $j\to \infty$;

\item $(\Sigma_j, \Omega_j)$ is $(\A^h, \epsilon_j, \delta_j)$-almost minimizing in $U$. 
\end{itemize}
Sometime we also say {\em $(V, \Omega)$ is $\A^h$-almost minimizing in $U$} without referring to the sequence $\{(\Sigma_j, \Omega_j)\}$.
\end{definition}

We now show that $\A^h$-almost minimizing implies $\A^h$-stationary and $\A^h$-stable.
\begin{lemma}\label{lem:am implies stationary and stable}
    Let $(V, \Omega)\in \VC(M)$ be $\A^h$-almost minimizing in $U$, then
    \begin{enumerate}[label=\roman*)]
        \item $(V, \Omega)$ is $\A^h$-stationary in $U$;
        \item $(V, \Omega)$ is $\A^h$-stable in $U$.
    \end{enumerate}
\end{lemma}

\begin{proof}
Suppose on the contrary that $(V,\Omega)$ is not $\mc A^h$-stationary in $U$. Then there exists $X\in \mk X(U)$ with associated flow $\phi^t$, 
such that for all $t\in [0,1]$,
\[ \frac{d}{dt}\mc A^h\big((\phi^t)_\#(V,\Omega)\big)<0. \] 
Let 
\[ \epsilon:=\mc A^h(V,\Omega)-\mc A^h\big((\phi^1)_\#(V,\Omega)\big)>0.\] 
Observe that there exists $\eta>0$ small enough such that for all $(V',\Omega')\in \mc{VC}(M)$, 
\[ \ms F\big((V,\Omega),(V',\Omega')\big)<\eta \quad \Longrightarrow \quad \frac{d}{dt}\mc A^h\big((\phi^t)_\#(V', \Omega')\big)<0, \quad \forall\, t\in [0,1] . \]
By assumption, there exists $(\Sigma_j,\Omega_j) (\in \ms E)\to (V,\Omega)$ which is $(\mc A^h,\epsilon_j,\delta_j)$-almost minimizing in $U$ with $\epsilon_j,\delta_j\to 0$. Since $(\Sigma_j,\Omega_j)$ converges to $(V,\Omega)$, then for all sufficiently large $i$, $\ms F\big((V,\Omega),(\Sigma_j, \Omega_j)\big)<\eta$; this implies that
\[   \frac{d}{dt}\mc A^h\big((\phi^t)_\#(\Sigma_j,\Omega_j)\big)<0, \quad \forall\, t\in [0,1].\]
Moreover, by taking sufficiently large $j$, we have $\epsilon_j<\epsilon/2$ and 
\[\mc A^h(\Sigma_j, \Omega_j) - \mc A^h\big((\phi^1)_\#(\Sigma_j,\Omega_j)\big) > \epsilon/2.\]
This leads to a contradiction. Hence we have proved that $(V,\Omega)$ is $\mc A^h$-stationary in $U$.

The proof of Item (\rom{2}) follows in the same way by using the second variations. We omit the details here. 
\end{proof}

Next, we introduce the notion of $\A^h$-replacements. 

\begin{definition}\label{def:pmc replacement} 
Given an open subset $U\subset M$ and $(V, \Omega)\in \VC(M)$, a pair $(V^*, \Omega^*)\in \VC(M)$ is said to be an {\em $\A^h$-replacement} of $(V, \Omega)$ in $U$ if 
\[ (V^*, \Omega^*) = (V, \Omega) \text{ outside } \closure(U), \quad \A^h(V^*, \Omega^*) = \A^h(V, \Omega), \text{ and }\]
\[ (V^*, \Omega^*) \text{ is a strongly $\A^h$-stationary and stable $C^{1, 1}$ $h$-boundary in $U$}. \]
\end{definition}

\begin{definition}\label{def:weak good replace}
As above, $(V, \Omega)$ is said to have {\em (weak) good replacement property in $U$} if for any $p\in U$, there exists $r_p>0$, such that $(V, \Omega)$ has an $\A^h$-replacement $(V^*, \Omega^*)$ in any open annulus $\mr{An }\subset\subset \mr{An}(p; 0, r_p)$. 
\end{definition}

The following stronger good replacement property plays a key role in our new proof of regularity of min-max pairs without invoking unique continuation. 

\begin{definition}[Replacement chain property]\label{def:replacement chain}
Let $(V,\Omega)\in\VC(M)$ and $U\subset M$ be an open set. $(V, \Omega)$ is said to have the \textit{replacement chain property in $U$} if the following statement holds. 
For any sequence of open subsets $B_1,\cdots,B_k\subset \subset U$, there exist a sequence $(V,\Omega)=(V_0,\Omega_0),(V_1,\Omega_1),\cdots, (V_k,\Omega_k)$ in $\VC(M)$ satisfying that 
\[ \text{$(V_j,\Omega_j)$ is an $\mc A^h$-replacement of $(V_{j-1},\Omega_{j-1})$ in $B_j$ for $j=1,\cdots,k$,}\]
and 
\[ \text{$(V_j,\Omega_j)$ is $\mc A^h$-stationary and stable in $U$.}\]
Furthermore, if there is another sequence of open subsets $B_1, \cdots, B_{k}, B_{k+1}', \cdots , B_{\ell}'\subset\subset U$, then the sequence of replacements  $(\wti V_j,\wti \Omega_j)$ can be chosen so that  
\[ (\wti V_j,\wti \Omega_j)=(V_j,\Omega_j)  \quad \forall j=1,\cdots, k. \]
\end{definition}

\begin{remark}\label{rem:consequence of replacement chain property}
By definition, we have,
\begin{enumerate}[label=\roman*)]
    \item if $(V, \Omega)\in \VC(M)$ satisfies the replacement chain property in $U$, and $B\subset\subset U$ is open, then an $\A^h$-replacement $(V^*, \Omega^*)$ of $(V, \Omega)$ in $B$ also satisfies the replacement chain property;
    
    \item if $(V, \Omega)$ has the replacement chain property  in $U$, $(V, \Omega)$ itself is $\mc A^h$-stationary and stable in $U$;

    \item the replacement chain property implies the (weak) good replacement property. 
\end{enumerate}
\end{remark}

\subsection{Existence of almost minimizing pairs}\label{ss:existence of almost minimizing pairs}

In this part, we use the Almgren-Pitts type combinatorial arguments to find an $\A^h$-min-max pair $(V, \Omega)$ which is $\A^h$-almost minimizing in small annuli. 

Let $\Pi$ be an $(X, Z)$-homotopy class of $(X, Z)$-sweepouts generated by some continuous $\Phi_0: X \to \ms E$ relative to $\Phi_0|_Z$. Suppose that the nontriviality condition \eqref{eq:width nontrivial1} holds. Let $\{\Phi_i\}_{i\in \mb N}\subset \Pi$ be a pull-tight minimizing sequence obtained by Theorem \ref{thm:pull tight}. Then every $(V, \Omega)\in \mf C(\{\Phi_i\})$ is $\A^h$-stationary.

\begin{theorem}[Existence of almost minimizing pairs]\label{thm:existence of almost minimizing pairs}
As above, suppose \eqref{eq:width nontrivial1} holds, then there exist an $\A^h$-stationary pair $(V_0, \Omega_0)\in \mf C(\{\Phi_i\})$, and a min-max subsequence $\{(\Sigma_j, \Omega_j) = \Phi_{i_j}(x_j)\}_{j\in \mb N} \subset \ms E$, such that {\em $(V_0, \Omega_0)$ is $\A^h$-almost minimizing in small annuli w.r.t. $\{(\Sigma_j, \Omega_j)\}$} in the following sense: for any $p\in M$, there exists $r_{am}(p)>0$, such that for any annulus $\mr{An}=\mr{An}(p; s, r)$ with $0<s<r<r_{am}(p)$,  $(V_0, \Omega_0)$ is $\A^h$-almost minimizing w.r.t. to $\{(\Sigma_j, \Omega_j)\}$ in $\mr{An}$.
\end{theorem}

We will adapt the proof for the area functional by Colding-Gabai-Ketover in \cite{Colding-Gabai_Ketover18}*{Appendix}. To do so, we introduce some notions. 

\begin{definition}\label{def:AM in L}
Given an $L\in \mb N$ and $p\in M$, a collection of annuli centered at $p$ 
\[ \ms C = \{\mr{An}(p;s_1,r_1),\cdots, \mr{An}(p;s_L,r_L)\} \] 
is said to be {\em $L$-admissible} if $2r_{j+1}<s_j$ for all $j=1,\cdots, L-1$. 

We say a pair $(V, \Omega)\in \VC(M)$ is \textit{$\A^h$-almost minimizing in $\ms C$ w.r.t. a sequence $\{(\Sigma_j, \Omega_j)\}\subset \ms E$}, if there exists $\epsilon_j\to 0$ and $\delta_j\to 0$, such that 
\begin{itemize}
    \item $(\Sigma_j, \Omega_j) \to (V, \Omega)$ in the $\ms F$-metric as $j\to \infty$;
    \item for each $j$, $(\Sigma_j, \Omega_j)$ is $(\A^h, \epsilon_j, \delta_j)$-almost minimizing in at least one annulus in $\ms C$.
\end{itemize}
\end{definition}

Assume that the parameter space $X$ is a cubical subcomplex of the cell complex $I(m, k_0)$ for some $m, k_0 \in \mb N$. Here $I(m, k)=I(1, k)\otimes \cdots I(1, k)$ ($m$-times), where $I(1, k)$ denotes the complex on $I=[0,1]$ whose $1$-cells and $0$-cells are, respectively,
\[ [1, 3^{-k}], [3^{-k}, 2\cdot 3^{-k}], \cdots, [1-3^{-k}, 1] \text{ and } [0], [3^{-k}], \cdots, [1-3^{-k}], [1].\]
We refer to \cite{Zhou19}*{Appendix A} for a summary of notions; (see also \cite{MN17}*{Section 2.1}).

\begin{lemma}\label{lem:AM in L annuli}
As above, there exist an integer $L=L(m)$ (depending only on the dimension of the large complex $I(m, k_0)$ where $X$ is embedded to), and a min-max subsequence $\{(\Sigma_j, \Omega_j) = \Phi_{i_j}(x_j)\}_{j\in \mb N} \subset \ms E$ converging to an $\mc A^h$-stationary pair $(V_0,\Omega_0)\in \mf C(\{\Phi_i\})$ such that $(V_0,\Omega_0)$ is $\mc A^h$-almost minimizing in every $L$-admissible collection of annuli w.r.t. $\{(\Sigma_j,\Omega_j)\}$.
\end{lemma}

The proof is essentially the same as that of \cite{Colding-Gabai_Ketover18}*{Lemma A.1}, and we provide some necessary details for completeness. 
\begin{proof}[Proof of Lemma \ref{lem:AM in L annuli}]
The lemma will follow directly from the following claim.
\begin{claim}
    There exists $L=L(m)\in \mb N$, such that for any $\epsilon>0$, there exist $\delta>0$, an $i>\frac{1}{\epsilon}$, and an $x\in X$ with 
    \begin{equation}\label{eq:Phi_i(x) Ah lower bound}
        \A^h\big(\Phi_i(x)\big) \geq \mf L^h -\epsilon,
    \end{equation}
    such that for any $L$-admissible collection $\ms C$, $\Phi_i(x)$ is $(\A^h, \epsilon, \delta)$-almost minimizing in at least one annulus in $\ms C$. 
\end{claim}
We can take 
\begin{equation}\label{eq:L(m)}
L(m)=(3^m)^{3^m}.    
\end{equation} 
If by contradiction that the claim does not hold, we can find $\epsilon_0>0$, such that for any $\delta>0$, any $i>1/\epsilon_0$, and any $x\in X$ satisfying \eqref{eq:Phi_i(x) Ah lower bound} (with $\epsilon_0$ in place of $\epsilon$), there exists an $L$-admissible collection $\ms C_{i, x}$, such that $\Phi_i(x)$ is not $(\A^h, \epsilon_0, \delta)$-almost minimizing in any annulus in $\ms C_{i, x}$.  Fix a $\delta_0\ll \epsilon_0$, and an $i_0>1/\epsilon_0$; we let
\[ S_{i_0} = \{x\in X: \Phi_i(x) \text{ satisfies \eqref{eq:Phi_i(x) Ah lower bound}}\}\subset X. \]
For simplicity we drop the sub-index $i_0$ for a moment. Following the same argument in \cite{Colding-Gabai_Ketover18}*{Lemma A.1}, we can find a finite open cover $\{\mc U_j\}$ of $S$, where each $\mc U_j\subset I(m, 0) = [0, 1]^m$ is open, such that
\begin{enumerate}
    \item each $\mc U_j$ can be associated with some annulus $\mr{An}_j$ belonging to some $\ms C_x$, $x\in S$, such that there exists a smooth isotopy $\psi_j \in \mk{Is}(\mr{An}_j)$, for any $y\in X\cap \mc U_j$,
    \begin{itemize}
        \item $\A^h\big( \psi_j(t, \Phi(y))\big) \leq \A^h\big( \Phi(y)\big) + 2\delta_0$ for all $t\in [0, 1]$, and
        \item $\A^h\big( \psi_j(1, \Phi(y))\big) \leq \A^h\big( \Phi(y) \big) - \epsilon_0/2$;
    \end{itemize}
    \item each $\mc U_j$ intersects at most $d(m)$ many other elements in $\{\mc U_j\}$;
    \item each $\mc U_j$ can be associated with a smooth function $\phi_j\in C^\infty_c(\mc U_j)$, $0\leq \phi_j \leq 1$, and for any $x\in S$, at least one $\phi_j(x) = 1$;
    \item if $\phi_j(x)$ and $\phi_{j'}(x)$ are both nonzero for some $j$ and $j'$, we have $\mr{An}_j\cap \mr{An}_{j'}=\emptyset$.
\end{enumerate}
We can homotopically deform $\Phi$ to $\Phi_1, \cdots, \wti\Phi$ in $\ms E$ successively using $\{\psi_j\}$ up to time $\phi_j(x)$ at each $x\in X$, such that
\[ \Phi_{j+1}(x) = \psi_j\big( \phi_j(x), \Phi_j(x) \big), \quad j =1, \cdots. \]
If we choose $\delta_0 < \frac{\epsilon_0}{8 d(m)}$, then we must have (resuming the subindex $i_0$)
\[ \sup_{x\in X}\A^h\big( \wti\Phi_{i_0}(x) \big) \leq \sup_{x\in X}\A^h\big( \Phi_{i_0}(x)\big) -\epsilon_0/2 + 2d(m)\cdot \delta_0 < \mf L^h, \]
for $i_0$ sufficiently large, which is a contradiction.
\end{proof}

As a direct corollary of the above result and Lemma \ref{lem:am implies stationary and stable}, we have that
\begin{corollary}\label{cor:property R}
As above, the $\A^h$-stationary pair $(V_0, \Omega_0)\in \mf C(\{\Phi_i\})$ satisfies 
\begin{equation}\label{eq:property R}
\begin{aligned}
    \text{Property {\bf(R)}}:\quad  & \text{for every $L(m)$-admissible collection $\ms C$ of annuli, $(V_0, \Omega_0)$ is $\A^h$-stable}\\
    & \text{in at least one annulus in $\ms C$.}
\end{aligned}
\end{equation}
\end{corollary}

\begin{proof}[Proof of Theorem \ref{thm:existence of almost minimizing pairs}]
The statement in Theorem \ref{thm:existence of almost minimizing pairs} follows by taking further subsequences of $\{(\Sigma_j, \Omega_j)\}$ in Lemma \ref{lem:AM in L annuli} the same way as in \cite{Colding-Gabai_Ketover18}*{Lemma A3}. We provide the details using a general version of this argument given in Appendix \ref{appen:am in a uniform subsequence}.

By the proof of Lemma \ref{lem:AM in L annuli}, there exist $\epsilon_j\to 0$ and $\delta_j\to 0$ such that 
\begin{itemize}
    \item $(\Sigma_j, \Omega_j) \to (V_0, \Omega_0)$ in the $\ms F$-metric as $j\to \infty$;
    \item for each $j$, $(\Sigma_j, \Omega_j)$ is $(\A^h, \epsilon_j, \delta_j)$-almost minimizing in at least one annulus in any $\ms C$.
\end{itemize}
Let $\mc P_j$ be the collection of annuli $\mr {An}$ where $(V_j,\Omega_j)$ is $(\mc A^h,\epsilon_j,\delta_j)$-almost minimizing.  Then for every $L$-admissible collection $\ms C$ of annuli, we have that $\ms C\cap \mc P_j\neq \emptyset$. Clearly, if $\mr {An}_1\subset \mr {An}\in \mc P_j$, then $\mr{An}_1\in \mc P_j$. Thus $\{\mc P_j\}$ satisfy the requirements in Proposition \ref{prop:general property}, and hence there exists a subsequence (still denoted by $\{\mc P_j\}$) such that for each $p\in M$, there exists $r_{am}(p)>0$ such that for each $0<s<r<r_{am}(p)$, $\mr{An}(p;s,r)\in \mc P_j$ for all sufficiently large $j$. This gives that $(V_0,\Omega_0)$ is $\mc A^h$-almost minimizing w.r.t. $\{(\Sigma_j,\Omega_j)\}$ in each $\mr{An}\subset \mr {An}(p,0,r_{am}(p))$ for all sufficiently large $j$. This completes the proof of Theorem \ref{thm:existence of almost minimizing pairs}.
\end{proof}


\section{Regularity of min-max pairs}\label{sec:regularity of min-max pairs}

In this section, we prove the main regularity results for $\A^h$-min-max pairs. We first prove that $\A^h$-stationary pairs with the replacement chain property are $C^{1,1}$ and strongly $\mc A^h$-stationary in Section \ref{ss:initial regularity}. We develop a novel way of using chains of replacements to prove the regularity without invoking unique continuation. We then construct the replacements using the aforementioned regularity results in Section \ref{ss:construction of replacements}, and prove the interior regularity in Section \ref{ss:regularity for almost minimizing pairs} and the full regularity in Section \ref{SS:proof of pmc min-max}.

\subsection{Initial regularity}\label{ss:initial regularity}
We start with the following characterization of tangent varifolds for an $\A^h$-stationary pair satisfying the weak good replacement property.

\begin{proposition}\label{prop:tangent plane}
Let $(V, \Omega)\in \VC(M)$ be $\A^h$-stationary in an open subset $U\subset M$. If $(V, \Omega)$ has (weak) good replacement property in $U$, then $V$ is integer rectifiable in $U$. In fact, for any $p\in \spt\|V\|\cap U$, every tangent varifold of $V$ at $p$ is an integer multiple of a plane in $T_pM$. 
\end{proposition}
\begin{proof}
The proof is the same as \cite{Colding-DeLellis03}*{Lemma 6.4} and \cite{SS23}*{Lemma 20.2}.
\end{proof}

The following lemma says that an $\A^h$-stationary, $C^{1}$ $h$-boundary\footnote{Note that we can define $C^1$-boundaries the same way as in Definition \ref{def:c11 boundary}.} that is strongly $\A^h$-stationary and $C^{1,1}$ outside a $C^1$ interface is strongly $\A^h$-stationary and $C^{1,1}$ in the whole region. This result will be used to glue two of strongly $\A^h$-stationary $C^{1,1}$ $h$-boundaries that match in the $C^1$-manner along an interface.

\begin{lemma}\label{lem:remove C1 line}
Given an open subset $W\subset M$, let $(\Sigma, \Omega)$ be an $\A^h$-stationary $C^1$-boundary in $W$. Suppose that $\Sigma$ decomposes into $C^1$-ordered sheets:
\[ \Gamma^1\leq \cdots \leq \Gamma^\ell. \]
Let $\mc T$ be a $C^1$-embedded surface in $W$ with $\partial\mc T\cap W=\emptyset$, which intersects transversely with $\Gamma^1,\cdots,\Gamma^\ell$. Suppose in addition that $\Sigma\res (W\setminus\mc T)$ is $C^{1,1}$ and $(\Sigma, \Omega)$ is strongly $\A^h$-stationary in $W\setminus \mc T$. Then $(\Sigma, \Omega)$ is a strongly $\mc A^h$-stationary $C^{1,1}$ $h$-boundary in $W$.
\end{lemma}
\begin{proof}
We first show that each sheet has bounded first variation in $W$, and hence has improved regularity. Let $\gamma^i := \Gamma^i \cap \mc T$, which is a $C^1$ curve by transversality for each $i$. Denote the two components of $\Gamma^i\setminus \gamma^i$ as $\Gamma^i_1$ and $\Gamma^i_2$ and the exterior unit co-normal along $\gamma^1$ by $\bm\eta^i_1$ and $\bm\eta^i_2$ respectively. Since $\Sigma^i$ is $C^1$, we know that
\[ \bm\eta^i_1 = -\bm\eta^i_2 \quad \text{along } \gamma^i. \]
Since $(\Sigma, \Omega)$ is $C^{1, 1}$ and strongly $\A^h$-stationary in $W\setminus \mc T$, the generalized mean curvature $H^i$ of $\Gamma^i$ (w.r.t. the unit outer normal $\nu$ given in Lemma \ref{lem:choice of outer normal}) exists $\mc H^2$-a.e. in $\Gamma^i\setminus \gamma^i$, and by Corollary \ref{cor:mean curvature charaterization}, is bounded $|H^i(x)| \leq |h(x)|$ for $\mc H^2$-a.e. $x\in \Gamma^i\setminus \gamma^i$.

For any $X\in \mk X(W)$, and any $1\leq i\leq \ell$, we have
\[
\begin{aligned}
    \int_{\Gamma^i} \dv X\,\mr d\mc H^2 & = \int_{\Gamma^i_1} \dv X \,\mr d\mc H^2 + \int_{\Gamma^i_2} \dv X\,\mr  d\mc H^2 \\
    & = \int_{\Gamma^i_1} H^i \cdot \lb{X, \nu}\,\mr d\mc H^2 + \int_{\gamma^i} \lb{X, \bm\eta^i_1}\,\mr d\mc H^1 + \int_{\Gamma^i_2} H^i \cdot \lb{X, \nu} \,\mr d\mc H^2 + \int_{\gamma^i} \lb{X,\bm \eta^i_2}\,\mr  d\mc H^1\\
    & = \int_{\Gamma^i} H^i \cdot \lb{X, \nu} \,\mr d\mc H^2.
\end{aligned}
\]
Note that we used the divergence theorem for Lipschitz vector fields on a $C^{1,1}$-surface with $C^1$-boundary. Hence $\Gamma^i$ has bounded first variation, and by the Allard regularity theorem \cite{Al75}, (see also \cite{Si}*{Theorem 24.2}), $\Gamma^i$ is  $C^{1,\alpha}$. Moreover, by Proposition \ref{prop:C11 regularity}, the $C^{1,1}$-regularity for multilayer $\A^h$-stationary boundaries, we know that $\Gamma^i$ is also in $C^{1,1}$. 

\medskip

We next prove strong $\A^h$-stationarity by a cutoff trick. We only need to check the Definition \ref{def:strong one-sided stationarity} near any points $p\in \Sigma\cap \mc T$. Assume without loss of generality that $\Theta^2(\Sigma, p) = \ell$, that is, all the $\ell$ sheets touch together at $p$. We will prove the desired inequality for $\Gamma^1$ and that for $\Gamma^\ell$ follows in a similar manner. Assume that $0\leq \eta_k \leq 1$ is a sequence of cutoff functions, such that $\spt(\eta_k)\subset W\setminus \mc T$ and $\eta_k(x)\to 1$ as $k\to\infty$ for any $x\in W\setminus \Gamma$. Denote by $W^1$ the bottom component of $W\setminus \Sigma$. Given any $X\in \mk X(W)$ supported in a sufficiently small neighborhood of $p$, such that $X$ points toward $W^1$ along $\Gamma^1$, then $X$ is a legitimate test vector field in Definition \ref{def:strong one-sided stationarity}. Since $\Gamma^1$ is a $C^{1,1}$-surface, we have by the divergence theorem,
\[
\int_{\Gamma^1} \dv X  = \int_{\Gamma^1} \dv X^\perp = \lim_{k\to\infty} \int_{\Gamma^1} \eta_k \dv X^\perp = \lim_{k\to\infty} \int_{\Gamma^1} \dv (\eta_k X^{\perp}) = \lim_{k\to\infty} \int_{\Gamma^1} \dv (\eta_k X).
\]
Then we have the following two cases.
\begin{itemize}
    \item If $W^1\subset \Omega$, then we have 
    \[
    \int_{\Gamma^1} \dv X  = \lim_{k\to\infty} \int_{\Gamma^1} \dv (\eta_k X) \geq \lim _{k\to\infty} \int_{\Gamma^1} \eta_k \lb{X, \nu} \cdot h = \int_{\Gamma^1} \lb{X, \nu} \cdot h.
    \]
    \item If $W^1\cap \Omega=\emptyset$, then we have 
    \[
     \int_{\Gamma^1} \dv X  = \lim_{k\to\infty} \int_{\Gamma^1} \dv (\eta_k X)\geq \lim _{k\to\infty} -\int_{\Gamma^1} \eta_k \lb{X, \nu} \cdot h = -\int_{\Gamma^1} \lb{X, \nu} \cdot h.
    \]
\end{itemize}
Here we use the fact that $(\Sigma, \Omega)$ is strongly $\A^h$-stationary in $W\setminus \mc T$ in the ``$\geq$'' above. Therefore, we have checked the requirement in Definition \ref{def:strong one-sided stationarity} for $\Gamma^1$ and finished the proof.
\end{proof}

In the next lemma, we will show that a replacement $(V^*, \Omega^*)$ will glue nicely with the original pair $(V, \Omega)$ along regular part under natural assumptions.

\begin{lemma}[Gluing Lemma]\label{lem:glue}
Suppose $(V,\Omega)\in \mc {VC}(M)$ satisfies the replacement chain property in an open set $U\subset M$. Assume further that $(V, \Omega)$ is a strongly $\A^h$-stationary, $C^{1,1}$ $h$-boundary in an open subset $W \subset U$. Let $B\subset\subset U$ be an open geodesic ball such that 
\begin{enumerate}[label=\arabic*)]
    \item\label{item:partial B mean curvature large} $H_{\partial B}> \|h\|_{L^\infty}$, and 
    \item $\partial B$ intersects with $\spt\|V\|$ transversely in $W$.
\end{enumerate}  
Let $(V^*,\Omega^*)$ be an $\A^h$-replacement of $(V,\Omega)$ in $B$ (which also satisfies the replacement chain property in $U$ by Remark \ref{rem:consequence of replacement chain property}). Then $(V^*, \Omega^*)$ is a strongly $\A^h$-stationary, $C^{1,1}$ $h$-boundary in  $W\cup B$.
\end{lemma}

\begin{proof}
By assumption, $V\res W$ and $V^*\res B$ are induced by $C^{1,1}$-almost embedded surfaces $\Sigma\subset W$ and $\Gamma\subset B$ respectively; $(\Sigma, \Omega)$ and $(\Gamma, \Omega^*)$ form strongly $\A^h$-stationary, $C^{1,1}$ $h$-boundary in $W$ and $B$ respectively. By Remark \ref{rem:consequence of replacement chain property}, both $(V, \Omega)$ and $(V^*, \Omega^*)$ are $\A^h$-stationary and stable in $U$, and Proposition \ref{prop:tangent plane} applies to both $V$ and $V^*$.

\medskip
We divide the proof into the following four steps.

\medskip
{\noindent\em Step 1: $\spt \|V^*\| \cap \partial B \subset \spt \|V\|\cap \partial B$, and $\Sigma\cap \partial B = [\closure(\Gamma)\setminus B] \cap W $}.
\medskip

The first part is a standard application of the Maximum Principle. Suppose on the contrary that $x\in \spt\|V^*\|\cap \partial B$, but $x\notin \spt\|V\|\cap \partial B$. Since $\spt\|V\|=\spt\|V^*\|$ outside $\closure(B)$, this implies that in a sufficiently small neighborhood $\wti{W}$ of $x$, $\spt\|V^*\|$ is contained in $\wti{W}\cap \closure(B)$. The $\A^h$-stationary condition implies that $V^*\lc \wti{W}$ has $\|h\|_{L^\infty}$-bounded first variation in $\wti{W}$, then the Maximum Principle (c.f. \cite{ZZ17}*{Proposition 2.13}, \cite{Whi10}*{Theorem 5}) and the mean curvature assumption \ref{item:partial B mean curvature large} of $\partial B$ imply that $\spt\|V^*\|\cap \wti{W}\cap \partial B = \emptyset$, which contradicts the choice of $x$.

To show the second part, since $\Sigma\cap\partial B = \spt\|V\|\cap W\cap\partial B$ and $\Gamma = \spt\|V^*\|\cap B$, we only need to show $\Sigma\cap \partial B\subset \closure(\spt\|V^*\|\cap B)\setminus B$. Given any $x\in \Sigma\cap \partial B$, we know by Proposition \ref{prop:tangent plane} that each tangent varifold of $V^*$ at $x$ is an integer multiple of a plane $P$. As $V$ and $V^*$ are identical outside $\closure(B)$, $P$ contains half of the tangent plane $T_x\Sigma$, so we must have $P=T_x\Sigma$. By the assumption, $P$ intersects $\partial B$ transversely, so $x$ must be a limit point of $\spt \|V^*\| \cap B$, and hence $x\in \closure(\spt\|V^*\|\cap B)\setminus B$. \hfill$\Box$

\medskip
\noindent\textit{Step 2: Let $x_i\in \spt\|V^*\|\cap U$ be a sequence of points with $x_i\to x \in U$ and $r_i\to 0$. Denote by $\pmb{\mu}_{x, r}: \mb R^L \to \mb R^L$ the dilation map $\pmb{\mu}_{x, r}(y) = \frac{y-x}{r}$. Then the blow up limit $\oli{V} = \lim(\pmb{\mu}_i)_\# V^*$, where $\pmb{\mu}_i$ denotes $\pmb{\mu}_{x_i, r_i}$, is induced by an embedded minimal surface in $T_x M$.} (Note that this conclusion depends only on the replacement chain property. We refer to a similar result for CMC min-max varifold in \cite{ZZ17}*{Lemma 5.10}.)

\medskip
Clearly $\oli V$ must be stationary in $T_xM$. We will prove that the blow up limit has replacement chain property in any fixed bounded open subset $W\subset T_x M$ for the area functional.\footnote{Note that Definition \ref{def:replacement chain} can be straightforwardly adapted to the area functional.} Then the desired regularity result follows from that of min-max varifolds \cite{Colding-DeLellis03}*{Proposition 6.3}.\footnote{Note that our replacement chain property is stronger than the good replacement property \cite{Colding-DeLellis03}*{Definition 6.2}.} 

We start by showing this for a single open subset. Since $\pmb{\mu}_i(M) \to T_x M$ locally uniformly, we can identify $\pmb{\mu}_i(M)$ with $T_x M$ on compact subsets for $i$ large. For any open subset $\mc B\subset W\subset T_xM$, denote $\mc B'_i:=\bm \mu_i^{-1}(\mc B)\subset U$. Then for sufficiently large $i$, there is an $\mc A^h$-replacement $(V_i,\Omega_i)$ of $(V^*,\Omega^*)$ in $\mc B_i'$. Up to a subsequence we have 
\[ \wti{V} = \lim_{i\to\infty} (\pmb{\mu}_i)_\# V_i \quad \text{ in } \mc V(T_x M). \]
We can deduce the following:
\begin{itemize}
    \item $\wti{V}$ and $\oli{V}$ are identical outside $\closure(\mc B)$.
    \item $\A^h(V_i, \Omega_i) = \A^h(V^*, \Omega^*)$ $\Longrightarrow$ $\|\wti{V}\|(W) = \|\oli{V}\|(W)$. 
    \item By the replacement chain property, $(V_i, \Omega_i)$ is $\A^h$-stationary and stable in $U$, hence $\wti{V}$ is stationary and stable in $T_x M$.
    \item Note that $(\pmb{\mu}_i)_\# (V_i \res {\mc B}_i')$ is a strongly $\A^{r_i h}$-stationary, stable $C^{1,1}$ $(r_i h)$-boundary in ${\mc B}$. Also the mass of $(\pmb{\mu}_i)_\# (V_i \res {\mc B}_i')$ is uniformly bounded by the monotonicity formula. Thus Proposition \ref{prop:compactness} implies that $\wti{V}$ is induced by a stable embedded minimal surface in $\mc {B}$.
\end{itemize}
Thus, $\wti{V}$ is a replacement of $\oli{V}$ in $\mc B$ for the area functional.   

When there is a list of open sets $\mc B_1, \cdots,\mc B_a \subset T_x M$, we can take a chain of replacements $(V_{i, k}, \Omega_{i, k})$ successively in $(\pmb{\mu}_i)^{-1}(\mc B_k)$. By similar arguments as above, we can show that the weak limits $\{\wti{V}_k := \lim_{i\to\infty} (\pmb{\mu}_i)_\# V_{i, k}\}$ is a chain of replacements for $\oli{V}$. This completes the proof of Step 2. \hfill$\Box$

\medskip

Denote by 
\[ \gamma:=\Sigma\cap \partial B.\]
Then $\gamma$ is a $C^{1,1}$-curve by transversality. Now let us fix an arbitrary point $p\in \gamma$.

\medskip
{\noindent\em  Step 3: Suppose that $x_i\in \Gamma$ and $x_i\to p$. Then 
\[ \lim_{i\to\infty} |\langle \nu_\Gamma (x_i),\nu_\Sigma(p)\rangle |= 1,  \]
where $\nu_\Sigma$ and $\nu_\Gamma$ are respectively the unit normal of $\Sigma$ and $\Gamma$ (which are both $C^{1,1}$ surfaces).
}

\medskip
Take $z_i\in \gamma$ such that 
\[  r_i:=\dist_M(x_i,\gamma)=\dist_M(x_i,z_i).\] 
Note that $V^*$ has a unique tangent plane in a neighborhood of $p$. To check this, by the regularity assumption of $V^*$, we only need to check the tangent plane of $V^*$ at any $x\in \spt\|V^*\|\cap \partial B\subset \Sigma$ near $p$. Indeed, any tangent plane $P$ of $V^*$ at $x$ must contain half of $T_x\Sigma$, and hence $P=T_x\Sigma$. Thus the unit normal vector $\nu^*$ of $V^*$ is well-defined. Clearly, $\nu^*(z_i)=\pm \nu_\Sigma(z_i)$, where the sign depends on the choice of the orientation. Since $x_i\to p$, we have $z_i\to p$ and $\nu^*(z_i)=\pm \nu_\Sigma(z_i)\to \pm\nu_\Sigma(p)=\pm\nu^*(p)$. By Step 2, $(\pmb \mu_{x_i, r_i})_\#V^*$ converges to an embedded minimal surface $Q$ in $T_p M$. Moreover, by the regularity of $\Sigma$, we know that $Q$ contains only several half-planes parallel to $T_p\Sigma$ in a half-space of $T_p M$ (one side of $\lim_{i\to\infty} \pmb \mu_{x_i, r_i}(\partial B)$); (we refer to \cite{ZZ17}*{Claim 3(B) on page 477} for more details of the converging scenario for $\{\pmb \mu_{x_i, r_i}(\Sigma)\}$.) Thus by the half space theorem in \cite{HM90}*{Theorem 1}, $Q$ consists of a union of parallel planes counted with integer multiplicity. Moreover, since $\Gamma$ is a stable $C^{1,1}$ $h$-boundary in $B$, the convergence $(\pmb \mu_{x_i, r_i})_\#V^* \to Q$ is $C^{1,\alpha}$ in $B_{1/2}(0)$. Thus $\nu_{\Gamma}(x_i)$ converges to $\nu_Q(=\nu_\Sigma(p))$, which is the unit normal of $Q$. Hence we have finished Step 3. \hfill$\Box$

\medskip
{\noindent\em Step 4: Graphical decomposition of $\Gamma$ around $p$.}

\medskip 
Take a geodesic ball $B_r(p)\subset\subset W$. We can assume that the conclusion in Step 3 holds for any $q\in \gamma \cap B_r(p)$. We will show that $V^*$ has $C^1$-graphical decomposition in a smaller neighborhood of $p$. By the $C^{1,1}$-regularity of $\Sigma$ and possibly shrinking $r$, we may assume that $\Sigma\res B_{10r}(p)$ has an ordered decomposition
\[ \Sigma^1\leq \cdots\leq \Sigma^\ell.\]

We start by introducing a family of cylindrical neighborhoods of $p$. We can identify $B_{10 r}(p)$ with the corresponding ball in $T_pM$. Denote by $P=T_p\Sigma$ the tangent plane, $\bm\pi$ the projection to $P$, and $\mc B_r(p):= B_r (p)\cap P$. Given $s>0$, denote 
\[  \mc K_{r,s}(p) := \{x\in \bm \pi^{-1}(\mc B_r(p)); \dist_M(x,P)\leq sr\}.  \]
Fix a constant $\delta>0$ small enough. Since $\Sigma$ is $C^{1,1}$ near $p$, by taking small enough $r$, $\Sigma$ will be sufficiently flat in the sense that: $ \Sigma\cap \mc K_{r,3\delta}(p) = \Sigma \cap \mc K_{r,\delta}(p)$. Since the tangent varifold of $V^*$ at $p$ is the same as $T_p\Sigma$, we can also assume that 
\[ \spt\|V^*\| \cap \mc K_{r,3\delta}(p) = \spt \|V^*\| \cap \mc K_{r,\delta}(p).\]  
By the argument in Step 3, we have for each $x\in \spt \|V^*\|\cap \mc K_{r,\delta}(p)$, the tangent plane of $\|V^*\|$ at $x$ is unique and the unit normal $\nu^*$ satisfies 
\begin{equation}\label{eq:close normal}
|\langle \nu^*(x),\nu(p)\rangle|>1-\epsilon.
\end{equation}
Here $\epsilon\to 0$ as $r\to 0$ by Step 3.

Next we use the argument in \cite{WZh22}*{Theorem 3.2} (see also \cite{Lin}*{Lemma 2.1}) to construct the desired graphical decomposition. We claim that $\bm\pi$ maps $\spt \|V^*\|\cap \mc K_{\delta,r}(p)$ onto $\mc B_{r}(p)$. Suppose not, then there exist $q\in \mc B_r(p)$, $t>0$ and $y\in \mc K_{r,\delta}(p)\cap \spt\|V^*\|$ such that 
\[ \bm\pi^{-1}\big(\mc B_t(q)\big)\cap \mc K_{r,\delta}(p)\cap \spt \|V^*\|=\emptyset, \quad \text{ and } \bm \pi(y)\in \partial \mc B_t(q).
   \]
Thus, the tangent plane of $V^*$ at $y$ must be parallel to that of the vertical cylinder $\bm\pi^{-1}(\partial\mc B_t(q))$, which contradicts \eqref{eq:close normal}.

Next we define the graph functions over $\mc B_r(p)$ inductively. Let $d^s$ be the signed distance function to $\mc B_r(p)$. Define 
\[
u_1(x):=\inf\{d^s(z):\, z\in \bm\pi^{-1}(\{x\})\cap \spt\|V^*\|\}.
\]
Then by the previous paragraph, $u_1$ is well-defined for all $x\in \mc B_r(p)$. Moreover, by Step 1 and 3, $u_1$ is a $C^1$-function. Denote by $\Gamma^1$ the graph of $u_1$. Then $\wti V^*:=V^*\res \mc K_{r,\delta}(p) - [\Gamma^1]$ is still an integer rectifiable varifold so that the tangent varifold of $y\in \spt \|\wti V^*\|$ is a unique plane which satisfies \eqref{eq:close normal}. Then we define $u_j$ and $V^*(j)$ inductively as follows:
\begin{gather*}
V^*(j):=V^*\res \mc K_{r,\delta}(p)-\sum_{i\leq j}[\Gamma^i], \quad \text{ where } \Gamma^i:=[\mr{graph}\, u_i];\\
 u_{j+1}(x):=\inf\left\{ d^s(z):\, z\in\bm \pi^{-1}(\{x\})\cap \spt \|V^*(j)\|\right\}.
\end{gather*}
By the same argument for $u_1$, one can prove that if $\spt\|V^*(j)\|$ is non-empty, then $u_{j+1}$ is well-defined for all $x\in \mc B_r(p)$. By the decomposition of $\Sigma$ in $B_{10r}(p)$, we conclude that $V^*$ consists of exactly $\ell$ number of $C^1$-graphs in $\mc K_{r, \delta}(p)$:
\[ \Gamma^1\leq \cdots\leq \Gamma^\ell. \]
This finishes Step 4. \hfill $\Box$

\medskip
So far, we have proved that $V^*$ is induced by a union of ordered $C^1$-surfaces near $p$, and hence $(V^*, \Omega^*)\res B_r(p)$ is an $\A^h$-stationary $C^1$-boundary. Moreover, by assumption, $(V^*, \Omega^*)\res (B_r(p)\setminus \partial B)$ is $C^{1,1}$ and strongly $\mc A^h$-stationary. Thus by Lemma \ref{lem:remove C1 line}, $(V^*, \Omega^*)\res B_r(p)$ is $C^{1,1}$ and strongly $\mc A^h$-stationary. Since the strong $\A^h$-stationarity is a local notion, this implies that $(V^*, \Omega^*)\res (W\cup B)$ is $C^{1,1}$ and strongly $\A^h$-stationary. This completes the proof of Lemma \ref{lem:glue}.
\end{proof}

Now we are ready to prove the first main regularity result. We use the existence of chains of replacements in a totally new way, as compared with Pitts \cite{Pi}*{Chap. 7}. 

\begin{theorem}[First Regularity]\label{thm:1st reg}
Given an open set $U\subset M$, let $(V, \Omega)\in \VC(M)$ satisfy the replacement chain property in $U$ (see Definition \ref{def:replacement chain}). Then $(V, \Omega)$ is induced by a strongly $\A^h$-stationary and stable $C^{1, 1}$ $h$-boundary in $U$. 
\end{theorem}

\begin{proof}
By the replacement chain property, we know that $(V, \Omega)$ is $\A^h$-stationary and stable in $U$, so we only need to prove the $C^{1,1}$-regularity and strong $\A^h$-stationarity.

By Proposition \ref{prop:tangent plane}, we know that $\Theta^2(\|V\|, x)\in \mb N$ for any $x\in \spt\|V\|\cap U$. We will prove the desired regularity by doing induction on $\Theta^2(\|V\|, x)$.  Fix an integer $m\in \mb N$. Define
\[\mc S(V, m) = \left\{x\in \spt\|V\|\cap U: \Theta^2(\|V\|, x) =m\right \},\]
\[ \mc S(V, \leq m) =  \cup_{i=1}^m \mc S(V, i), \] 
and
\[ \mc S(V, > m) =  \spt\|V\|\cap U \setminus \mc S(V, \leq m). \]
By the upper semi-continuity of the density function $\Theta^2(\|V\|, \cdot)$ (which holds as $V$ has uniformly bounded first variation via Lemma \ref{lem:Ah stationary pair in VC has bounded 1st variation}), we know that for each $m\in \mb N$,
\[\text{$\mc S(V, \leq m)/\mc S(V, >m)$ is a relatively open/closed subset of $\spt\|V\|\cap U$.}\] 

By the Allard regularity, standard elliptic regularity, and Lemma \ref{lem:regular part is PMC}, we know that $\mc S(V, 1)$ constitutes a smoothly embedded surface $\Sigma_1$, whose mean curvature w.r.t. the outer normal of $\Omega$ is prescribed by $h$.

Suppose by induction that $V\res \mc S(V, \leq m-1)$ is induced by a $C^{1,1}$-almost embedded surface $\Sigma_{m-1}$, and $(\Sigma_{m-1}, \Omega)$ is a strongly $\mc A^h$-stationary, stable, $h$-boundary in  $U\setminus \mc S(V, >m-1)$ for some $m\geq 2$. We will prove the same regularity for $V\lc \mc S(V, \leq m)$ so as to finish the induction.

\medskip

Now we fix $q\in \mc S(V,m)$ and a sufficiently small constant $\delta>0$. By  the monotonicity formula and the fact that $\mc S(V, \leq m)$ is relatively open in $\spt\|V\|$, we can find 
\begin{itemize}
    \item $s_0>0$, and
    \item a geodesic ball with sufficiently small radius $U_0\subset U$ of $q$ with $\spt\|V\|\cap \closure(U_0) \subset \mc S(V, \leq m)$ (note that this implies $U_0\setminus \mc S(V, m)$ is an open subset of $U_0$),
\end{itemize}  
such that for all $x\in \spt\|V\| \cap \closure(U_0)$ and $s\leq s_0$,
\begin{equation}\label{eq:ratio bound}
\frac{\|V\|(B_s(x))}{\pi s^2}\leq m+\delta. 
\end{equation}
Without loss of generality, we may assume that 
\[ \closure(U_0)\subset \cap_{x\in U_0} B_{s_0/2}(x). \]

Now fix $\bm r\in (0,s_0/4)$ small enough (we will let $\bm r\to 0$ in the end), and let $B_1,\cdots ,B_N$ be finitely many geodesic balls centered on $\mc S(V, m)\cap \closure(U_0)$ with radius $\bm r$ in $M$, so that 
\[  \mc S(V,m)\cap \closure(U_0)\subset \cup_{i=1}^NB_i, \quad \text{and} \quad B_i \cap \big(\mc S(V,m)\cap \closure(U_0)\big) \neq \emptyset.\]
We can also assume that the mean curvature $H_{\partial B_i}>\|h\|_{L^\infty(M)}$.

By slightly enlarging $B_1$ to $\wti B_1$, we may assume that $\partial \wti{B}_1$ is transverse to $\Sigma_{m-1}$ (if $\partial\wti{B}_1\cap \Sigma_{m-1}\neq \emptyset$).  Let $(V_1,\Omega_1)$ be an $\A^h$-replacement of $(V, \Omega)$ in $\wti{B}_1$. Then $(V_1, \Omega_1)\res \wti{B}_1$ is a strongly $\mc A^h$-stationary, $C^{1,1}$ $h$-boundary in $\wti{B}_1$. Furthermore, by Remark \ref{rem:consequence of replacement chain property}, $(V_1, \Omega_1)$ also satisfies the replacement chain property in $U$, so $(V_1, \Omega_1)$ is $\A^h$-stationary and stable in $U$, and hence $V_1$ is integer rectifiable in $U$ by Proposition \ref{prop:tangent plane}. Moreover, by Lemma \ref{lem:glue}, it is also $C^{1,1}$ and strongly $\mc A^h$-stationary in $\wti{B}_1\cup [U_0\setminus \mc S(V, m)]$.

\medskip
Next we construct a sequence $(V_2,\Omega_2),\cdots, (V_N,\Omega_N)\in \VC(M)$ and a sequence of balls $\wti{B}_2, \cdots, \wti B_N$ inductively, so that for $j=2,\cdots,N$,
\begin{itemize}
    \item $\wti B_j\supset B_j$ is a ball with radius slightly larger than $\bm r$; 
    \item $(V_j,\Omega_j)$ is an $\A^h$-replacement of $(V_{j-1},\Omega_{j-1})$ in $\wti B_j$;
    \item $(V_j, \Omega_j)$ satisfies the replacement chain property in $U$ and hence is $\A^h$-stationary, stable and integer rectifiable in $U$; 
    \item $(V_{j},\Omega_{j})$ is $C^{1,1}$ and strongly $\mc A^h$-stationary in 
    \[  \wti{B}_1 \cup\cdots\cup \wti B_{j}\cup [U_0\setminus \mc S(V,m)].  \]
\end{itemize}   
In fact, assume that we have finished the construction at the $(j-1)$-th step, we can take a slightly larger ball $\wti B_j\supset B_j$ so that $\partial \wti B_j$ is transverse to $\spt \|V_{j-1}\|$ in $\wti B_1\cup\cdots\cup \wti B_{j-1}\cup [U_0\setminus \mc S(V,m)]$. By the replacement chain property of $(V_{j-1}, \Omega_{j-1})$, we can always find an $\A^h$-replacement $(V_j,\Omega_j)$ of $(V_{j-1},\Omega_{j-1})$ in $\wti B_j$. Then by Remark \ref{rem:consequence of replacement chain property} and Lemma \ref{lem:glue}, $(V_j, \Omega_j)$ is $C^{1,1}$ and strongly $\mc A^h$-stationary in 
$\wti B_1\cup \cdots \cup \wti B_{j} \cup [ U_0\setminus  \mc S(V,m)]$, so that we can continue the induction process. 

Since $U_0\subset \wti B_1 \cup \cdots\cup \wti B_N \cup [U_0\setminus \mc S(V, m)]$, we know that $(V_N,\Omega_N)$ is a strongly $\mc A^h$-stationary, stable, $C^{1,1}$ $h$-boundary in $U_0$, and 
\begin{equation}\label{eq:Ah equals}
\mc A^h(V,\Omega)=\mc A^h(V_1,\Omega_1)=\cdots=\mc A^h(V_N,\Omega_N).
\end{equation}

\medskip
So far, for each $\bm r>0$, we can cover the $\mc S(V,m)\cap \mr{Clos}(U_0)$ by finitely many small balls with radius $\leq 2\bm r$ and then use the replacement chain property to get a pair in $\VC(M)$, denoted by $(V_{\bm r},\Omega_{\bm r})$, which is a strongly $\mc A^h$-stationary, stable $C^{1,1}$ $h$-boundary in $U_0$. Note that $(V_{\bm r}, \Omega_{\bm r})$ is still $\A^h$-stationary and integer rectifiable in $U$, and
\[ (V_{\bm r} ,\Omega_{\bm r}) = (V, \Omega) \text{ outside } \closure(\ms B_{\bm r}), \text{ where } \ms B_{\bm r} =\cup_{j=1}^N \wti B_j. \]
By the standard compactness theory for integral varifolds with uniformly bounded first variation and for sets of finite perimeter, we have (up to a subsequence) 
\[ (V_{\bm r}, \Omega_{\bm r}) \to (V^*, \Omega^*) \text{ under the $\ms F$-metric}, \]
where $V^*$ is integer rectifiable in $U$. Moreover, by Proposition \ref{prop:compactness}, $(V_{\bm r}, \Omega_{\bm r})\res U_0$ converges to a strongly $\mc A^h$-stationary, stable, $C^{1,1}$ $h$-boundary $\Sigma^*\subset U_0$ (associated with $\Omega^*$) in the sense of $C^{1,\alpha}_{loc}$, that is, $(V^*, \Omega^*)\res U_0 = (\Sigma^*, \Omega^*\res U_0)$. Observe that $\mc S(V, m)\cap \closure(U_0)$ is a compact and $2$-rectifiable subset of $M$,
and for each $\bm r>0$ and for any $x\in \wti B_j$, we have $\dist_M(x,\mc S(V,m))\leq 2\bm r$. It follows that  
\[  \ms B_{\bm r} \to \mc S(V,m) \cap \closure(U_0), \quad \text{ as } \bm r\to 0. \]
Hence we have that 
\begin{equation}\label{eq:spt}
\spt \|V^*\|\subset \spt \|V\|, \quad \text{and}\quad \Omega^*=\lim_{r\to 0}\Omega_r=\Omega.
\end{equation}
Moreover,
\begin{equation}\label{eq:in U0}
V^*\res \big(U\setminus [\mc S(V,m) \cap \closure(U_0)] \big)=V\res \big(U\setminus [\mc S(V,m) \cap \closure(U_0)] \big).  
\end{equation}
Note that \eqref{eq:Ah equals} gives
\begin{equation*}
    \mc A^h(V^*,\Omega^*)=\mc A^h(V,\Omega).
\end{equation*}  
Together with \eqref{eq:spt} and \eqref{eq:in U0}, we obtain
\begin{equation}\label{eq:bar U0}
    \|V^*\|\big(\mr{Clos}(U_0)\big)=\|V\|\big(\mr {Clos}(U_0)\big). 
\end{equation}
Since $\closure(U_0)\subset B_{s_0}(y)$ for all $y\in \mc S(V,m)\cap \closure(U_0)$, by \eqref{eq:ratio bound}, \eqref{eq:in U0} and \eqref{eq:bar U0}, we have
\[
\frac{\|V^*\|(B_{s_0}(y))}{\pi s_0^2}=\frac{\|V\|(B_{s_0}(y))}{\pi s_0^2}<m+\delta.
\]
By the monotonicity formula and the fact that $\Theta^2(\|V^*\|, y)\in \mb N$ for $\mc H^2$-a.e. $y\in \spt\|V^*\|\cap U$, we have that  $\Theta^2(\|V^*\|, y)\leq m$ for $\mc H^2$-a.e. $y\in \mc S(V, m)\cap \closure(U_0)$. Recall that for $\mc H^2$-a.e. $y\in \mc S(V, m)\cap \closure(U_0)$, $\Theta^2(\|V\|, y)=m$.
Thus we obtain 
\[ \Theta^2(\|V^*\|, y)\leq \Theta^2(\|V\|, y), \quad \text{for $\mc H^2$-a.e. } y \in \mc S(V, m)\cap \mr{Clos}(U_0). \]
Together with \eqref{eq:spt}, \eqref{eq:in U0} and \eqref{eq:bar U0}, we conclude that $V^*=V$. This gives the desired regularity of $V$ near $q$. By the arbitrariness of $q$, we proved the desired regularity of $\mc S(V,\leq m)$. This completes the induction and the theorem is proved.
\end{proof}

\subsection{Construction of replacements}\label{ss:construction of replacements}
Given an embedded separating surface $(\Sigma,\Omega)\in \ms E$, an open set $U\subset M$ and $\delta>0$, set 
\[ \mk{Is}^h_\delta(U)=\{ \psi\in \mk{Is}(U); \ \mc A^h\big( \psi(t,\Sigma, \Omega)\big)\leq \mc A^h(\Sigma,\Omega)+\delta, \ \forall\, t\in[0,1]\}  \]
and 
\[  m_\delta:=\inf\{ \mc A^h\big(\psi(1,\Sigma,\Omega)\big); \psi \in \mk{Is}^h_\delta(U)\}.  \]
We say that a sequence $\{(\Sigma_k,\Omega_k)\}_{k\in \mb N}\subset \ms E$ is {\em minimizing in Problem $(\Sigma,\Omega, \mk{Is}_\delta^h(U))$} if there exists a sequence $\{\psi_k\}_{k\in\mb N}\in \mk{Is}_\delta^h(U)$ with 
\[ (\Sigma_k,\Omega_k)= \psi_k(1,\Sigma,\Omega), \quad \text{ and }\quad \mc A^h(\Sigma,\Omega)\geq \mc A^h(\Sigma_k,\Omega_k)\to m_\delta, \text{ as } k\to\infty.   \]

We will use the following lemma, which says that any isotopy in a sufficiently small ball that does not increase $\A^h(\Sigma_k, \Omega_k)$ can be realized by an isotopy in $\mk{Is}^h_\delta(U)$. 
The proof will be given in Appendix \ref{appen:proof of minimizing lemma in small ball}.

\begin{lemma}\label{lem:minimizing in small ball}
Suppose that $\{(\Sigma_k,\Omega_k)\}_{k\in \mb N}$ is minimizing in Problem $(\Sigma,\Omega,\mk{Is}^h_\delta(U))$. Then given $U'\subset \subset U$, there exists $\rho_0>0$ such that for $k$ sufficiently large, the following holds: for any $B_{2\rho}(x)\subset U'$ ($\rho<\rho_0$) and $\varphi\in\mk{Is}(B_{\rho} (x))$ with $\mc A^h (\varphi(1,\Sigma_k,\Omega_k )) \leq \mc A^h (\Sigma_k,\Omega_k )$, there
exists an isotopy $\Phi\in \mk{Is}(B_{2\rho}(x))$ such that
\begin{gather*}
\Phi(1,\cdot)=\varphi(1,\cdot),\  \text{ and } \ \mc A^h(\Phi(t,\Sigma_k,\Omega_k)\leq \mc A^h(\Sigma_k,\Omega_k)+\delta, \ \text{for all }\, 0\leq t\leq 1.
\end{gather*}
Moreover, the constant $\rho_0$ depends on $\mc H^2(\Sigma)$, $\|h\|_{L^\infty(M)}$, $U'$, $M$ and $\delta$, but does not depend on the minimizing sequence $\{(\Sigma_k,\Omega_k)\}$.
\end{lemma}

Now we will use the First Regularity Theorem \ref{thm:1st reg} to prove the regularity for  constrained $h$-minimizing problems. The key step is to verify the replacement chain property where we need Lemma \ref{lem:minimizing in small ball}.

\begin{proposition}[Regularity of constrained $h$-minimizer]\label{prop:reg of con min}
Assume that $(\Sigma,\Omega)\in \ms E$ is $(\mc A^h,\epsilon, \delta)$-almost minimizing in $U$. Suppose that $\{(\Sigma_k,\Omega_k)\}$ is minimizing in Problem $(\Sigma,\Omega,\mk{Is}_\delta^h(U))$. Then $(\Sigma_k,\Omega_k)$ converges (subsequentially without relabeling) to some $(\hat V,\hat \Omega)\in \VC(M)$ with
\begin{equation}\label{eq:Ah value comparison after replacement}
\A^h(\Sigma, \Omega) - \epsilon\leq \A^h(\hat{V}, \hat{\Omega}) \leq \A^h(\Sigma, \Omega);    
\end{equation}
and moreover, $(\hat V, \hat\Omega)\res U$ is a strongly $\mc A^h$-stationary and stable $C^{1,1}$ $h$-boundary in $U$.
\end{proposition}

\begin{proof}
Clearly \eqref{eq:Ah value comparison after replacement} follows from the definition of $(\A^h, \epsilon, \delta)$-almost minimizing property. Since $\{(\Sigma_k,\Omega_k)\}$ is minimizing in Problem $(\Sigma,\Omega,\mk{Is}_\delta^h(U))$, we conclude that $(\hat V,\hat \Omega)$ is $\mc A^h$-stationary and stable in $U$.  The proof is similar to that in Lemma \ref{lem:am implies stationary and stable}.

We next prove the regularity. Fix $p\in U$. Denote by 
\[ r_1:=\dist_M(p,\partial U).\]
Let $\rho_0$ be the constant in Lemma \ref{lem:minimizing in small ball} for the Problem $(\Sigma, \Omega, \mk{Is}_\delta^h(U))$ and $U'=B_{r_1/4}(p)$, and
\[ r_0:=\min\{ \rho_0, r_1/4\}.\]
We will prove that $(\hat V,\hat \Omega)$ satisfies the replacement chain property in $B_{r_0}(p)$. Then by Theorem \ref{thm:1st reg}, $(\hat V, \hat\Omega)\res B_{r_0}(p)$ is a strongly $\mc A^h$-stationary and stable $C^{1,1}$ $h$-boundary. By the arbitrariness of $p$, it follows that $(\hat V, \hat\Omega)\res U$ is a strongly $\mc A^h$-stationary, $C^{1,1}$ $h$-boundary. 

Consider a finite collection of open subsets $B_1, \cdots , B_j\subset\subset B_{r_0}(p)$. Let $\{(\Sigma_{k,\ell},\Omega_{k,\ell})\}_{\ell\in\mb N}$ be {\em a} {\em minimizing} {\em sequence} {\em of} {\em Problem} $(\Sigma_k,\Omega_k,\mk{Is}^h(B_1))$, i.e. there exist $\{\Psi_{k,\ell}\}\subset \mk{Is}(B_1)$ such that 
\[
(\Sigma_{k,\ell},\Omega_{k,\ell})=\Psi_{k,\ell}(1,\Sigma_k,\Omega_k),\quad 
\lim_{\ell\to\infty}\mc A^h(\Sigma_{k,\ell},\Omega_{k,\ell}) = \inf\{\mc A^h(\psi(1,\Sigma_k,\Omega_k);\psi\in\mk{Is}(B_1)\}.
\]
Denote by $(\wti V_k,\wti \Omega_k)$ the limit of $(\Sigma_{k,\ell},\Omega_{k,\ell})$ as $\ell\to \infty$. Then $(\wti V_k, \wti\Omega_k)\res B_1$ is a strongly $\mc A^h$-stationary and stable $C^{1,1}$ $h$-boundary in $B_1$ by Theorem \ref{thm:interior regularity}. Denote by $(\wti V,\wti\Omega)$ the limit of $(\wti V_k,\wti \Omega_k)$ as $k\to \infty$.  We can take sufficiently large $\ell(k)$ so that 
\[  (\Sigma_{k,\ell(k)},\Omega_{k,\ell(k)}) \to (\wti V,\wti\Omega) \quad \text{ as } k\to\infty.  \]
By the definition of $(\Sigma_{k, \ell}, \Omega_{k, \ell})$, we have that $(\wti V,\wti \Omega)=(\hat V,\hat \Omega)$ on $U\setminus \mr{Clos}(B_1)$. Then by the regularity of $(\wti V_k, \wti\Omega_k)\res B_1$ and Proposition \ref{prop:compactness}, $(\wti V, \wti\Omega)\res B_1$ is a strongly $\mc A^h$-stationary and stable $C^{1,1}$ $h$-boundary.

Note that by Lemma \ref{lem:minimizing in small ball}, each $(\Sigma_{k,\ell(k)},\Omega_{k,\ell(k)})$ can be constructed from $(\Sigma_k,\Omega_k)$ via an isotopy $\Phi_k\in \mk{Is}(B_{2r_0}(p))$ satisfying 
\[ \A^h\big(\Phi_k(t, \Sigma_k, \Omega_k)\big) \leq \A^h(\Sigma_k, \Omega_k) + \delta, \text{ for all } 0\leq t\leq 1. \]
Since $\{(\Sigma_k, \Omega_k)\}$ is a minimizing sequence in Problem $(\Sigma, \Omega, \mk{Is}_{\delta}^h(U))$, there exist isotopies $\{\Phi_k'\}\subset \mk{Is}(U)$, such that
\[(\Sigma_k, \Omega_k) = \Phi_k'(1, \Sigma, \Omega),\quad \A^h(\Sigma_k, \Omega_k)\leq \A^h(\Sigma, \Omega) \text{ and}\] 
\[ \A^h\big(\Phi_k'(t, \Sigma, \Omega)\big) \leq \A^h(\Sigma, \Omega) + \delta, \text{ for all } 0\leq t\leq 1. \]
By concatenating $\Phi_k'$ with $\Phi_k$, we know that  $(\Sigma_{k,\ell(k)},\Omega_{k,\ell(k)})$ is obtained from $(\Sigma, \Omega)$ via an isotopy $\wti \Phi_k\in\mk{Is}(U)$ with
\[
\wti\Phi(1,\Sigma,\Omega)=(\Sigma_{k,\ell(k)},\Omega_{k,\ell(k)}),\quad   \mc A^h(\wti \Phi(t,\Sigma,\Omega))\leq \mc A^h(\Sigma,\Omega)+\delta, \text{ for all } 0\leq t\leq 1.
\]
Thus $\{(\Sigma_{k,\ell(k)},\Omega_{k,\ell(k)})\}$ is also minimizing in Problem $(\Sigma,\Omega,\mk{Is}_\delta^h(U))$, which implies that $\mc A^h(\hat V,\hat \Omega)=\mc A^h(\wti V,\wti\Omega)$.

All of these give that $(\wti V, \wti\Omega)$ is an $\A^h$-replacement of $(\hat V, \hat\Omega)$ in $B_1$. Observe that the replacement $(\wti V, \wti\Omega)$ is also the limit of a minimizing sequence in Problem $(\Sigma,\Omega,\mk{Is}_\delta^h(U))$. Therefore $(\wti V,\wti \Omega)$ has a replacement in $B_2$. Continuing the process, we have proved that $(\hat V,\hat \Omega)$ satisfies the replacement chain property in $B_{r_0}(p)$. This finishes the proof of Proposition \ref{prop:reg of con min}.
\end{proof}

\subsection{Interior regularity for almost minimizing pairs}\label{ss:regularity for almost minimizing pairs}
In this part, we use the regularity of constrained $h$-minimizer (Proposition \ref{prop:reg of con min}) to construct $C^{1,1}$-replacements of a given $\mc A^h$-almost minimizing pair. Then the First Regularity Theorem \ref{thm:1st reg} applies to give the full regularity.

\begin{theorem}[Regularity of $\A^h$-almost minimizing pairs]\label{thm:replacement for AM}
Given an open set $U\subset M$, let $(V, \Omega)\in\VC(M)$ be $\A^h$-almost minimizing in $U$; that is, 
\begin{itemize}
    \item there exists a sequence $\{(\Sigma_j, \Omega_j)\}_{j\in \mb N}\subset \ms E$ that converges to $(V,\Omega)$ under the $\ms F$-metric;
    \item $(\Sigma_j,\Omega_j)$ is $(\mc A^h, \epsilon_j, \delta_j)$-almost minimizing in $U$ for some $\epsilon_j\to 0$ and $\delta_j\to 0$ as $j\to\infty$.
\end{itemize}
Then $(V, \Omega)\res U$ is induced by a strongly $\mc A^h$-stationary and stable $C^{1,1}$ $h$-boundary.
\end{theorem}

\begin{proof}
Clearly, $(V, \Omega)\res U$ is $\mc A^h$-stationary and stable by Lemma \ref{lem:am implies stationary and stable}. To prove the desired regularity, we will show that given any $U'\subset \subset U$, $(V,\Omega)$ satisfies the replacement chain property in $U'$; then it follows from Theorem \ref{thm:1st reg}.

Fix $U'\subset \subset U$, 
and let $B_1,\cdots, B_N\subset\subset U'$ be an arbitrary collection of open subsets. We now construct the desired chain of replacements as follows. We start with $B_1$. Let $\{(\Sigma_{j,\ell},\Omega_{j,\ell})\}_{\ell\in \mb N}$ be
a minimizing sequence for Problem $(\Sigma_j,\Omega_j,\mk{Is}_{\delta_j}^h(B_1))$. Then by Proposition \ref{prop:reg of con min}, $(\Sigma_{j,\ell},\Omega_{j,\ell})$ converges to some $(\hat V_j,\hat \Omega_j)\in \mc{VC}(M)$ which is a strongly $\mc A^h$-stationary and stable $C^{1,1}$ $h$-boundary in $B_1$, and satisfies 
\begin{equation}\label{eq:Ah comparison after replacements 2}
\A^h(\Sigma_j, \Omega_j)-\epsilon_j\leq \A^h(\hat V_j, \hat\Omega_j) \leq \A^h(\Sigma_j, \Omega_j).    
\end{equation}
We remark that 
\begin{equation}\label{eq:Vj and Sigmaj}
(\hat V_j, \hat\Omega_j)\res (M\setminus \mr{Clos}(B_1)) = (\Sigma_j, \Omega_j)\res (M\setminus \mr{Clos}(B_1)).
\end{equation}
Now letting $j\to \infty$, by Proposition \ref{prop:compactness}, $(\hat V_j,\hat \Omega_j)$ converges to $(V^*,\Omega^*)$, where $(V^*, \Omega^*)\res B_1$ is a strongly $\mc A^h$ stationary, stable, $C^{1,1}$ $h$-boundary, and $\hat V_j\res B_1\to V^*\res B_1$ in the sense of $C^{1,\alpha}_{loc}$. 

We now verify that $(V^*,\Omega^*)$ is a replacement of $(V,\Omega)$ in $B_1$. Note that by \eqref{eq:Vj and Sigmaj}, \[ (V^*, \Omega^*)\res  (M\setminus \mr{Clos}(B_1)) = (V, \Omega)\res (M\setminus \mr{Clos}(B_1)).\]  
Letting $j\to \infty$, we obtain from \eqref{eq:Ah comparison after replacements 2} that
\[ \mc A^h(V^*,\Omega^*)=\mc A^h(V,\Omega).\]
So far, we have verified that $(V^*,\Omega^*)$ is a replacement of $(V, \Omega)$ in $B_1$.

\medskip
In the next, we will show that $(V^*,\Omega^*)$ is also the limit of a sequence $\{(\Sigma_j',\Omega'_j)\}\subset \ms E$ which is $(\mc A^h, \epsilon_j, \delta_j)$-almost minimizing in $U$. Then by the same argument, one can construct a replacement of $(V^*,\Omega^*)$ in $B_2$. Moreover, we can continue the process and then the lemma is proved.

Indeed, for each $j\in \mb N$, we can take $\ell(j)$ sufficiently large so that $(\Sigma_{j,\ell(j)},\Omega_{j,\ell(j)})$ converges to $(V^*,\Omega^*)$ as $j\to\infty$, and 
\[  \mc A^h(\Sigma_{j,\ell(j)},\Omega_{j,\ell(j)})\leq \mc A^h(\Sigma_{j},\Omega_{j}).\]
Together with the fact that $(\Sigma_{j,\ell(j)},\Omega_{j,\ell(j)})$ can be constructed from $(\Sigma_j,\Omega_j)$ via an isotopy in $\mk{Is}_{\delta_j}^h(B_1)$, we can then conclude that $(\Sigma_{j,\ell(j)},\Omega_{j,\ell(j)})$ is $(\mc A^h, \epsilon_j, \delta_j)$-almost minimizing in $U$ (similarly as in Proposition \ref{prop:reg of con min}). This completes the proof of Theorem \ref{thm:replacement for AM}.
\end{proof}

\subsection{Proof of Theorem \ref{thm:pmc min-max theorem}}\label{SS:proof of pmc min-max}

We will show that the pair $(V_0, \Omega_0)$ obtained in Theorem \ref{thm:existence of almost minimizing pairs} satisfies the regularity conclusion of Theorem \ref{thm:pmc min-max theorem}. We can find finitely many balls $\{B_{r_i}(p_i)\}_{i=1}^{\mk m}$, where $r_i$ is the almost minimizing radius given by Theorem \ref{thm:existence of almost minimizing pairs}, to cover $M$. Then for each small enough open set $U$ lying in some $\mr{An}(p_i; 0, r_i)$, $(V_0, \Omega_0)$ is $\A^h$-almost minimizing w.r.t. $\{(\Sigma_j, \Omega_j)\}$ in $U$, and by Theorem \ref{thm:replacement for AM}, $(V_0, \Omega_0)\res U$ is a $C^{1,1}$ and strongly $\A^h$-stationary boundary. So we know that $(V_0, \Omega_0)$ is a $C^{1,1}$ and strongly $\A^h$-stationary boundary in $M\setminus\{p_1, \cdots, p_{\mk m}\}$.

Next, we prove the $C^{1,1}$-version of removable singularity result, that is, $(V_0, \Omega_0)$ extends as a $C^{1,1}$ and strongly $\A^h$-stationary boundary across each $p_i$. The argument will be the same way as in the smooth case; see \cite{ZZ17}*{Step 4, page 479}. We write $V_0$ as $\Sigma_0$. Fix $p_i$ and we drop the index for a moment. Assume that $\Theta^2(\|V_0\|, p)=m\in \mb N$. By Lemma \ref{lem:am implies stationary and stable}, $(\Sigma_0, \Omega_0)$ is $\A^h$-stable in any annulus $\mr{An}\subset \mr{An}(p; 0, r_p)$. By Proposition \ref{prop:compactness} and Proposition \ref{prop:tangent plane}, we have that for any sequence $r_j\to 0$, 
\[ \pmb \mu_{p, r_j}(\Sigma_0) \to m\cdot S, \]
in $C^{1,1}_{loc}(T_pM\setminus\{0\})$ for some 2-plane $S\subset T_pM$. (Note that a prior, $S$ may depend on $\{r_j\}$.) Therefore, there exists $\sigma_0>0$ small enough, such that for any $0<\sigma\leq\sigma_0$, $\Sigma_0$ has an $m$-sheeted, ordered, $C^{1,1}$-graphical decomposition in $\mr{An}(p;\sigma/2, \sigma)$:
\[ \Sigma_0\res\mr{An}(p; \sigma/2, \sigma) = \cup_{i=1}^m\Gamma^i(\sigma). \]
By shrinking $\sigma\to 0$, we can continue each sheet $\Gamma^j(\sigma_0)$ in $\mr{An}(p; \sigma_0/2, \sigma_0)$ to the whole punctured ball $B_{\sigma_0}(p)\setminus\{p\}$, and we denote this sheet by $\Gamma^j$. By Corollary \ref{cor:mean curvature charaterization}, each $\Gamma^j$ has $c$-bounded first variation in $B_{\sigma_0}(p)\setminus\{p\}$ for $c=\|h\|_{L^\infty}$. By a standard cutoff trick, each $\Gamma^j$ can be extended as a varifold with $c$-bounded first variation in $B_{\sigma_0}(p)$. It is easy to see that each tangent varifold $C^j$ of $\Gamma^j$ at $p$ is an integer multiple of some $2$-plane by the $C^{1,1}_{loc}$ convergence. Moreover, the multiplicity of $C^j$ has to be one as $\Theta^2(\|V_0\|, p)=m$. By the Allard regularity theorem, each $\Gamma^j$ extends to a $C^{1, \alpha}$-surface across $p$. Since their union $\Sigma_0\res B_{\sigma_0}(p) = \cup\Gamma^j$ is $\A^h$-stationary, $\Sigma_0$ must be $C^{1,1}$ by the regularity result Proposition \ref{prop:C11 regularity}. The strongly $\A^h$-stationarity also extends across $p$ by the same standard cutoff trick. This completes the proof of Theorem \ref{thm:pmc min-max theorem}.

\section{Passing to limit}\label{sec:passing to limit}

Given an $h\in C^\infty(M)$ and a sequence of positive numbers $\varepsilon_k\to 0$ as $k\to\infty$, for simplicity, we write $\A^k$ for $\A^{\varepsilon_k h}$ in this part. Assume that $X\subset I(m, k_0)$ is a cubic complex and $Z\subset X$ is a sub-complex. For a given $(X, Z)$-homotopy class $\Pi$ generated by some fixed continuous $\Phi_0: X\to \ms E$ relative to $\Phi_0|_Z$, we consider the min-max problems associated with $\Pi$ for each $\A^k$, $k\in \mb N$. We assume that the nontriviality condition \eqref{eq:width nontrivial1} holds for all $k$. For each $k\in \mb N$, applying Theorem \ref{thm:existence of almost minimizing pairs} to the $\A^k$-functional, we obtain a min-max pair $(V_k, \Omega_k)\in \VC(M)$ and an associated min-max sequence $\{(\Sigma_{k, j}, \Omega_{k, j})\}_{j\in \mb N}\subset \ms E$, such that $(V_k, \Omega_k)$ is $\A^k$-stationary and $\A^k$-almost minimizing in small annuli w.r.t. $\{(\Sigma_{k, j}, \Omega_{k, j})\}$. By Theorem \ref{thm:pmc min-max theorem}, $(V_k, \Omega_k)$ is a strongly $\A^k$-stationary, $C^{1,1}$ $(\varepsilon_k h)$-boundary $(\Sigma_k, \Omega_k)$ with $\A^k(\Sigma_k, \Omega_k) = \mf L^{\varepsilon_k h}(\Pi)$.

Let 
\[ \text{$V_\infty\in \mc V(M)$ be a subsequential varifold limit of $\{\Sigma_k\}$}. \]
Then it is clear that $V_\infty$ is stationary for the area functional. 

In this section, we will show that $V_\infty$ is induced by a closed embedded minimal surface, and the weighted genus bound (see \eqref{eq:genus bound0}) holds for specially chosen $h$.

\subsection{Strong convergence}

In this part, we will show the smooth regularity of $V_\infty$ and the $C^{1, 1}_{loc}$-subsequential convergence of $\Sigma_k$ to $V_\infty$.  By Corollary \ref{cor:property R}, we know that for every $L=L(m)$-admissible (see \eqref{eq:L(m)}) collection $\ms C$ of annuli, $(\Sigma_k, \Omega_k)$ is $\A^k$-stable in at least one annulus in $\ms C$. Therefore, we know that $V_\infty$ also satisfies\footnote{It is clear that the varifold limits of $\A^k$-stable pairs as $k\to\infty$ are $\A^0$-stable.} 
\begin{equation}\label{eq:property R2}
\begin{aligned}
    \text{Property {\bf(R')}}:\quad  & \text{for every $L(m)$-admissible collection $\ms C$ of annuli,}\\
    & \text{$V_\infty$ is stable (for area) in at least one annulus in $\ms C$.}
\end{aligned}
\end{equation}

\begin{proposition}\label{prop:annulus picking for Ak almost minimizing}
There exists a subsequence (without relabelling) of $\{(\Sigma_k, \Omega_k)\}_{k\in \mb N}$, 
such that
\begin{equation}\label{eq:property S}
\begin{aligned}
    \text{Property {\bf(S)}}:\quad  & \text{given any $p\in M$, there exists $r_p>0$, such that}\\
    & \text{for each $\mr{An}(p; s, r)$ with $0<s<r<r_p$, }\\
    & \text{$(\Sigma_k, \Omega_k)$ is $\A^k$-stable in $\mr{An}(p; s, r)$ for all sufficiently large $k$.}
\end{aligned}
\end{equation}
\end{proposition}
\begin{proof}
Let $\mc P_k$ be the collection of annulli $\mr{An}$ so that $(\Sigma_k, \Omega_k)$ is $\A^k$-stable in $\mr{An}$. Obviously if $\mr{An}\in \mc P_k$, any sub-annulus also belongs to $\mc P_k$. Also we know $\mc P_k \cap \ms C\neq \emptyset$ for any $L$-admissible collection $\ms C$ of annuli. By Proposition \ref{prop:general property}, there exists a subsequence of $\{(\Sigma_k, \Omega_k)\}$, such that for any $p\in M$, there exists $r_p>0$, so that for any $0<s<r<r_p$, $\mr{An}(p; s, r)\in \mc P_k$ for sufficiently large $k$. This is exactly Property {\bf(S)}.
\end{proof}

\begin{theorem}\label{thm:convergence of pmc to minimal}
As above, $\spt\|V_\infty\|$ is a closed embedded minimal surface $\Sigma_\infty$. Moreover, there exists a finite set of points $\mc Y\subset M$, such that up to a subsequence, $\{\Sigma_k\}_{k\in\mb N}$ (without relabeling) converges in $C^{1,1}_{loc}$ to $\Sigma_\infty$ in any compact subset of $M\setminus \mc Y$.
\end{theorem}
\begin{proof}
By Proposition \ref{prop:compactness}, $\Sigma_k$ convergences subsequentially in $C^{1, \alpha}_{loc}$ to a $C^{1,1}$ almost embedded surface $\Sigma_\infty\subset M\setminus \mc Y$ in any compact subsets of $M\setminus \mc Y$. Then $\Sigma_\infty$ is smoothly embedded as $V_\infty$ is stationary. Hence the convergence is in $C^{1,1}_{loc}$ by Proposition \ref{prop:compactness}(\rom{1}). To show the removable singularity of $\Sigma_\infty$, for a fixed $p\in \mc Y$, given any $\mr{An}(p; s, r)\subset \mr{An}(p; 0, r_p)$, by Property {\bf(S)} \eqref{eq:property S}, we can use the $\A^k$-stability for the sequences $(\Sigma_k, \Omega_k)$ to deduce that $\Sigma_\infty$ is stable in $\mr{An}(p; s, r)$. Then $\Sigma_\infty$ extends smoothly across $p$ by the standard removable singularity theorem (for instance \cite{SS}).
\end{proof}

\subsection{Passing almost minimizing to limit and topology bound}

Now we will prove that the support $\Sigma_\infty$ of $V_\infty$, which is a closed embedded minimal surface, has total genus less than $\mk g_0$ -- the genus of elements in $\ms E$. To approach it, we will find a diagonal subsequence of elements in $\ms E$ approaching $V_\infty$ which satisfies certain almost minimizing in small annuli property.

\begin{proposition}\label{prop:hk to h} 
There exist a subsequence $\{(\Sigma_{k,j(k)},\Omega_{k,j(k)})\}$ and $\epsilon_k\to 0$, $\delta_k\to 0$, such that $\Sigma_{k, j(k)}$ converges to $V_\infty$ as varifolds, and
\begin{enumerate}
\item for any $L(m)$-admissible collection $\ms C$, $(\Sigma_{k,j(k)},\Omega_{k,j(k)})$ is $(\mc A^k,\epsilon_k,\delta_k)$-almost minimizing in at least one of $\ms C$;

\item given any $p\in M$, there exists $r_p>0$ such that for each $\mr{An}(p;s,r)$ with $0<s<r<r_p$, $(\Sigma_{k,j(k)},\Omega_{k,j(k)})$ is $(\mc A^k,\epsilon_k,\delta_k)$-almost minimizing in $\mr{An}(p;s,r)$ for all $k$ sufficiently large. 
\end{enumerate}
\end{proposition}
\begin{proof}
Note that $(V_k,\Omega_k)$ is $\mc A^k$-almost minimizing w.r.t. $\{(\Sigma_{k,j},\Omega_{k,j})\}_{j\in \mb N}$. Thus there exist $\epsilon_{k,j},\delta_{k,j}\to 0$ as $j\to \infty$ so that $(\Sigma_{k,j},\Omega_{k,j})$ is $(\mc A^k,\epsilon_{k,j},\delta_{k,j})$-almost minimizing in any $L(m)$-admissible collection of annuli. By taking $j(k)$ sufficiently large, we have that for any $L(m)$-admissible collection $\ms C$, $(\Sigma_{k,j(k)},\Omega_{k,j(k)})$ is $(\mc A^k,\epsilon_k,\delta_k)$-almost minimizing in $\ms C$, where $\epsilon_k:=\epsilon_{k,j(k)}$ and $\delta_k:=\delta_{k,j(k)}$.

Now let $\mc P_k$ be the collection of annuli $\mr{An}$ so that $(\Sigma_{k,j(k)},\Omega_{k,j(k)})$ is $(\mc A^k,\epsilon_k,\delta_k)$-almost minimizing in $\mr{An}$. Thus by Proposition \ref{prop:general property}, there exists a subsequence (without relabeling) so that given $p\in M$, there exists $r_p>0$ such that for each $\mr{An}(p;s,r)$ with $0<s<r<r_p$, $(\Sigma_{k,j(k)},\Omega_{k,j(k)})$ is $(\mc A^k,\epsilon_k,\delta_k)$-almost minimizing in $\mr{An}(p;s,r)$ for all sufficiently large $k$. This finishes the proof of Proposition \ref{prop:hk to h}.    
\end{proof}

Note that $V_\infty$ may not be $\mc A^0$-almost minimizing ($\A^0$ is the area functional) in small annuli (which is the requirement in \cite{Ketover13} to obtain the genus bound). However, we can take a special $h$ so that as $\varepsilon_k\to 0$, the min-max solution $(\Sigma_k,\Omega_k)$ w.r.t. $\mc A^{\varepsilon_k h}$-functional will converge to a minimal surface which still has the genus bound. 

\begin{theorem}[Genus bound]\label{thm:genus bound}
Let $(M, g)$ be a closed, oriented, three dimensional Riemannian manifold, and $V_\infty$ be as above. Suppose that there are finitely many pairwise disjoint balls $B_1,\cdots,B_\alpha\subset M$ such that 
\begin{enumerate}
    \item $\spt\|V_\infty\|\cap B_j$ is a disk for $j=1,\cdots,\alpha$;
    \item $h\equiv 0$ in a small neighborhood of $\spt\|V_\infty\|\setminus \cup_{j}B_j$.
\end{enumerate}
Assume that $ V_\infty=\sum_{i=1}^N m_i[\Gamma_i]$,
where $\{\Gamma_i\}_{i=1}^N$ is a pairwise disjoint collection of connected, closed, embedded, minimal surfaces. 
Denote by $I_O \subset \{1,\cdots,N\}$ (resp. $I_U$) the collection of $i$ such that $\Gamma_i$ is orientable (resp. non-orientable). Then we have
\begin{equation}\label{eq:genus bound0}
\sum_{i\in I_O} m_i\cdot \mk g(\Gamma_i)+\frac{1}{2}\sum_{i\in I_U} m_i \cdot (\mk g(\Gamma_i)-1)\leq \mk g_0,
\end{equation}
where $\mk g_0$ and $\mk g(\Gamma_i)$ are the genus of $\Sigma_0$ and $\Gamma_i$, respectively. 
\end{theorem}
\begin{proof}
Let $\{\gamma_i\}_{i=1}^k$ be a collection of simple closed curves contained in $\cup_{i=1}^N \Gamma_i$. Since $\spt \|V_\infty\|\cap B_j$ is a disk, then we can perturb $\{\gamma_i\}_{i=1}^k$ (not relabelled) in the same isotopy class so that $\cup_i\gamma_i$ does not intersect $\cup_jB_j$. Thus we have that $\varepsilon_k h\equiv 0$ in a neighborhood of $\cup_i\gamma_i$. Note that by Proposition \ref{prop:hk to h}, given $p\in M$, there exists $r_p>0$ such that for each $\mr{An}(p;s,r)$ with $0<s<r<r_p$, $(\Sigma_{k,j(k)},\Omega_{k,j(k)})$ is $(\mc A^{\varepsilon_k h},\epsilon_k,\delta_k)$-almost minimizing in $\mr{An}(p;s,r)$ for all sufficiently large $k$, where $\epsilon_k,\delta_k\to 0$ as $k\to\infty$. In particular, by possibly perturbing $\{\gamma_i\}$ and shrinking $r_p$, we can assume that $(\Sigma_{k,j(k)},\Omega_{k,j(k)})$ is $(\mc A^0,\epsilon_k,\delta_k)$-almost minimizing in $B_{r_p}(p)$ for any $p\in\cup_i\gamma_i$. 
Then one can lift the curves by the same argument (see \cite{Ketover13}*{Remark 1.3}) in \cite{Ketover13}*{Proposition 2.2}, which yields the desired genus bounds; see also \cite{DeLellis-Pellandini10}. Hence Theorem \ref{thm:genus bound} is proved.
\end{proof}

\section{Existence of supersolution}\label{sec:existence of supersolution}

In the above section, we proved that a sequence of $\A^{\varepsilon_k h}$-min-max pairs $\{(\Sigma_k, \Omega_k)\}_{k\in\mb N}$ converges to a limit minimal surface $\Sigma_\infty$ when $\varepsilon_k \to 0$ as $k\to \infty$. In this part, assuming $h$ changes sign on $\Sigma_\infty$, we show that if the convergence has multiplicity greater than one, then $\Sigma_\infty$ admits a nonnegative weak supersolution to a variant of the Jacobi equation (see \eqref{eq:full ineq on Sigma}). The idea is to subtract the mean curvature equations satisfied by the top and bottom sheets and take re-normalized limit. We will present all the subtle details as surfaces involved in our setting are merely $C^{1,1}$.

\medskip

We start with some general assumptions. Let $(\Sigma_k, \Omega_k)$ be a sequence of strongly $\mc A^{\varepsilon_k h}$-stationary, $C^{1,1}$ $(\varepsilon_k h)$-boundary. Suppose that $\Sigma_k$ converges as varifolds to a closed, embedded, minimal surface $\Sigma$ with multiplicity $m\in \mb N$ as $k\to \infty$ ($\varepsilon_k\to 0$). Furthermore, we assume that
\begin{enumerate}[label=(\Alph*)]
    \item\label{item:h changes sign}  $h$ changes sign on $\Sigma$, and
    \item the convergence is $C^{1,1}_{loc}$ in any compact subset away from a finite set $\mc Y\subset M$.
\end{enumerate}

Fix a unit normal $\nu$ of $\Sigma$. We will use Fermi coordinates of $\Sigma$ given by the normal exponential map $(x, z)\to \exp_x\big(z\nu(x)\big)$, $(x, z)\in \Sigma\times (-\delta_0, \delta_0)$. By the $C^{1, 1}_{loc}$-convergence, for any open subset $\mc U\subset\subset \Sigma\setminus \mc Y$, and for all $k$ large enough (depending only on $\mc U$), $\Sigma_k$ has an $m$-sheeted ordered decomposition $\Gamma_k^1\leq \cdots \leq \Gamma_k^m$ inside the thickened neighborhood $\mc U_\delta = \mc U\times (-\delta, \delta)$, and each sheet $\Gamma_k^\iota$ ($1\leq \iota \leq m$) is a normal graph of some $u_k^\iota\in C^{1,1}(\mc U)$, such that
\[ u_k^1\leq \cdots \leq u_k^m, \text{ and } u_k^\iota \to 0 \text{ in $C^{1,1}(\mc U)$ as $k\to\infty$}.\]
By choosing $\mc U$ large enough, we may also assume that $h$ changes sign in $\mc U$. 

We let $H_k^\iota$ denote the generalized mean curvature of $\Gamma_k^\iota$ w.r.t. the upward pointing normal (in the same direction as $\partial/\partial z$). By Corollary \ref{cor:mean curvature charaterization}, we have
\begin{equation}\label{eq:slice mean curvature bound}
|H_k^\iota| \leq \varepsilon_k |h|,\quad \text{$\mc H^2$-a.e. in $\Gamma_k^\iota$, for each $1\leq \iota\leq m$}.
\end{equation}

Denote by $\nu_k$ the unit outer normal of $\Sigma_k$ induced from $\Omega_k$. By Lemma \ref{lem:choice of outer normal}, $\nu_k$ alternates orientation among $\{\Gamma^\iota_k\}$. After possibly switching $\Omega_k$ to $M\setminus\Omega_k$ and $h$ to $-h$ simultaneously and up to a subsequence, we may always assume that $\nu_k$ points upward along the top sheet $\Gamma^m_k$. That is,
\begin{equation*}
\text{$\Omega_k$ does not contain the region above the top sheet $\Gamma_k^m$.}
\end{equation*} 
Then we have by strong $\A^{\varepsilon_k h}$-stationarity that 
\begin{equation}\label{eq:top}
    H_k^m \geq  \varepsilon_kh, \quad \text{$\mc H^2$-a.e. in $\Gamma_k^m$}.
\end{equation} 
In fact, by Proposition \ref{prop:mean curvature compared with pmc function}, we know  that in a neighborhood where $h>0$, we must have $H_k^m =  \varepsilon_kh$, and the strictly inequality $H_k^m> \varepsilon_kh$ can only happen in a neighborhood where $h<0$. Moreover, by Corollary \ref{cor:mean curvature charaterization}, we know that $H_k^{m-1}$ is either equal to 0 or $-\varepsilon_k h$ for all points in $\{h>0\}$. In particular,
\begin{equation}\label{eq:Hm-1}
H_k^{m-1}\leq 0 \quad \text{ for all points in $\{h>0\}$}.
\end{equation}
Applying similar arguments to $\Gamma_k^1$, we have for $\mc H^2$-a.e. in $\Gamma_k^1$,
\begin{equation}\label{eq:bottom}
H_k^1\leq  \varepsilon_kh \quad \text{ if $m$ is odd}; \quad H_k^1\leq - \varepsilon_kh \quad \text{ if $m$ is even}.
\end{equation}

Let $f$ be a $C^{1,1}$ test function defined on $\Sigma$ with compact support in $\mc U$. Denote by 
\[ \varphi_k := u_k^m-u_k^1 \geq 0. \] 
The following are derived in Appendix \ref{appen:jacobi calculations}, in particular \eqref{eq:uniform estimates of coefficients}, \eqref{eq:estimates of sigmapm}, and \eqref{eq:height diff 2 order pde}.
\begin{lemma}\label{lem:small Wj}
There exist Lipschitz 2-tensors ${\bm\alpha}_k$, vector fields ${\bm\beta}_k$, and functions $\zeta_k$, $\sigma^1_k$ and $\sigma^m_k$ defined on $\mc U$, so that 
\begin{align}\label{eq:integral form}
\int_{\mc U}\lb{\nabla \varphi_k, \nabla f} & - \big(\Ric(\nu, \nu)+|A^\Sigma|^2\big) \varphi_k\cdot f\, \mr d\mc H^2(x) \\ \nonumber
& = \int_{\mc U} \bm\alpha_k(\nabla \varphi_k, \nabla f) + \bm\beta_k\cdot(\varphi_k\nabla f + f\nabla \varphi_k) + \zeta_k\varphi_k f\, \mr d\mc H^2(x)\\ \nonumber
& + \int_{\mc U} \big(H_{k}^m\cdot \sigma^m_k - H_{k}^1\cdot \sigma_k^1 \big)\cdot f\, \mr d\mc H^2(x),
\end{align}
where 
\begin{equation}\label{eq:uniform estimates for W_k}
\|\bm\alpha_k\|_{C^{0,1}(\mc U)}, \|\bm\beta_k\|_{C^{0,1}(\mc U)}, \|\zeta_k\|_{C^{0,1}(\mc U)}, \|\sigma^m_k-1\|_{C^{0,1}(\mc U)}, \|\sigma^1_k-1\|_{C^{0,1}(\mc U)} \to 0 \text{ as } k\to\infty.
\end{equation} 
Similarly, we also have for $v_k=u_k^m-u_k^{m-1}\geq 0$ (note that $v_k\leq \varphi_k$),
\begin{align}
\int_{\mc U}\langle\nabla v_k,\nabla f\rangle & -\big(\Ric(\nu,\nu)+|A^{\Sigma}|^2\big)v_k\cdot f\,\mr d\mc H^2(x) \label{eq:integral form for each iota}\\ 
& =\int_{\mc U} \oli{\bm\alpha}_k(\nabla v_k,\nabla f)+ \oli{\bm\beta}_k\cdot(v_k\nabla f + f\nabla v_k)+\oli{\zeta}_kv_kf\, \mr d\mc H^2(x) \nonumber\\
& +\int_{\mc U} \big(H_{k}^m\cdot \sigma_k^m - H_{k}^{m-1}\cdot \sigma_k^{m-1} \big)\cdot f\, \mr d\mc H^2(x),\nonumber
\end{align}
where 
\[ 
\|\oli{\bm\alpha}_k\|_{C^{0,1}(\mc U)}, \|\oli{\bm\beta}_k\|_{C^{0,1}(\mc U)}, \|\oli{\zeta}_k\|_{C^{0,1}(\mc U)}, \|\sigma^m_k-1\|_{C^{0,1}(\mc U)}, \|\sigma^{m-1}_k-1\|_{C^{0,1}(\mc U)} \to 0 \text{ as } k\to\infty. 
\] 
\end{lemma}

We claim that the $L^2$-magnitude of the height differences should always dominate the size of parameters. We will use the height difference $v_k$ of the top two sheets. 
\begin{lemma}\label{lem:epsilon j upper bound}
There exists $c>0$ such that for all sufficiently large $k$,
\begin{equation}\label{eq:vj controls varepsilonj}
\varepsilon_k \leq c\cdot \|v_k\|_{L^2(\mc U)},\quad \text{and hence}\quad \varepsilon_k \leq c\cdot \|\varphi_k\|_{L^2(\mc U)}.   \end{equation} 
\end{lemma}
\begin{proof}
Suppose on the contrary that up to a subsequence, as $k\to \infty$,
\begin{equation}\label{eq:assume epsilon large}
(\|v_k\|_{L^2(\mc U)})/\varepsilon_k \to 0 .
\end{equation}
Then for any $\eta\in C^2_c(\mc U)$, by letting $f=\eta^2 v_k$ in \eqref{eq:integral form for each iota}, we obtain
\begin{align*}  
\int_{\mc U}|\nabla v_k|^2\eta^2\,\mr d\mc H^2\leq C\int_{\mc U}|\nabla \eta|^2|v_k|^2+|v_k|^2\eta^2&+\varepsilon_k \eta^2v_k \,\mr d\mc H^2\\
&+\frac{1}{10}\int_{\mc U} (|v_k||\nabla (\eta^2v_k)|+\eta^2|v_k||\nabla v_k|)\,\mr d\mc H^2.
\end{align*}
Simplifying it, we obtain
\[
\int_{\mc U}|\nabla v_k|^2\eta^2\,\mr d\mc H^2\leq C\int_{\mc U}|\nabla \eta|^2|v_k|^2+|v_k|^2\eta^2+\varepsilon_k v_k \,\mr d\mc H^2.
\]
Thus given $\mc U'\subset \subset \mc U$,
\[
\int_{\mc U'}|\nabla v_k/\varepsilon_k|^2\,\mr d\mc H^2\leq C(\mc U,\mc U',M,\sup|h|)\int_{\mc U}|v_k/\varepsilon_k|^2+1\,\mr d\mc H^2.
\]
It follows that up to a subsequence, $v_k/\varepsilon_k$ weakly converges to some $w$ weakly in $W^{1,2}(\mc U')$. Together with \eqref{eq:assume epsilon large}, we have that $w=0$.

Observe that by \eqref{eq:top} and \eqref{eq:Hm-1}, if $h(x,0)>0$, then for sufficiently large $k$,
\[H_k^m(x, u_k^m(x))-H_k^{m-1}(x, u_k^{m-1}(x))\geq  \varepsilon_kh(x, u_k^m(x)). \] 
Letting $k\to\infty$ in \eqref{eq:integral form for each iota}, we then get for any $f\in C^2_c(\{h>0\}\cap \mc U')$ with $f\geq 0$ and $f>0$ somewhere,
\[
\int_{\mc U'}hf\leq 0.
\]
Note that we can choose $\mc U'$ large so that $\{h>0\}\cap \mc U'\neq \emptyset$. This leads to a contradiction. Hence Lemma \ref{lem:epsilon j upper bound} is proved.
\end{proof}

Up to a subsequence without relabeling, we can assume that there exists $c\in [0,\infty)$, 
\begin{equation}\label{eq:ratio of varphi and varepsilon}
\lim_{k\to\infty} \varepsilon_k/ \|\varphi_k\|_{L^2(\mc U)} = c.
\end{equation}
Let $\wti \varphi_k=\varphi_k/\|\varphi_k\|_{L^2(\mc U)}$; (note that $\|\varphi_k\|_{L^2(\mc U)}>0$). We consider the limit of $\wti \varphi_k$ on $\mc U$ as $k\to \infty$. By taking an exhaustion $\{\mc U_k\}$ of $\Sigma\setminus \mc Y$, a diagonal argument will give a limit function defined on $\Sigma\setminus \mc Y$. We will also prove that such a limit is uniformly bounded (and hence non-trivial) and is a supersolution in the following sense.

\begin{lemma}\label{lem:solution away Y}
Up to subsequence, $\wti\varphi_k$ converges to a uniformly bounded $C^{1,\alpha}_c$-function $\varphi:\Sigma\setminus \mc Y\to [0, \infty)$ with $\|\varphi\|_{L^2(\Sigma\setminus \mc Y)}=1$. Furthermore, for any $f\in C^1_c(\Sigma\setminus \mc Y)$ and $f\geq 0$, we have
\begin{enumerate}
    \item if $m\geq 3$ is odd, then 
    \begin{equation}\label{eq:odd away Y}
    \int_{\Sigma\setminus \mc Y}\langle\nabla\varphi,\nabla f\rangle- \big(\Ric(\nu,\nu)+|A^{\Sigma}|^2\big)\varphi f\,\mr d\mc H^2 \geq0;
    \end{equation}
    \item if $m$ is even, then
    \begin{equation}\label{eq:even away Y}
    \int_{\Sigma\setminus \mc Y}\langle\nabla\varphi,\nabla f\rangle- \big(\Ric(\nu,\nu)+|A^{\Sigma}|^2\big)\varphi f\,\mr d\mc H^2 
     \geq 2c\int_{\Sigma\setminus \mc Y}hf\,\mr d\mc H^2.
    \end{equation}
\end{enumerate}
\end{lemma}
\begin{proof}
Recall that by \eqref{eq:slice mean curvature bound}, we always have
\begin{equation}\label{eq:H difference upper bound}
    |H_k^m| + |H_k^1| \leq \varepsilon_k\big(|h(x, u_k^m(x))| + |h(x, u_k^1(x))|\big).
\end{equation}
Therefore, if we renormalize the weak equation \eqref{eq:integral form} by $\|\varphi_k\|_{L^2(\mc U)}$, the renormalized terms of $(H_k^m\cdot \sigma_k^m -H_k^{1}\cdot\sigma_k^1)$  will have uniform $L^\infty$ upper bound by \eqref{eq:uniform estimates for W_k} and \eqref{eq:vj controls varepsilonj}. 
Then by applying the interior H\"older estimates \cite{GT}*{Theorem 8.24} to the renormalized weak equation of \eqref{eq:integral form}, we have for any open domain $\mc U'\subset\subset \mc U$, 
\[  
\|\wti \varphi_k\|_{C^{\alpha}(\overline{\mc U'})}\leq C(\|\wti\varphi_k\|_{L^2(\mc U)}+1)\leq C,
\]
where $C$ is a constant independent of $k$. Applying the $C^{1,\alpha}$-estimates \cite{GT}*{Theorem 8.32}, we know that a subsequence of $\wti\varphi_k$ converges to a nonnegative function $\varphi\in C^{1,\alpha}(\mc U)$ in the sense of $C^{1,\alpha}_{loc}(\mc U)$. Note that above argument works for any $\mc U\subset\subset\Sigma\setminus\mc Y$. Taking an exhaustion of $\Sigma\setminus \mc Y$, we can extend $\varphi$ to $\Sigma\setminus \mc Y$. Next we will show that $\varphi$ is a supersolution.

\medskip
{\em Now we first consider the case when $m\geq 3$ is odd.} By \eqref{eq:top} and \eqref{eq:bottom}, we know that for $\mc H^n$-a.e. $x\in U$, 
\begin{equation}\label{eq:odd:low bound for mean curvature difference}
    H_k^m-H_k^1\geq \varepsilon_k\big(h(x, u_k^m(x)) - h(x, u_k^1(x))\big) \geq  -C\varepsilon_k\varphi_k,
\end{equation}
for some constant $C$ independent of $k$. Plugging \eqref{eq:odd:low bound for mean curvature difference} and \eqref{eq:uniform estimates for W_k} into \eqref{eq:integral form}, and then taking the limit, we have 
\begin{equation*}
\int_{\mc U}\langle\nabla\varphi,\nabla f\rangle- \big(\Ric(\nu,\nu)+|A^{\Sigma}|^2\big)\varphi f\,\mr d\mc H^n\geq 0, \quad \forall f\in C^1_c(U) \text{ and } f\geq 0.
\end{equation*}
Taking an exhaustion of $\Sigma\setminus \mc Y$, we can extend $\varphi$ to $\Sigma\setminus \mc Y$ satisfying \eqref{eq:odd away Y}.

\medskip
{\em Next we consider the case when $m\geq 2$ is even.} Recall that by \eqref{eq:top} and \eqref{eq:bottom},
\[ H_k^m - H_k^1 \geq \varepsilon_k \big( h(x,u_k^m(x)) + h(x,u_k^1(x))\big).  \]
Plugging it into \eqref{eq:integral form} and using \eqref{eq:uniform estimates for W_k}, and then taking the limit, we have 
\begin{equation}\label{eq:<2ch on U}
     \int_{\mc U}\langle\nabla\varphi,\nabla f\rangle- \big(\Ric(\nu,\nu)+|A^{\Sigma}|^2\big)\varphi f\,\mr d\mc H^n\geq  2c\int_{\mc U} hf\,\mr d\mc H^n , \quad \forall f\in C^1_c(U) \text{ and } f\geq 0. 
     \end{equation}
Using an exhaustion of $\Sigma\setminus \mc Y$ again, we can extend $\varphi$ to $\Sigma\setminus \mc Y$ satisfying \eqref{eq:even away Y}.

\medskip
It remains to show that $\wti\varphi_k$ is pointwisely bounded independent of $k$. We sketch the proof when $m$ is even and the other case is similar. Taking $r$ small enough so that the constant mean curvature foliation \cite{Zhou19}*{Proposition D.1} exists in a neighborhood of $B_{3r}(y)\cap \Sigma$ for each $y\in \mc Y$. Then by the argument above for $\mc U=\Sigma\setminus \cup_{y\in\mc Y} \overline B_{r}(y)$, $\mc U'=\Sigma\setminus \cup_{y\in \mc Y}\overline B_{2r}(y)$, there exists a function $\varphi\in C^{1,\alpha}(\mc U')$ satisfying \eqref{eq:<2ch on U}. Note that by a standard argument \cite{Zhou19}*{Page 802, Part 7}, one can prove that for any $x\in \mr{dom}(\wti\varphi_k)\cap B_{3r}(y)\setminus \{y\}$,
\[ \wti \varphi_k(x)\leq C(\sup_{\partial B_{3r}(y)}\wti \varphi_k + 1). \]
Thus, $\wti\varphi_k$ is uniformly bounded independent of $k$. This implies that the $L^2$-norm of $\wti\varphi_k$ cannot concentrate near $\mc Y$, and hence we must have $\|\varphi\|_{L^2(\Sigma\setminus \mc Y)}=1$. 
This completes the proof of Lemma \ref{lem:solution away Y}.
\end{proof}

In the next, we prove that $\varphi\in W^{1,2}(\Sigma\setminus \mc Y)$, which implies that $\varphi$ can be extended across $\mc Y$.
\begin{proposition}\label{prop:existence of supsolution}
Let $\{(\Sigma_k, \Omega_k\}_{k\in\mb N}$ be a sequence of strongly $\mc A^{ \varepsilon_kh}$-stationary, $C^{1, 1}$ $ \varepsilon_k h$-boundary in $(M,g)$ with $\varepsilon_k\to 0$ as $k\to 0$. Suppose that $\Sigma_k$ converges as varifolds to a closed, embedded, two-sided, minimal surface $\Sigma$ with multiplicity $m\geq 2$. Suppose in addition that the convergence is $C^{1, 1}_{loc}$ away from a finite set $\mc Y$. Then $\Sigma$ admits a nonnegative function $\varphi\in W^{1,2}(\Sigma)$ with $\|\varphi\|_{L^2(\Sigma)}=1$ and a constant $c\geq 0$ satisfying 
\begin{equation}\label{eq:full ineq on Sigma} 	\int_{\Sigma}\langle\nabla\varphi,\nabla f\rangle - \big(\Ric(\nu,\nu) + |A^{\Sigma}|^2\big)\varphi f\,\mr d\mc H^2 \geq \int_{\Sigma}2c hf\,\mr d\mc H^2, \ \forall f\in C^1(\Sigma) \text{ and } f\geq 0.
\end{equation}
Here $c=0$ if $m\geq 3$ is odd.
\end{proposition}
\begin{proof}
We will first show that there is a constant $C>0$ such that 
\begin{equation}\label{eq:gradient L2 bound}
\int_{\Sigma\setminus \mc Y}|\nabla \varphi|^2\,\mr d\mc H^2\leq C\int_{\Sigma\setminus \mc Y}\varphi^2+1\,\mr d\mc H^2.  
\end{equation}
Together with the fact that $\varphi$ is uniformly bounded, we conclude that $\varphi$ could be extended to be a function $\varphi\in W^{1,2}(\Sigma)$. Hence \eqref{eq:full ineq on Sigma} can be derived from \eqref{eq:odd away Y} and \eqref{eq:even away Y} using a standard log-cutoff trick.

In general, the supersolution inequalities \eqref{eq:odd away Y} and \eqref{eq:even away Y} are not enough to derive the bound \eqref{eq:gradient L2 bound}. Instead, we will use the identity \eqref{eq:integral form} together with the mean curvature estimates \eqref{eq:H difference upper bound} and the comparison estimates \eqref{eq:vj controls varepsilonj} to prove the desired bound \eqref{eq:gradient L2 bound}. In fact, let $\{\eta_r\}_{r>0}$ be a family of log-cut-off functions so that $\eta_r=0$ in $\cup_{y\in \mc Y}B_r(y)$, and as $r\to 0$,
\[
0\leq \eta_r\leq 1, \quad \eta_r\to 1 \text{ on } \Sigma\setminus\mc Y, \quad \text{and} \quad \int_\Sigma|\nabla \eta_r|^2\to 0. 
\]
Then by taking $f=\eta_r^2\varphi_k$ in \eqref{eq:integral form} and then applying \eqref{eq:H difference upper bound} and \eqref{eq:vj controls varepsilonj}, we obtain that for all sufficiently large $k$,
\begin{align*}
\int_{\Sigma\setminus \mc Y}|\nabla  \varphi_k|^2\cdot \eta_r^2\,\mr d\mc H^2&\leq C\int_{\mc U}|\nabla \eta_r|^2|\varphi_k|^2+|\varphi_k|^2\eta_r^2+\varepsilon_k\varphi_k\cdot \eta_r^2\,\mr d\mc H^2+\\
&\hspace{10em}   +\frac{1}{10}\int_{\mc U} \eta_r^2\varphi_k|\nabla \varphi_k|+\varphi_k|\nabla (\eta_r^2 \varphi_k)|\,\mr d\mc H^2\\
&\leq C\int_{\mc U}|\nabla \eta_r|^2|\varphi_k|^2+|\varphi_k|^2\eta_r^2+\varepsilon_k\varphi_k\,\mr d\mc H^2+\frac{1}{2}\int_{\mc U}|\nabla\varphi_k|^2\cdot\eta_r^2\,\mr d\mc H^2.
\end{align*}
Simplifying it and using \eqref{eq:ratio of varphi and varepsilon}, we obtain
\[  \int_{\Sigma\setminus \mc Y}|\nabla \wti \varphi_k|^2\cdot \eta_r^2\,\mr d\mc H^2\leq C\int_{\Sigma\setminus \mc Y}\wti \varphi_k^2(\eta_r^2+|\nabla \eta_r|^2)+(c+1)\wti \varphi_k\,\mr d\mc H^2.     \]
Taking $k\to \infty$, it follows that
\[  \int_{\Sigma\setminus \mc Y} |\nabla \varphi|^2\cdot \eta_r^2\,\mr d\mc H^2\leq C\int_{\Sigma\setminus \mc Y}\varphi^2(\eta_r^2+|\nabla \eta_r|^2)+1\,\mr d\mc H^2.
\]
Recall that $\varphi$ is uniformly bounded. Hence, as $r\to0$, \eqref{eq:gradient L2 bound} follows immediately. This completes the proof of Proposition \ref{prop:existence of supsolution}.
\end{proof}
\begin{remark}
Let $\Sigma_k$ be the same as in Proposition \ref{prop:existence of supsolution}. Suppose that the limit surface $\Sigma$ is one-sided. Then by the same argument, the connected double cover $\wti \Sigma$ of $\Sigma$ admits a non-negative function $\varphi\in W^{1,2}(\wti\Sigma)$ satisfying \eqref{eq:full ineq on Sigma}.
\end{remark}

\section{Multiplicity one for Simon-Smith min-max theory}\label{sec:multiplicity one}

In this section, we will prove two multiplicity one theorems in the Simon-Smith setting, that is, for relative min-max in the space of oriented separating surfaces in Section \ref{ss:multiplicity one theorem for relative min-max}, and for the classical min-max in the space of un-oriented surfaces in Section \ref{SS:multiplicity one for Simon-Smith}.

\subsection{Multiplicity one for relative Simon-Smith min-max}\label{ss:multiplicity one theorem for relative min-max}
In this part, we will show how to choose the correct prescribing function $h$ so as prove the first multiplicity one type result. Recall that the space $\ms E$ of embedded separating surfaces of genus $\mk g_0$ is defined in \eqref{eq:ms E}.

We have the following compactness for minimal surfaces with bounded area and satisfying Property ${\bf(R')}$. 
\begin{theorem}\label{thm:compactness}
Let $L$ be a positive integer and $\Lambda>0$ be a constant. Let $\Sigma_k$ be a sequence of closed, embedded, minimal surfaces satisfying
\begin{itemize}
    \item $\mc H^2 (\Sigma_k)\leq \Lambda$, and
    \item Property {\bf(R')} \eqref{eq:property R2} (with $\Sigma_k$ in place of $V_\infty$ and $L$ given in the assumption).
\end{itemize}
Then $\Sigma_k$ converges subsequentially to a closed, embedded, minimal surface $\Sigma$ possibly counted with integer multiplicity in the sense of varifolds. Furthermore, if $\Sigma_k\neq \Sigma$ for infinitely many $k$, then $\Sigma$ is degenerate.
\end{theorem}
\begin{proof}
The proof of convergence is the same as that of Theorem \ref{thm:convergence of pmc to minimal}. Indeed, we know that away from a finite set of points $\mc Y$, the convergence $\Sigma_k\to \Sigma_\infty$ is locally smooth. If $\Sigma_k\neq \Sigma$ for infinitely many $k$, then one can construct a nontrivial Jacobi field along $\Sigma_\infty$ in the same way as \cite{Sharp17}, and hence $\Sigma_\infty$ is degenerate.
\end{proof}

\begin{theorem}[Multiplicity one for relative min-max]\label{thm:multiplicity one for relative min-max}
Let $(M,g)$ be a three dimensional closed Riemannian manifold.  Let $X\subset I(m,k_0)$ be a cubical complex and $Z\subset X$ be a subcomplex. Let $\Phi_0:X\to \ms E$ be a continuous map and $\Pi$ be the $(X,Z)$-homotopy class of $\Phi_0$. Assume that
\[ \mf L(\Pi) > \max_{x\in Z} \mc H^2\big(\Phi_0(x)\big).\]
Then there exists a pairwise disjoint collection of connected, closed, smoothly embedded, minimal surfaces $\Gamma=\cup_{i=1}^N\Gamma_i$ and positive integers $\{m_i\}_{i=1}^N$, so that
\[\mf L(\Pi)=\sum_{i=1}^Nm_i\mc H^2 (\Gamma_i).\]
and 
\begin{enumerate}
     \item if $\Gamma_i$ is two-sided and unstable, then $m_i=1$;
     \item if $\Gamma_i$ is one-sided, then the connected double cover of $\Gamma_i$ is stable.
\end{enumerate}
Furthermore, if $M$ is orientable, then
\begin{equation}\label{eq:genus bound1}
     \sum_{i\in I_O}m_i\mk g(\Gamma_i)+\frac{1}{2}\sum_{i\in I_U}m_i(\mk g(\Gamma_i)-1)\leq \mk g_0,
\end{equation}
where $\mk g_0$ is the genus associated with $\ms E$, and $I_O$ (resp. $I_U$) is the collection of $i$ such that $\Gamma_i$ is orientable (resp. non-orientable).
\end{theorem}
\begin{proof}
Suppose that $(M,g)$ is bumpy. Then for a given constant $\Lambda$ (e.g. $\Lambda:=\mf L(\Pi)+1$), let $\mc M(\Lambda)$ be the collection of closed embedded minimal surfaces $\Gamma$ satisfying $\mc H^2 (\Gamma)\leq \Lambda$ and Property {\bf(R')} \eqref{eq:property R2} for $L=L(m)$. Note that by Theorem \ref{thm:compactness}, $\mc M(\Lambda)$ is a finite set since $g$ is bumpy. Denote by $\{S_1,\cdots, S_\alpha\}$ the collection of those embedded minimal surfaces. Then we take $p_1,\cdots,p_\alpha$ and $q_1,\cdots, q_\alpha$ in $M$ so that $p_i,q_i\in  S_j$ if and only if $j=i$. Let $r>0$ be a small number so that
\begin{itemize} 
\item $B_r(p_1),\cdots,B_r(p_\alpha)$, $B_r(q_1),\cdots,B_r(q_\alpha)$ are pairwise disjoint;
\item $B_r(p_i)\cup B_r(q_i)$ intersects $ S_j$ if and only if $j=i$;
\item $B_r(p_i)\cap  S_i$ and $B_r(q_i)\cap  S_i$ are both embedded disks for all $i=1,\cdots, \alpha$.
\end{itemize}
Next we take a smooth function $h:M\to \mb [-1,1]$ satisfying that for all $i=1,\cdots ,\alpha$,
\begin{enumerate}
\item $h=0$ outside $\cup_i(B_r(p_i)\cup B_r(q_i))$;
\item $h> 0$ in $B_{r/2}(p_i)\cap  S_i$ and $h<0$ in $B_{r/2}(q_i)\cap  S_i$;
\item if $ S_i$ is two-sided, then $\int_{ S_i}h\phi_i\,\mr d\mc H^2=0$, where $\phi_i$ is the first eigenfunction of the Jacobi operator on $ S_i$;
\item if $ S_i$ is one-sided, then $\int_{\wti  S_i}h\phi_i\,\mr d\mc H^2=0$, where $\phi_i$ is the first eigenfunction of the Jacobi operator on $\wti  S_i$ and $\wti S_i$ is the connected double cover of $ S_i$.
	\end{enumerate} 
 
Now we choose $\varepsilon_k\to 0$. Then for sufficiently large $k$, we have 
\[     
\mf L^{\varepsilon_k h}(\Pi)>\max\left\{\max_{x\in Z}\mc A^{\varepsilon_kh}\big(\Phi_0(x)\big),0\right\}.
\] 
Applying Theorem \ref{thm:pmc min-max theorem} to the $\A^{\varepsilon_k h}$-functional for each $k$, we obtain a min-max pair $(V_k, \Omega_k)\in \VC(M)$, which is a strongly $\A^k$-stationary, $C^{1,1}$ $(\varepsilon_k h)$-boundary with $\A^{ \varepsilon_kh}(V_k, \Omega_k) = \mf L^{\varepsilon_k h}(\Pi)$. By Theorem \ref{thm:convergence of pmc to minimal}, up to subsequence, $V_k$ (not relabelled) converges as varifolds to $V_\infty$ with $V_\infty=\sum_{i=1}^N m_i[\Gamma_i]$, where $\{\Gamma_i\}$ is pairwise disjoint collection of connected, closed, embedded, minimal surfaces. Since $ \varepsilon_k\to0$, we have
\[ \lim_{k\to\infty}\mf L^{ \varepsilon_kh}(\Pi) = \mf L(\Pi), 
\]
which yields that $\mc H^2 (\Gamma_i)\leq \Lambda$. Furthermore, by the discussion before Proposition \ref{prop:annulus picking for Ak almost minimizing}, we know that $\Gamma_\infty = \cup \Gamma_i$ satisfies Property {\bf(R')} \eqref{eq:property R2}. Hence $\Gamma_i$ is one of $S_1,\cdots, S_\alpha$. By relabelling, we assume that $\Gamma_i=S_i$ for $i=1,\cdots,N$.  Observe that by the construction of $h$, the sign of $h$ changes on $\Gamma_i$. Since the convergence is locally $C^{1,1}$ away from finitely many points, and $(V_k,\Omega_k)$ is strongly $\mc A^{ \varepsilon_k h}$-stationary, then by Proposition \ref{prop:existence of supsolution}, any two-sided connected component $\Gamma_i$ with multiplicity $m_i\geq 2$ admits a nontrivial and nonnegative $\varphi\in W^{1,2}(\Gamma_i)$ such that 
\[
\int_{\Gamma_i}\langle\nabla\varphi,\nabla f\rangle- \big(\Ric(\nu,\nu) +|A^{\Gamma_i}|^2\big)\varphi f\,\mr d\mc H^n\geq  \int_{\Gamma_i}2c hf\,\mr d\mc H^n , \ \forall f\in C^1(\Gamma_i) \text{ and } f\geq 0,
\]
for some constant $c\geq 0$. Let $\phi_i$ be the first eigenfunction of the Jacobi operator of $\Gamma_i$. By the choice of $h$,
\begin{align*}
	0=\int_{\Gamma_i} 2ch\phi_i\,\mr d\mc H^2 \leq \int_{ \Gamma_i}\varphi L_{\Gamma_i}\phi_i\,\mr d\mc H^2 = \lambda_1(\Gamma_i) \int_{\Gamma_i}\varphi\phi_i\,\mr d\mc H^2.
	\end{align*}
Recall that $\phi_i>0$ everywhere and $\varphi\geq 0$ with $\varphi>0$ somewhere. It follows that $\int_{\Gamma_i} \varphi \phi_i\,\mr d\mc H^2>0$. Thus we conclude that the first eigenvalue $\lambda_1(\Gamma_i) \geq 0$, that is, if $\Gamma_i$ is two-sided and $m_i\geq 2$, then $\Gamma_i$ is stable. This proves the first item.

For one-sided connected component $\Gamma'\subset \spt\|V_\infty\|$, the same argument gives that the double cover of $\Gamma'$ is stable.

Note that by the choice of $h$, for each $\Gamma_i$, we have that  $h=0$ outside two disjoint balls $B_r(p_i)$ and  $B_r(q_i)$. Moreover, $B_r(p_i)\cap \Gamma_i$ and $B_r(q_i)\cap \Gamma_i$ are both disks. Hence the desired genus bound follows from Theorem \ref{thm:genus bound}.   

\medskip

For the general case when $g$ is not bumpy, one can take a sequence of bumpy metrics $g_i$ converging to $g$ in the sense of $C^3$. Then the theorem follows from the conclusion for bumpy metrics, Property {\bf(R')}, and standard compactness theorem of closed embedded minimal surfaces; see Theorem \ref{thm:compactness} or \cite{Sharp17}.
\end{proof}

\subsection{Multiplicity one for classical Simon-Smith min-max theory}\label{SS:multiplicity one for Simon-Smith}

In this subsection, we use Theorem \ref{thm:multiplicity one for relative min-max} together with the double cover lifting argument in \cite{Zhou19} to prove a multiplicity one type theorem for the classical Simon-Smith min-max theory \cites{Smith82, Colding-DeLellis03}. Here we will use the version of Simon-Smith theory for un-oriented smoothly embedded surfaces. 

Let $(M,g)$ be a three dimensional closed manifold and $\Sigma_0$ be a connected closed surface of genus $\mk g_0$. Then denote 
\[ \ms X(\Sigma_0):=\{\phi(\Sigma_0) \big| \phi : \Sigma_0 \to M \text{ is a smooth separating embedding}\},\]
and 
\[  
\ms Y(\Sigma_0):=\{\phi(\Sigma_0)\big|\phi:\Sigma\to M \text{ is a smooth map whose image is a 0-or 1-dimensional graph}\}.
\]
For simplicity, we denote 
\[ \overline {\ms X}(\Sigma_0) = \ms X(\Sigma_0)\cup \ms Y(\Sigma_0).\] 
We endow $\oli{\ms X}(\Sigma_0)$ with the un-oriented smooth topology for immersions. We sometime simply write $\ms X, \ms Y, \oli{\ms X}$ when there is no ambiguity.

Let $X\subset I(m, k_0)$ be a cubical complex, and $Z_0\subset X$ a subcomplex. 
Fix a continuous map: 
\[\Phi_0: X\to \oli{\ms X}(\Sigma_0), \text{ such that } \Phi_0(Z_0)\subset \ms Y(\Sigma_0). \]
We let $\Pi$ be the set of all continuous maps $\Phi: X\to \oli{\ms X}(\Sigma_0)$ which is homotopic to $\Phi_0$ relative to $\Phi_0|_{Z_0}$. We call such $\Phi$ {\em an $(X, Z_0)$-sweepout by $\Sigma_0$}, or simply {\em a sweepout}. Note that we always have
\[ \mc H^2\big(\Phi(x)\big) =0, \quad \text{for any } x\in Z_0. \]
Define
\[ \mf L(\Pi):=\inf_{\Phi\in \Pi}\sup_{x\in X}\mc H^2(\Phi(x)).\]

Note that $\overline {\ms X}(\Sigma_0)$ or ${\ms X}(\Sigma_0)$ endowed with the oriented smooth topology (see \eqref{eq:ms E}), denoted as $\wti{\overline {\ms X}}(\Sigma_0)$ or $\wti{\ms X}(\Sigma_0)$, forms a nontrivial double cover over $\overline {\ms X}(\Sigma_0)$ or ${\ms X}(\Sigma_0)$. Denote $\overline{\lambda}\in H^1\big(\overline{\ms X}, \mb Z_2\big)$ as the dual to the nontrivial element of $\pi_1(\overline{\ms X})$ coming from the projection $\bm\pi: \wti{\overline {\ms X}} \to \overline {\ms X}$. Note that given any $\phi: S^1 = [0, 1]/\{0\sim 1\} \to \overline {\ms X}$, $\overline{\lambda}[\phi]\neq 0$ if and only if the lifting of $\wti\phi: [0, 1]\to \wti{\overline {\ms X}}$ satisfies that $\wti\phi(1)$ is $\wti\phi(0)$ with the opposite orientation. In this case $\phi$ forms a sweepout of $M$ in the sense of Almgren-Pitts \cite{MN17}*{Definition 3.4}.

\begin{theorem}[Theorem \ref{thm:a rough multiplicity one theorem}]
\label{thm:classical multiplicity one}
Let $(M,g)$ be a closed three dimensional Riemannian manifold. Suppose that $\Pi$ is a homotopy class of $(X, Z_0)$-sweepouts by $\Sigma_0$ with 
\[ \mf L(\Pi)>\sup_{x\in Z_0}\mc H^2\big(\Phi_0(x)\big)=0 .\]
Then there exist a pairwise disjoint collection of connected, closed, embedded, minimal surfaces $\{\Gamma_j\}_{j=1}^N$ and positive integers $m_j$ so that 
\[  \mf L(\Pi)= \sum_{j=1}^N m_j\mc H^2 (\Gamma_j) ,\]	
and
\begin{enumerate}
    \item if $\Gamma_j$ is unstable and two-sided, then $m_j=1$;
    \item if $\Gamma_j$ is one-sided, then the connected double cover of $\Gamma_i$ is stable.
\end{enumerate}
Furthermore, if $M$ is orientable, then the genus bound \eqref{eq:genus bound1} holds with $\mk g_0 = \mk g(\Sigma_0)$.
\end{theorem}

\begin{proof}
The proof will follow in the same structure as the proof of \cite{Zhou19}*{Theorem 5.2}. As we are using continuous sweepouts, the arguments here are simpler as compared with \cite{Zhou19}. For sake of completeness, we will provide necessary details. Note that all notations related to min-max construction in Section \ref{sec:min-max and tightening} are valid in the current setting, that is, $h\equiv 0$.

We can assume that $g$ is bumpy, and the general case follows by approximation in the same was as in Theorem \ref{thm:multiplicity one for relative min-max}.

Following the same procedure as Theorem \ref{thm:pull tight}, Theorem \ref{thm:existence of almost minimizing pairs} and Theorem \ref{thm:pmc min-max theorem}, we can find a pull-tight minimizing sequence $\{\Phi_i\}_{i\in\mb N}\subset \Pi$, such that every $V\in \mf C(\{\Phi_i\})$ is stationary (for area), and moreover, if $V\in\mf C(\{\Phi_i\})$ is almost minimizing (for area) in every $L=L(m)$-admissible collection of annuli w.r.t. some min-max subsequence, then $V$ has support a closed, smoothly embedded, minimal surface $\Sigma$ with $\|V\|(M) = \mf L(\Pi)$, and satisfies Property {\bf(R')} \eqref{eq:property R2}, that is, give any $L$-admissble collection of annuli $\ms C$, $\Sigma$ is stable in at least one of them. 
 Lemma \ref{lem:AM in L annuli} implies such $V$ always exists. 

\medskip
\noindent\textit{Step 1. We will do an extra tightening process to find another minimizing sequence, still denoted as $\{\Phi_i\}$, such that for $i$ sufficiently large, either $\Phi_i(x)$ is close to a smooth min-max minimal surface, or the area $\mc H^2 (\Phi_i(x))$ is strictly less than ${\mf L}(\Pi)$.}
\medskip

Let $\mc S$ be the collection of all stationary 2-varifolds with mass lying in $[\mf L(\Pi)-1, \mf L(\Pi)+1]$, whose support is a closed, smoothly embedded, minimal surface satisfying Property {\bf(R')} \eqref{eq:property R2} for $L=L(m)$. By the bumpiness assumption, $\mc S$ is a finite set.

Choose a small $\overline{\epsilon}>0$, and let 
\[ Z_i = \{x\in X: \mf F(\Phi_i(x), \mc S)\geq \overline{\epsilon} \}, \text{ and } Y_i = \closure(X\setminus Z_i). \]
Clearly we can make $Z_0\subset Z_i$ for $\oli\epsilon$ small enough. Consider the sub-coordinating sequence $\{\Phi_i|_{Z_i}\}_{i\in \mb N}$. Then we can define $\mf L(\{\Phi_i|_{Z_i}\})$ and $\mf C(\{\Phi_i|_{Z_i}\})$ in the same way as in Section \ref{ss:min-max problem}. 
\begin{lemma}[c.f. \cite{Zhou19}*{Lemma 5.7}]\label{lem:a dichotomy}
We have the following dichotomy:
\begin{enumerate}[label=\roman*)]
    \item either no element $V\in \mf C(\{\Phi_i|_{Z_i}\})$ is almost minimizing in every $L=L(m)$-admissible collection of annuli w.r.t. some min-max subsequence, 
    \item or $\mf L(\{\Phi_i|_{Z_i}\}) < \mf L(\Pi)$.
\end{enumerate}
\end{lemma}
\begin{proof}
    The proof is the same as \cite{Zhou19}*{Lemma 5.7} using the pull-tight and min-max regularity results, so we omit it.
\end{proof}

Let $\lambda_i = \Phi_i^*(\oli\lambda)\in H^1(X, \mb Z_2)$. Note that $\Phi_i(Y_i)$ lies in the $\oli\epsilon$-neighborhood of a finite set $\mc S$ (in the $\mf F$-metric). When $\overline{\epsilon}$ is small enough, we know that no continuous map $\phi$ from $S^1$ to such a neighborhood can form a sweepout of $M$, and hence must satisfy: $\oli\lambda[\phi] = 0$ by the discussion above this theorem. 

Consider the inclusion maps $\iota_i: Y_i \to X$.  
When $\overline{\epsilon}$ is small enough, we must have 
\begin{equation}\label{eq:pullback trivial on Zi}
\iota_{i}^*(\lambda_i) = 0 \in H^1(Y_i, \mb Z_2).    
\end{equation}

\begin{lemma}
As above, there exists another minimizing sequence $\{\Phi'_i\}_{i\in \mb N}$ with $\Phi'_i$ homotopic to $\Phi_i$ for each $i$, such that Lemma \ref{lem:a dichotomy}(\rom{2}) holds true.
\end{lemma}
\begin{proof}
We may assume $\mf L(\{\Phi_i|_{Z_i}\}) = \mf L(\Pi)$, then  Lemma \ref{lem:a dichotomy}(\rom{1}) must be true, that is, no element $V\in \mf C(\{\Phi_i|_{Z_i}\})$ is almost minimizing in every $L=L(m)$-admissible collection of annuli w.r.t. some min-max subsequence. By the same argument in the proof Lemma \ref{lem:AM in L annuli}, we can find $\epsilon_0>0$, such that for any $\delta>0$, $i>1/\epsilon_0$, and any $x\in Z_i$ satisfying:
\[ \mc H^2 \big(\Phi_i(x) \big) \geq \mf L(\Pi) - \epsilon_0,\]
there exists an $L$-admissible collection $\ms C_{i, x}$ such that $\Phi_i(x)$ is not $(\epsilon_0, \delta)$-almost minimizing in any annulus in $\ms C_{i, x}$. We can then follow the same deformation process as in the proof of Lemma \ref{lem:AM in L annuli} to deform $\Phi_i$ using isotopies to some $\Phi'_i: X\to \oli{\ms X}(\Sigma_0)$, such that 
\[ \sup_{x\in Z_i}\mc H^2 \big(\Phi'_i(x)\big)\leq \sup_{x\in Z_i}\mc H^2 \big(\Phi_i(x)\big) - \epsilon_0/4 <\mf L(\Pi), \]
for $i_0$ sufficiently large. This completes the proof.
\end{proof}

Since $\Phi_i\to \Phi'_i$ is a homotopic deformation, we know that \eqref{eq:pullback trivial on Zi} still holds true for $\Phi'_i$. 

Note that $\Phi_i(Y_i)$ lies in an $\oli\epsilon$-neighborhood of $\mc S$ under the $\mf F$-metric, so each $\Phi_i(x)$, $x\in Y_i$, has area bounded uniformly away from zero, and hence $\Phi_i(Y_i)$ lies inside $\ms X(\Sigma_0)$; (that is, $\Phi_i(Y_i)$ does not contain degenerate graphs). The deformation process used isotopies, so $\Phi_i'(Y_i)$ also lies inside $\ms X(\Sigma_0)$.

In the following, we will still write $\Phi'_i$ as $\Phi_i$.

\medskip
\noindent\textit{Step 2. For $i$ large enough, we can lift the maps $\Phi_i: Y_i \to {\ms X}(\Sigma_0)$ to its double cover $\wti\Phi_i: Y_i \to \wti{\ms X}(\Sigma_0)$, the existence of which is guaranteed by \eqref{eq:pullback trivial on Zi}.}
\medskip

Note that $\partial Y_i\subset Z_i$, so we have $\sup_{x\in \partial Y_i}\mc H^2 \big(\wti\Phi_i(x)\big) < \mf L(\Pi)$ for $i$ large enough.  
We claim that
\begin{claim}[c.f. \cite{Zhou19}*{Lemma 5.8}]
Let $\wti\Pi_i$ be the $(Y_i, \partial Y_i)$-homotopy class associated with $\wti\Phi_i|_{Y_i}$ in $\wti{\ms X}(\Sigma_0)$, then we must have
\[ {\mf L}(\wti\Pi_i) \geq {\mf L}(\Pi) > \sup_{x\in \partial Y_i}\mc H^2 \big(\wti\Phi_i(x)\big). \]
\end{claim}
The proof is essentially the same as that of \cite{Zhou19}*{Lemma 5.8}, so we omit it. 

\medskip
\noindent\textit{Step 3. We apply Theorem \ref{thm:multiplicity one for relative min-max} to $\wti\Pi_i$ and finish the proof by taking $i\to\infty$.}
\medskip

Note that $\wti{\ms X}(\Sigma_0)$ can be identified with our total space $\ms E$ \eqref{eq:ms E} in Section \ref{ss:min-max problem}. Applying Theorem \ref{thm:multiplicity one for relative min-max} to each $\wti\Pi_i$, we obtain a disjoint collection $\Gamma_i = \cup_{j=1}^{N_i}m_{i,j}\Gamma_{i, j}$ of connected, closed, smoothly embedded, minimal surfaces with integer multiplicity $m_{i,j}$, satisfying all conclusions in Theorem \ref{thm:multiplicity one for relative min-max}. By the proof therein, $\Gamma_i$ also satisfies Property {\bf(R')} \eqref{eq:property R2} for all $i$. Note that ${\mf L}(\wti\Pi_i)\leq \sup_{x\in Y_i}\mc H^2 (\Phi_i(x))\to {\mf L}(\Pi)$. Since there are only finitely many such $\Gamma_i$, we know that for $i$ sufficiently large ${\mf L}(\wti\Pi_i) = {\mf L}(\wti\Pi_{i+1}) = \cdots = {\mf L}(\Pi)$.
Hence we finish the proof of Theorem \ref{thm:classical multiplicity one}.
\end{proof}

\section{Existence of minimal spheres}\label{sec:existence of minimal spheres}

In this section, we apply the multiplicity one theorem in Riemannian three-spheres $(M,g)$ and prove the existence of four distinct embedded minimal two-spheres if $(M,g)$ does not contain degenerate-and-stable minimal two-spheres. Without loss of generality, we always assume that $(M,g)$ contains only finitely many embedded minimal two-spheres.

Note that the three-sphere admits a nontrivial homotopy class of $[-1, 1]\ttimes\mb{RP}^{k-1}$-sweepouts of two-spheres for each $k=1,2,3,4$. (Here $[-1, 1]\ttimes\mb{RP}^{k-1}$ is the twisted interval bundle over $\mb{RP}^{k-1}$.) If $(M,g)$ has no stable minimal two-spheres, the Multiplicity One Theorem \ref{thm:classical multiplicity one} applies directly to give an embedded minimal two-sphere for each homotopy class of $[-1, 1]\ttimes\mb{RP}^{k-1}$-sweepouts. The well-known Lusternik–Schnirelmann theory implies that the min-max widths are strictly increasing in a bumpy metric, and hence these are distinct two-spheres. We will provide more details for this part in Section \ref{subsec:four sweepouts}.

If $(M,g)$ contains stable minimal-two spheres, we will follow the strategy due to A. Song \cite{Song18} to cut $M$ along a disjoint collection of stable minimal two-spheres and consider the remaining compact manifolds glued with cylindrical ends. To develop the min-max theory in such a non-compact manifold, we approximate it by a sequence of compact domains with mean concave boundary and generalize the work of M. Li \cite{Li-CPAM} for free boundary minimal surfaces with controlled topology.

\medskip
In Section \ref{subsec:min-max in compact domains}, we introduce some notations from \cite{Li-CPAM} and generalize our PMC min-max Theorem \ref{thm:pmc min-max theorem} to the free boundary setting when the prescribing function is supported in the interior of the compact manifolds. In Section \ref{subsec:construction of non-compact}, we present the construction of non-compact manifolds with cylindrical ends and sequences of compact manifolds approximating them. Moreover, we will also prove that as the compact manifolds having enough long ``tails'', the free boundary min-max theory will produce closed minimal surfaces with genus bounds.  Next we prove the uniform upper bound of the first two widths in those compact manifolds in Section \ref{subsec:2 width upper bound}. Finally, Section \ref{subsec:four minimal S2} is devoted to the proof of the main theorem.

\subsection{Min-max theory in compact manifolds with boundary}\label{subsec:min-max in compact domains}
We now recall the min-max theory for free boundary minimal surfaces with controlled topology by M. Li \cite{Li-CPAM}.

Let $(M,g)$ be a Riemannian three-sphere and $N\subset M$ be a compact domain with smooth boundary $\partial N$. Denote by $\mk{Is}$ the space of all isotopies on $M$. Define 
\[\mk{Is}^{\mr{out}}:=\big\{\{\varphi_s\}\in \mk{Is}; N\subset \varphi_s(N) \text{ for all } s\in[0,1] \big\}\]
to be the isotopies in $M$ that can push points out of the compact set $N$, but not into $N$. Given an open subset $U\subset M$, we define $\mk{Is}^{\mr{out}}(U)$ to be those in $\mk{Is}^{\mr{out}}$ that is supported in $U$.

We will use Theorem \ref{thm:pmc min-max theorem} and the regularity theorem in \cite{Li-CPAM}*{Theorem 4.7} to prove that for such a compact manifold $N\subset M$, the relative min-max theory will produce free boundary minimal surfaces. We now generalize some notions to their free boundary counterparts.

Let $h: N \to \mb R$ be a fixed smooth function, such that
\[ \text{$h=0$ in a neighborhood of $\partial N$.} \]
Recall that $X\subset I(m, k_0)$ is a cubical complex and $Z\subset X$ is a subcomplex, and $\ms E$ is the space of embedded separating surfaces of genus $\mk g_0$ in $M$ \eqref{eq:ms E}. A continuous map $\psi: X\to C^\infty(M, M)$ is said to be \textit{outward isotopic deformation}, if for each $x\in X$, there exists an outward isotopy $\{\varphi_{x, s}\}_{s\in[0,1]}\in \mk{Is}^{\mr {out}}$ such that $\psi_x := \psi(x)$ is equal to $\varphi_{x, 1}$.  Let $\Phi_0:X\to \ms E$ be a continuous map. A family $\bm \Xi$ of $(X,Z)$-sweepouts homotopic to $\Phi_0$ relative to $\Phi_0|_Z$ is said to be \textit{saturated}, if for any $\Phi\in \bm \Xi$, and any outward isotopic deformation $\psi: X\to C^\infty(M, M)$ with $\psi|_Z = \Id$,
\[ \Phi'(x) := (\psi_x)_\#\Phi(x) \quad \text{also belongs to $\bm \Xi$.} \]

To produce free boundary solutions in $N$, we will only count the area and volume restricted to $N$. Precisely, given $(V,\Omega)\in \VC(M)$, we define
\begin{equation*}
\mc A^h_N(V,\Omega)= \|V\|(N)-\int_{\Omega\cap N} h\,\mr d \Vol.     
\end{equation*}
For any saturated family $\bm\Xi$ of $(X,Z)$-sweepouts, we define
\begin{equation*}
\mf L^h_N(\bm\Xi)=\inf_{\Phi\in \bm\Xi}\max_{x\in X}\mc A^h_N \big(\Phi(x)\big).
\end{equation*}
We also use $\mc A^h_N(V,\Omega;g)$ and $\mf L^h_N(\bm\Xi; g)$ to indicate the metrics. We can adapt the notions related to min-max construction in Definition \ref{def:min-max sequence} in a straightforward manner to this setting. Given a fixed $\bm\Xi$, we can similarly define \textit{minimizing sequences, min-max subsequences, and critical sets}. In particular, the critical set of a minimizing sequence $\{\Phi_i\}\subset \bm\Xi$ is defined by 
\[
    \mf C(\{\Phi_i\})=\left\{ (V, \Omega)\in \VC(N) \left|
    \begin{aligned}  
         &\ \exists \text{ a min-max subsequence }\{\Phi_{i_j}(x_j)\} \text{ such }\\
         &\text{ that }   \ms F\big(\Phi_{i_j}(x_j)\res N, (V, \Omega)\big) \to  0 \text { as } j\to\infty 
    \end{aligned}\right\}\right..
\]
Note that the $\ms F$-metric is defined on $N$.

We can also extend the notion of $C^{1,1}$-almost embedded surfaces in Section \ref{SS:C11 almost embedded surface} to the current setting. Let $\Sigma$ be a compact smooth surface (2-dimensional manifold) with smooth boundary $\partial\Sigma$. A $C^{1,1}$-immersion $\phi: \Sigma\to N$ with $\phi(\partial\Sigma)\subset \partial N$ is called \textit{almost embedded} if Definition \ref{def:C11 almost embedding} holds in the interior $\interior(N)$, and \cite{LZ16}*{Definition 2.6}\footnote{\cite{LZ16}*{Definition 2.6} refers to the notion of almost proper embedding, which says that $\Sigma$ is an embedding into $M$ near $\partial N$, $\phi(\partial\Sigma)\subset \partial N$, but $\phi(\interior(\Sigma))$ may touch $\partial N$.} holds near $\partial N$. We will simply write $\phi(\Sigma, \partial\Sigma)$ as $(\Sigma, \partial\Sigma)$. Similarly, we can define \textit{$C^{1,1}$ boundary (c.f. Definition \ref{def:c11 boundary}) and $C^{1,1}$ free $h$-boundary} (c.f. Definition \ref{def:C11 h-boundary}). To be precise, a $C^{1,1}$ almost embedded surface $(\Sigma, \partial\Sigma)\subset (N, \partial N)$ is called a \textit{boundary} if there exists $\Omega\in \C(N)$, such that $\Sigma\res \interior(N) = \partial\Omega\res\interior(N)$\footnote{Since $h=0$ near $\partial N$, we only need this boundary structure in the interior $\interior(N)$.}, and a triple $(\Sigma, \partial\Sigma, \Omega)$ is called a \textit{free $h$-boundary} if for any vector field $X\in \mk X(N, \Sigma)$\footnote{This means $X(q)\in T_q(\partial N)$ for all $q$ in a neighborhood of $\partial \Sigma$ in $\partial N$.} (see \cite{LZ16}*{page 501}), we have $\delta \A^h_N|_{\Sigma, \Omega}(X) = 0$. By the first variation formula, $\Sigma$ must meet $\partial N$ orthogonally along $\partial\Sigma$, and $\interior(\Sigma)$ is minimal near $\partial N$; see \cite{LZ16}*{Definition 2.8}. Note that since the touching set $\mc S(\Sigma)$ (where $\Sigma$ touches with itself) lies in the interior $\interior(N)$, we can extend the \textit{strong $\A^h$-stationarity} (Definition \ref{def:strong one-sided stationarity}) without any change to this setting.

We have the following extension of \cite{Li-CPAM}*{Theorem 4.7}. Note that the proper embeddedness in the regularity statement claimed in \cite{Li-CPAM}*{Theorem 4.7} is incorrect; this part has been corrected in \cite{LZ16}*{Theorem 5.2}, and the corresponding statement is almost proper embeddedness. 

\begin{theorem}\label{thm:relative min-max in any compact domain} 
Let $\bm\Xi$ be a saturated family of $(X,Z)$-sweepouts relative to $\Phi_0|_Z$. Suppose 
\begin{equation*}
\mf L^h_N(\bm \Xi)>\max \left\{\max_{x\in Z}\mc A^h_N\big(\Phi_0(x)\big), 0\right\}.
\end{equation*}
Then there exist a minimizing sequence $\{\Phi_i\}_{i\in\mb N}\subset \bm \Xi$, and a strongly $\mc A^h$-stationary, $C^{1,1}$, free $h$-boundary $(\Sigma,\Omega)$ lying in the critical set $\mf C(\{\Phi_i\})$, that is, $(\Sigma,\Omega)$ is the $\ms F$-limit of some min-max subsequence $\{\Phi_{i_j}(x_j)\res N\}$, and 
\[ \mc A^h_N(\Sigma,\Omega)=\mf L^h_N(\bm\Xi). \]
Moreover, the min-max sequence $\{\Phi_{i_j}(x_j)\}_{j\in\mb N}$ can be chosen so that there exist $\epsilon_j,\delta_j\to 0$ such that $\Phi_{i_j}(x_j)$ is $(\mc A^h,\epsilon_j,\delta_j)$-almost minimizing in any $L(m)$-admissible collection of annuli $\mr{An}(p;s_1,r_1),\cdots, \mr{An}(p;s_L,r_L)\subset N\setminus \partial N$. 
\end{theorem}
\begin{proof}
The existence of the desired minimizing sequence follows by adapting to the free boundary setting the tightening process in Section \ref{SS:tightening} (in a similar way as \cite{Li-CPAM}*{Proposition 5.1}) and the existence of almost minimizing pairs in Theorem \ref{thm:existence of almost minimizing pairs}. Indeed, for any $L(m)$-admissible collection of annuli in $N\setminus \partial N$, the min-max sequence can be chosen so that $\Phi_{i_j}(x_j)$ is $(\mc A^h,\epsilon_j,\delta_j)$-almost minimizing in at least one of them.

Now the interior regularity follows from Theorem \ref{thm:pmc min-max theorem}.

It remains to consider the regularity around $\partial N$. Note that $h\equiv 0$ in a neighborhood of $\partial N$. Then by the free boundary min-max theory \cite{Li-CPAM}*{Theorem 4.7} and \cite{LZ16}*{Theorem 5.2}, $\Sigma$ is a almost embedded free boundary minimal surface in a neighborhood of $\partial N$.
\end{proof}

\subsection{Construction of non-compact manifolds with cylindrical ends}\label{subsec:construction of non-compact}
We recall the construction of non-compact manifolds with cylindrical ends by A. Song in \cite{Song18}*{Section 2.2}. Let $(N,\partial N,g)$ be a compact three dimensional Riemannian manifold such that $\partial N$ is a closed, embedded, stable minimal surface  with \textit{a contracting neighborhood}, that is, there is a map 
\[ \varphi:\partial N\times [0,\hat t]\to N,\]
so that $\varphi$ is a diffeomorphism to its image, $\varphi(\partial N\times \{0\})=\partial N$, and for all $t\in (0,\hat t]$, $\varphi(\partial N\times \{t\})$ has non-zero mean curvature vector\footnote{Here the mean curvature vector is defined as $-\dv(\nu)\nu$ for a choice of unit normal $\nu$.} pointing towards $\partial N$. We endow $\partial N\times [0,+\infty)$ with the product metric. Let $\ms C(N)$ be the non-compact manifold
\[   N\cup (\partial N\times [0,+\infty))  \]
by identifying $\partial N$ with $\partial N\times \{0\}$. We endow it with the metric $\hat g$ such that $\hat g=g$ on $N$ and is equal to the product metric on $\partial N\times [0,\infty)$. Note that $\hat g$ is Lipschitz continuous. 

Now we approximate $\ms C(N)$ by compact manifolds as follows. 
Let $N_\epsilon:=N\setminus \varphi(\partial N\times [0,\epsilon))$. Denote by $\nu$ the unit outward normal vector field of $\partial N_\epsilon$. For a small constant $\delta_\epsilon>0$, the map
\[   \gamma_\epsilon:\partial N_\epsilon\times [-\delta_\epsilon,0]\to N_\epsilon , \quad (x,t)\longmapsto \exp(x,t\nu) \]
is well-defined and gives Fermi coordinates on one side of $\partial N_\epsilon$. Then by A. Song \cite{Song18}*{Section 2.2}, there exist smooth metrics $g_\epsilon$ on $N_\epsilon$ satisfying Lemma 4, Lemma 5 and Lemma 7 in \cite{Song18}, such that $(N_\epsilon, g_\epsilon)$ approaches $(\ms C(N),\hat g)$ in appropriate sense. In particular, 
\begin{enumerate}[label=(\roman*)]
    \item $g_\epsilon = g$ in $N_\epsilon\setminus \gamma_\epsilon(\partial N_\epsilon\times [-\delta_\epsilon, 0])$;
    \item for $t \in [-\delta_\epsilon,0]$, the slices $\gamma_\epsilon(\partial N_{\epsilon}\times  \{t\})$ have non-zero mean curvature vector pointing towards $\partial N_\epsilon$ with respect to the new metric $g_\epsilon$;
    \item\label{item:gepsilon=g along surfaces} $\gamma_\epsilon^*(g_\epsilon)=\gamma_\epsilon^*(g)$ on $\partial N_\epsilon\times \{t\}$ for all $t\in [-\delta_\epsilon,0]$.
\end{enumerate}

Assume that $(N,\partial N,g)$ is isometrically embedded into a closed three-manifold $(M,\wti g)$. Then one can extend the metric $g_\epsilon$ to a metric $\wti g_\epsilon$ on $M$ so that $\wti g_\epsilon=g_\epsilon$ on $N_\epsilon$.

Given a continuous $\Phi_0: X\to \oli{\ms X}(\Sigma_0)$ (see Section \ref{SS:multiplicity one for Simon-Smith} for notations), we can similarly define saturated families $\bm\Xi$ of $(X, Z_0)$-sweepouts in $\oli{\ms X}(\Sigma_0)$ homotopic to $\Phi_0$ w.r.t. outward isotopic deformations for $N_{\epsilon}$.

\begin{proposition}\label{prop:min max in Nk}
Let $\epsilon_k\to 0$ be a sequence of positive constants. With notions as above, we use $N_k$ and $\wti g_k$ to denote $N_{\epsilon_k}$ and $\wti g_{\epsilon_k}$ for simplicity. Suppose that 
\[\liminf_{k\to\infty}\mf L_{N_{k}}(\bm \Xi;\wti g_{k})<\infty.\]
Then up to a subsequence, for sufficiently large $k$, there exist a collection of pairwise disjoint, connected, closed, embedded, minimal surfaces $\Gamma^k_1,\cdots, \Gamma^k_{I_k}\subset (N\setminus \partial N,g)$ and positive integers $m^k_1,\cdots, m^k_{I_k}$, such that
\[\mf L_{N_{k}}(\bm \Xi;\wti g_{k})= \sum_{i=1}^{I_k}m^k_i\mc H^2(\Gamma^k_i),\]
where $m^k_i=1$ if $\Gamma^k_i$ is unstable. Moreover, the varifold $\sum [\Gamma^k_i]$ satisfies Property {\bf(R')} \eqref{eq:property R2} for any $L(m)$-admissible collection of annuli in $N_k\setminus \partial N_k$.

Furthermore, if $M$ is orientable, then the genus bound \eqref{eq:genus bound1} holds with $\mk g_0 = \mk g(\Sigma_0)$.
\end{proposition}
\begin{proof}
Without loss of generality, we assume that for all $k$,
\[ \mf L_{N_k}(\bm \Xi;\wti g_k)\leq \Lambda_0<\infty.\]
We can assume that $\wti g_{k}$ is bumpy, and the general case follows by approximating $\wti g_k$ by bumpy metrics. 
Denote by $\ms D_k:=\{S_1,\cdots, S_{\beta_k}\}$ the collection of connected, closed, embedded, minimal surface $S$ in $(N_k\setminus \partial N_k,\wti g_k)$ satisfying
\[\mc H^2 (S)\leq \Lambda_0+1,\]
and Property {\bf(R')} \eqref{eq:property R2} for $L=L(m)$. 
Then using the monotonicity formula for minimal surfaces \cite{Song18}*{Lemma 2} and the mean concave foliation near $\partial N_k$, we know that there exists $d_0>0$ independent of $k$, such that  
\[ \cup_{j=1}^{\beta_k}S_j\subset B(p,d_0;\wti g_k),\] 
where $p$ is a fixed point in $N\setminus \partial N$.

Now we claim that one can adapt the argument in Theorem \ref{thm:classical multiplicity one} to prove that for all sufficiently large $k$, there exists a stationary varifold $V_k=\sum_{i=1}^{I_k}m^k_i[\Gamma^k_i]$ that achieves $\mf L_{N_k}(\bm \Xi;\wti g_k)$, where $\Gamma^k_i\in \ms D_k$. Indeed, by the same arguments in Step 1 and Step 2 in the proof of Theorem \ref{thm:classical multiplicity one}, we have that for each sufficiently large $k$, one can approximate $\mf L_{N_k}(\bm \Xi;\wti g_k)$ by a sequence of widths of relative sweepouts of separating surfaces. Then it suffices to adapt the argument in Theorem \ref{thm:multiplicity one for relative min-max} to the current free boundary settings. 

We now indicate the modification in Theorem \ref{thm:multiplicity one for relative min-max}. By the same strategy, we will choose a suitable prescribing function $h$ on $M$ and then approximate $\mf L_{N_k}(\bm\Xi; \wti g_k)$ by $\mf L_{N_k}^{\varepsilon_j h}(\bm\Xi; \wti g_k)$. Note that the minimal surfaces in $\ms D_k$ are all closed and lie in the interior. Then the prescribing function $h$ can be chosen such that $\spt(h)\subset  B(p,d_0+1;\wti g_k)$. Next we replace Theorem \ref{thm:pmc min-max theorem} by Theorem \ref{thm:relative min-max in any compact domain} to obtain a strongly $\A^{\varepsilon_j h}$-stationary, $C^{1,1}$, free $\varepsilon_jh$-boundary $(\Sigma_{k,j}, \Omega_{k, j})$ 
in $(N_k,\wti g_k)$ for each homotopy class of relative sweepouts; moreover, $(\Sigma_{k, j}, \Omega_{k, j})$ satisfies Property {\bf(R)} \eqref{eq:property R} for any $L=L(m)$-admissble collection of anuuli in $N_k\setminus\partial N_k$. Observe that $\gamma_{\epsilon_k}(\partial N_k\times \{t\})$ has non-zero mean curvature vector pointing towards $\partial N_k$ and $\partial N_k$ is mean concave. Thus, $\Sigma_{k, j}$ cannot contain $\partial N_k$ (since $\interior(\Sigma_{k, j})$ is minimal near $\partial N_k$), and by the area bound and the monotonicity formula, $\Sigma_{k,j}$ can not touch $\partial N_k$, that is, $\Sigma_{k,j}$ is closed, for all sufficiently large $k$. Then by letting $\varepsilon_j\to 0$, using the same argument as in Theorem \ref{thm:multiplicity one for relative min-max}, we know that there exists a stationary varifold $V_k=\sum_{i=1}^{I_k}m^k_i[\Gamma^k_i]$ lying in $(N_k\setminus \partial N_k, \wti{g}_k)$ that achieves $\mf L_{N_k}(\bm \Xi;\wti g_k)$, where $\Gamma^k_i\in \ms D_k$, and 
\begin{itemize}
    \item $m^k_i=1$ if $\Gamma^k_i$ is unstable;
    \item $V_k$ satisfies Property {\bf(R')} \eqref{eq:property R2} for any $L(m)$-admissible collection of annuli in $N_k\setminus \partial N_k$;
    \item if $M$ is orientable, then the genus bound \eqref{eq:genus bound1} holds with $\mk g_0 = \mk g(\Sigma_0)$.
\end{itemize}

It remains to prove that $V_k$ is stationary w.r.t. the metric $g$ for all sufficiently large $k$. Indeed, by the same argument in \cite{Song18}, the limit of $\cup_i\Gamma_i^k$ is a minimal surface in $(N\setminus \partial N,g)$. Note that such a limit does not intersect $\partial N$. By the Hausdorff convergence of $\Gamma_i^k$, we conclude that $\Gamma_i^k\subset N_k\cap\{\wti g_k = g\}$ for all sufficiently large $k$. In particular, $\Gamma_i^k$ is a minimal surface in $(N\setminus \partial N,g)$. 
This completes the proof of Proposition \ref{prop:min max in Nk}.
\end{proof}

\subsection{Sweepouts in three-spheres}\label{subsec:four sweepouts}
In this subsection, we will first recall the fact that the three-sphere always admits a nontrivial $k$-parameter ($k=1,2,3,4$) sweepout of two-spheres; see also \cite{HK19}*{Section 2}. Then we apply the Multiplicity One Theorem for Simon-Smith min-max theory (Theorem \ref{thm:classical multiplicity one}) to construct four distinct embedded minimal two-spheres if the manifold has no stable minimal two-spheres.

Let $\mb{S}^3\subset\mb R^4$ be the standard unit round three-sphere, and $x_1,\cdots, x_4$ be the four coordinate functions. Consider the spaces
\[
\ms X:=\{\phi(\mb S^2)\big|\phi: \mb S^2\to \mb S^3 \text{ is a smooth embedding}\},
\]
and 
\[
\ms Y:=\{\phi(\mb S^2)\big|\phi:\mb S^2\to\mb S^3 \text{ is a smooth map whose image is a point or an interval}\}.
\]
Denote by $\oli{\ms X}=\ms X\cup \ms Y$ endowed with un-oriented smooth topology.

For each $i=1, 2, 3, 4$, let $\ms P_i$ be the collection of continuous maps $\Phi: X\to \oli{\ms X}$, 
with $\Phi(Z_0)\subset \ms Y$, so that there exists some $\oli\lambda\in H^1(\oli {\ms X},\ms Y;\mb Z_2)$, such that $\alpha = \Phi^*(\oli\lambda)$ satisfies 
\[ \alpha^i \neq 0 \in H^i(X, Z_0;\mb Z_2). \]

Next we describe four explicit sweepouts that belongs to $\ms P_i$ for each $i=1, 2, 3, 4$.  
We use $[-1, 1]\ttimes \mb{RP}^3$ to denote the twisted $[-1,1]$-bundle over $\mb{RP}^3$, and $[a_0, a_1,a_2,a_3,a_4]$ to denote a point in $[-1, 1]\ttimes\mb{RP}^3$; that is $a_0\in [-1, 1]$, $a_1^2+\cdots+ a_4^2=1$, and $(a_0, a_1, a_2, a_3, a_4)$ is identified with $(-a_0, -a_1, -a_2, -a_3, -a_4)$. 
When $a_0\neq \pm 1$, we denote
\[ \mc G([a_0, a_1,a_2,a_3,a_4]):=\{a_1x_1+a_2x_2+a_3x_3+a_4x_4= a_0\}\cap \mb S^3;\]
when $a_0=\pm 1$, $\mc G(a_0, a_1, a_2, a_3, a_4)$ denotes a point given by $\pm(a_1, a_2, a_3, a_4)\in \mb S^3$.

We now define four maps:
\begin{gather*} 
\Psi_1: [-1,1]\ttimes \mb{RP}^0\to \oli{\ms X}, \quad a_0 \longmapsto \mc G(a_0,1,0,0,0);\\
\Psi_2: [-1,1]\ttimes \mb{RP}^1\to \oli{\ms X}, \quad [a_0,a_1,a_2]\longmapsto \mc G(a_0,a_1,a_2,0,0);\\
\Psi_3: [-1,1]\ttimes\mb{RP}^2\to \oli{\ms X}, \quad [a_0,a_1,a_2,a_3]\longmapsto \mc G(a_0,a_1,a_2,a_3,0);\\
\Psi_4: [-1,1]\ttimes\mb{RP}^3\to\oli{\ms X}, \quad [a_0,a_1,a_2,a_3,a_4]\longmapsto \mc G(a_0,a_1,a_2,a_3,a_4).
\end{gather*}
For simplicity, we use $\mc X_i$ to denote $[-1,1]\ttimes\mb {RP}^{i-1}$, (and as compared with our definition of $\ms P_i$, $Z_0=\partial \mc X_i$ for each $i$.) Let $f: [-1, 1] \to \mc X_i$ be an arbitrary embedding of a fiber of the $[-1,1]$-bundle.

We now show that $\Psi_i\in\ms P_i$ for $i=1,2,3,4$. Denote by $\bm{\iota}: \oli{\ms X} \to \mc Z_2(S^3; \mb Z_2)$ the natural inclusion map into the space of mod-2 integral cycles. Note that $\bm{\iota}(\ms Y) = \{0\}$, that is, the image of each element in $\ms Y$ is a zero cycle.  Consider the chain of maps:
\[
\begin{tikzcd}
({[-1, 1]}, \partial{[-1, 1]}) \arrow[r, "f"]
&(\mc X_i, \partial \mc X_i) \arrow[r, "\Psi_i"]
&(\oli{\ms X}, \ms Y) \arrow[r, "\bm{\iota}"] 
&(\mc Z_2(S^3;\mb Z_2),\{0\});
\end{tikzcd}
\]
the composition map $\bm{\iota}\circ \Psi_i\circ f: ([-1, 1], \partial{[-1, 1]}) \to (\mc Z_2(S^3;\mb Z_2),\{0\})$ is then a sweepout in the sense of Almgren; see \cite{Alm62}, \cite[Definition 3.4]{Marques-Neves16} and \cite[Theorem 5.8]{Zhou15}. Therefore, 
we know that $(\bm{\iota}\circ \Psi_i\circ f)^*: H^1(\mc Z_2(S^3;\mb Z_2),\{0\}; \mb Z_2) \to H^1([-1, 1], \partial{[-1, 1]}; \mb Z_2) = \mb Z_2$ is nontrivial, and by the chain of pull-back maps:
\[  
\begin{tikzcd}
H^1(\mc Z_2(S^3;\mb Z_2),\{0\};\mb Z_2) \arrow[r,"\bm{\iota}^*"] &H^1(\oli{\ms X},\ms Y;\mb Z_2)\arrow[r,"\Psi_i^*"] &H^1(\mc X_i,\partial \mc X_i;\mb Z_2)\\ \arrow[r,"f^*"]
&H^1([-1, 1], \partial[-1, 1]; \mb Z_2),
\end{tikzcd}
\]
we also know that $\Psi_i^*: H^1(\oli{\ms X}, \ms Y; \mb Z_2) \to H^1(\mc X_i, \partial\mc X_i; \mb Z_2)$ is nontrivial. This together with the structure of the relative cohomology ring 
\[
H^*(\mc X_i,\partial \mc X_i;\mb Z_2)\simeq \mb Z_2[\alpha]/[\alpha^{i+1}] \] 
implies that $\Psi_i\in \ms P_i$ for $i=1,2,3,4$.

Recall that 
 \[ \mf L(\ms P_i):=\inf_{\Phi\in \ms P_i}\sup_{x\in \mr{dom} \Phi}\mc H^2(\Phi(x)).\] 
The next result follows from Lusternik–Schnirelmann theory. We will provide a detailed proof in Appendix \ref{appen:LS inequality}, which is borrowed from \cite{MN17}*{Theorem 6.1} with minor modifications.
\begin{lemma}
\label{lem:LS theory}
Suppose that $(M,g)$ contains only finitely many embedded minimal two-spheres. Then 
\[   0<\mf L(\ms P_1)<\mf L(\ms P_2)<\mf L(\ms P_3)<\mf L(\ms P_4).  \]
\end{lemma}

\begin{remark}
Haslhofer-Ketover \cite{HK19}*{Theorem 5.2} proved a similar result. 
\end{remark}

\begin{theorem}\label{thm:no stable}
Suppose that $(M,g)$ is a Riemannian 3-sphere. If $(M,g)$ does not contain any embedded stable minimal two-sphere, then $(M,g)$ admits at least four distinct embedded minimal two-spheres.
\end{theorem}
\begin{proof}
Without loss of generality, we assume that $(M,g)$ contains only finitely many embedded minimal 2-spheres. Thus by Lemma \ref{lem:LS theory}, 
\[   0<\mf L(\ms P_1)<\mf L(\ms P_2)<\mf L(\ms P_3)<\mf L(\ms P_4).  \]
Note that $\ms P_i$ may contain many different homotopy classes of $(X, Z_0)$-sweepouts, but since there are only finitely many minimal $2$-spheres in $(M, g)$, the min-max values of these homotopy classes have to stabilize.  Therefore, by the Multiplicity One Theorem for Simon-Smith min-max theory (Theorem \ref{thm:classical multiplicity one}), each $\mf L(\ms P_i)$ is realized by a disjoint union of some closed minimal two-spheres with integer multiplicities. Since $(M,g)$ has no stable minimal spheres, any two embedded minimal $2$-spheres have to intersect with each other. Thus $\mf L(\ms P_i)$ is achieved by some minimal two-sphere $\Gamma_i$ with integer multiplicity $m_i$. Since $\Gamma_i$ is unstable by assumption, we have $m_i=1$. Thus we conclude that $\Gamma_1,\cdots,\Gamma_4$ are four embedded minimal two-spheres with 
\[  
0<\mc H^2 (\Gamma_1)<\mc H^2 (\Gamma_2)<\mc H^2 (\Gamma_3)<\mc H^2 (\Gamma_4).
\]
This finishes the proof of Theorem \ref{thm:no stable}.
\end{proof}

\subsection{Simon-Smith width upper bound}\label{subsec:2 width upper bound}
In this subsection, we will assume that $N$ and $M$ in Section \ref{subsec:construction of non-compact} are diffeomorphic to the three-ball and the three-sphere, respectively.
Then the construction in Section \ref{subsec:four sweepouts} can be applied. In particular, $\ms P_j$ is well-defined for $j=1,2,3,4$. The goal of this subsection is to give a uniform upper bound for $\mf L_{N_\epsilon}(\ms P_j;\wti g_\epsilon)$ (independent of $\epsilon$), where $N_\epsilon$ and $\wti g_\epsilon$ are as in Section \ref{subsec:construction of non-compact}.

We first introduce the following result which is from the mean curvature flow with surgery.
\begin{lemma}[{\citelist{\cite{HK19}*{Theorem 8.1}\cite{LM20}*{Proposition 3.6}}}]\label{lem:one paramter foliation}
Let $(N,\partial N,g)$ be a compact Riemannian three-ball whose boundary $\partial N$ is a stable minimal sphere. If $(N\setminus \partial N,g)$ does not contain any stable minimal spheres, there exists a minimal sphere $S$ of index one and a smooth foliation $\{S_t\}_{t\in[-1,1]}$ of $N$ such that $S_{-1}$ is a point, $S_0=S$, $S_1=\partial N$ and 
\[  \mc H^2(S_t)\leq \mc H^2(S).\]
\end{lemma}
\begin{proof}
By \cite{KLS19}*{Lemma 8.1} and the Simon-Smith min-max theory, there exists an embedded minimal sphere $S$ with index one. Clearly, $N\setminus S$ has two connected components $N^+$ and $N^-$. Since $S$ is unstable, then $S$ has {\em an expanding neighborhood} $U$ in $N$, that is, $U$ can be foliated by spheres and $\partial U$ has non-zero mean curvature vector pointing away from $U$. Thus one can use the mean curvature flow with surgery to obtain a smooth foliation of $N\setminus U$. Combining with the foliation of $U$, we obtain a foliation of $N$. One can modify the foliation slightly around $\partial U$ to get a smooth foliation satisfying all of the requirements. This completes the proof. 
\end{proof}

Denote by $W$ the area of $S$ in Lemma \ref{lem:one paramter foliation}. Let $(N_{\epsilon},g_{\epsilon})$ and $(M,\wti g_\epsilon)$ be the manifolds constructed in the previous Subsection \ref{subsec:construction of non-compact}. Without loss of generality, we assume that there exists a smooth foliation $\{K_t\}_{t\in[0,1]}$ of $M\setminus (N\setminus \partial N)$ so that $K_0=\partial N$, $K_1=q'\in M\setminus N$ is a point, and 
\begin{equation}\label{eq:foliation out N} 
\mc H^2(K_t)\leq \mc H^2(\partial N)+\epsilon.
\end{equation}
Note that this can be done because one can arbitrarily deform the metric in $M\setminus N$.

\begin{lemma}\label{lem:2 width upper bound}
As above, if $(N\setminus \partial N,g)$ does not contain any stable minimal spheres, then for all sufficiently small $\epsilon$, we have that
\[ 
\mf L_{N_\epsilon}(\ms P_1;\wti g_{\epsilon})\leq W; \quad \mf L_{N_\epsilon}(\ms P_2;\wti g_{\epsilon})\leq 2W.
\]
\end{lemma}
\begin{proof}
By Lemma \ref{lem:one paramter foliation}, there exists a smooth foliation $\{S_t\}_{t\in[-1,1]}$ such that $S_{-1}$ is a point, $S_1=\partial N$ and 
\[\mc H^2(S_t)\leq W.\]
Now we take a sufficiently small constant $\tau>0$, so that $S_{1-\tau}$ is sufficiently close to $\partial N$ (in the $C^3$ topology) as smooth graphs. Fix this $\tau$. Note that for sufficiently small $\epsilon$, $\gamma_\epsilon(\partial N_\epsilon\times \{-2\delta_\epsilon\})$ is also a smooth graph over $\partial N$. Then one can foliate the region between $S_{1-\tau}$ and $\gamma_{\epsilon}(\partial N_\epsilon\times \{-2\delta_\epsilon\})$ by spheres with area close to $\mc H^2(S_{1-\tau})<W$. Recall that (see \ref{item:gepsilon=g along surfaces} in Section \ref{subsec:construction of non-compact}) $\gamma_\epsilon(\partial N_{\epsilon}\times [-2\delta_\epsilon,0])$ has a natural foliation $\{\gamma_\epsilon(\partial N_\epsilon \times \{t\})\}_{t\in[-2\delta_\epsilon, 0]}$ with 
\[  \mc H^2(\partial N_\epsilon\times \{t\};\wti g_\epsilon)= \mc H^2(\partial N_\epsilon\times \{t\};g)< W.  \] 
Combining with the foliation on $M\setminus N$ satisfying \eqref{eq:foliation out N}, we obtain a foliation of $M$ and each leaf has area (w.r.t. $\wti g_\epsilon$) less than or equal to $W$. One can modify the foliation around $S_{1-\tau}$ and $\partial N$ slightly to get a smooth foliation $\{\wti S_t\}_{t\in[-1,1]}$ (by reparametrization) of $(M, \wti g_{\epsilon})$. Since the area of $S_{1-\tau}$ and $\partial N$ is strictly less than $W$, the smooth foliation can be chosen so that 
\[  \mc H^2(\wti S_t)\leq W.\]
This gives the first inequality. Then the second part follows from \cite{HK19}*{Theorem 4.1}.
\end{proof}

\subsection{Existence of four minimal two-spheres}\label{subsec:four minimal S2}
This subsection is devoted to prove the existence of four distinct minimal two-spheres in any Riemannian three-sphere $(M,g)$ which does not contain any degenerate-and-stable minimal two-spheres. Note that this has been proven when $(M,g)$ has no stable minimal spheres; see Theorem \ref{thm:no stable}. The next result gives that if a Riemannian three-ball $(N,g)$ has a stable minimal two-sphere as its boundary, then $(N,g)$ admits at least two unstable minimal two-spheres.

\begin{proposition}
\label{prop:two minimal S2 in ball}
Let $(N,\partial N,g)$ be a compact three-ball with smooth boundary. Suppose that $\partial N$ is a stable minimal two-sphere with a contracting neighborhood in $N$. If $N\setminus \partial N$ does not contain any stable minimal two-sphere, $N\setminus \partial N$ admits at least two distinct embedded minimal two-spheres. 
\end{proposition}
\begin{proof}
The existence of the first minimal two-sphere follows from \cite{KLS19}*{Lemma 8.1} and the Simon-Smith min-max theory. Denote by $W$ the area of this minimal sphere. Without loss of generality, we assume that $N$ contains only finitely many embedded minimal two-spheres. Let $(N_{\epsilon_k},g_{\epsilon_k})\subset (M,\wti g_{\epsilon_k})$ be a sequence of domains constructed in Section \ref{subsec:construction of non-compact}. 
We use $N_k$ and $\wti g_k$ to denote $N_{\epsilon_k}$ and $\wti g_{\epsilon_k}$, respectively. Then by Lemma \ref{lem:2 width upper bound},
\[ 
\mf L_{N_k}(\ms P_2;\wti g_k)\leq 2W.
\]
Thus for all sufficiently large $k$, Proposition \ref{prop:min max in Nk} can be applied to produce embedded minimal two-spheres $\Sigma_k\subset (N\setminus \partial N,g)$ (with multiplicity) that achieves $\mf L_{N_k}(\ms P_2;\wti g_k)$. Since $(N\setminus \partial N,g )$ has no stable minimal two-spheres, $\Sigma_k$ is unstable, which yields that it has multiplicity one.

Summarizing that for all sufficiently large $k$, $\mf L_{N_k}(\ms P_2;\wti g_k)$ is achieved by an embedded minimal two-sphere with multiplicity one in $(N,\partial N)$. Then applying the Lusternik–Schnirelmann theory as in Lemma \ref{lem:LS theory}, one can prove that 
\[\mf L_{N_k}(\ms P_2;\wti g_k)>\mf L_{N_k}(\ms P_1;\wti g_k).\] 
Thus we have proved that $(N\setminus \partial N,g)$ contains at least two distinct embedded minimal two-spheres. This finishes the proof of Proposition \ref{prop:two minimal S2 in ball}.
\end{proof}

Now we are ready to prove the main theorem.
\begin{theorem}[Theorem \ref{thm:theoremA}]
\label{thm:4 minimal spheres in bumpy}
Let $(M,g)$ be a Riemannian three dimensional sphere so that $(M, g)$ does not contain any degenerate-and-stable minimal two-spheres. Then $(M,g)$ contains at least four distinct embedded minimal two-spheres. In particular, if $g$ is bumpy or if $\mr{Ric}_g>0$, then $(M, g)$ admits at least four distinct embedded minimal two-spheres.	
\end{theorem}
\begin{proof}
Without loss of generality, we assume that $M$ contains only finitely many embedded minimal two-spheres. If $(M,g)$ does not contain any stable minimal two-sphere, then the conclusion follows from Theorem \ref{thm:no stable}.

Now we assume that $(M,g)$ contains stable minimal embedded two-spheres. Then by assumption, those stable minimal two-spheres in $(M,g)$ are non-degenerate, that is, strictly stable. Then one can cut $M$ along a collection of pairwise disjoint stable embedded minimal two-spheres. Denote by $N_1,\cdots,N_k$ the metric completion of those connected components. Clearly, at least two of them (denoted by $N_1,N_2$) are diffeomorphic to a three-ball. Then Proposition \ref{prop:two minimal S2 in ball} can be applied to construct two distinct embedded minimal two-spheres in each $N_i\setminus \partial N_i$ for $i=1,2$. It follows that $(M,g)$ contains at least 5 embedded minimal two-spheres. This completes the proof of Theorem \ref{thm:4 minimal spheres in bumpy}.
\end{proof}

\begin{remark}
Finally, we remark that the min-max theory can produce degenerate-and-stable closed minimal surfaces with higher multiplicity by the authors' earlier work \cite{Wang-Zhou22}. Note that even the results in \cite{Wang-Zhou22} were stated for the Almgren-Pitts theory, they also hold true for the Simon-Smith theory.
\end{remark}

\appendix

\section{$C^{1,1}$-estimates for multilayer $\A^h$-stationary boundaries}\label{appen:C11 regularity and estimates}

We recall the $C^{1,1}$-regularity and estimates for multilayer $\A^h$-stationary boundaries in 
\cite{Wang-Zhou-C11-2023}; see also \cite{SS23}*{Section 11}. Fix $h\in C^\infty(M)$ and an open subset $W\subset M$. Let $(\Sigma, \Omega)\in\VC(W)$ be a $C^{1,\alpha}$-boundary\footnote{Note that Definition \ref{def:c11 boundary} can be straightforwardly adapted to the $C^{1,\alpha}$-setting.} in $W$. Suppose that $\Sigma$ decomposes to finitely many ordered $C^{1, \alpha}$-surfaces $\Gamma^1\leq \cdots \leq \Gamma^\ell\subset W$, where $\partial \Gamma^i\cap W =\emptyset$ for each $i=1, \cdots, \ell$. 
\begin{proposition}[\cite{Wang-Zhou-C11-2023}*{Theorem 1.4, Corollary 1.5}]
\label{prop:C11 regularity}
Assume that $(\Sigma, \Omega)$ is $\A^h$-stationary in $W$, and each $\Gamma^i$ has $\|h\|_{L^\infty(M)}$-bounded first variation in $W$. Then each $\Gamma^i$ is $C^{1,1}$ in $W$.  
\end{proposition}
Without loss of generality, for an arbitrary point $q\in \Gamma^1\cap \cdots\cap \Gamma^\ell \cap W$, we may choose a local coordinate system in a neighborhood of $q$, still denoted as $W$, such that each $\Gamma^i$ is written as a graph over the unit disk $B_2$ of the common tangent plane $P=T_q\Gamma^i$, (note that $\{\Gamma^i\}$ meet tangentially at $q$ by assumption). 
By \cite[Remark 1.5]{Wang-Zhou-C11-2023}, we can choose a minimal surface $B'$ (near $B_2$) containing $q$ and rewrite $\Gamma^i$ as graphs of $C^{1,\alpha}$-functions $u^i: B'\to \mb R$
with
\[ u^1\leq u^2\leq \cdots \leq u^\ell. \]
For simplicity, denote by $B'_r$ the geodesic ball in $B'$ centered at $q$.
We recall the $C^{1,1}_{loc}$-estimates of this system.
\begin{proposition}[\cite{Wang-Zhou-C11-2023}*{Theorem 1.4, Corollary 1.5}]
\label{prop:C11 estimates}
Under the same assumption as in Proposition \ref{prop:C11 regularity}, assume further that $\|u^i\|_{C^{1,\alpha}(B'_1)}<1$ for each $i$, then we have
\[ \sum_{i=1}^\ell \|u_i\|_{C^{1,1}(B'_{1/2})}\leq C\big(\sum_{i=1}^\ell \|u_i\|_{C^{1,\alpha}(B'_1)} + \|h\|_{C^{1}(B'_1\times (-1, 1))} \big),\]
for some constant $C>1$ depending only on $\ell$, $\alpha$, and the metric $g$ in $W$.
\end{proposition}

\section{Proof of Lemma \ref{lem:minimizing in small ball}}\label{appen:proof of minimizing lemma in small ball}
The proof is adapted from Colding-De Lellis \cite{Colding-DeLellis03}*{\S 7.4} with minor modifications to include volume terms.  

Denote by $(V^*,\Omega^*)$ the limit of  $(\Sigma_k,\Omega_k)$. Then $(V^*,\Omega^*)$ is $\mc A^h$-stationary in $U$. Denote by $m_0$ the upper bound of $\|V^*\|(M)$. Denote 
\[ d:=\dist_M(U',\partial U).\] 
We take $r_1$ small enough so that for any integral varifold which has $c$-bounded first variation in $B_{r_1}(x)$, one has 
\[  2\frac{\|V\|(B_r(x))}{ r^2} \geq \frac{\|V\|(B_t(x))}{t^2} \quad \text { for all } r_1\geq r\geq t>0,\ x\in U'. \]
Note that $r_1$ depends only on $c$ and $M$. 
Let $\rho_0<\min\{d/10, r_1\}$ be a small constant which will be specified later. Then for all $B_\rho(x)\subset U'$ with $\rho<\rho_0$
\[
\|V^*\|(B_{2\rho}(x))\leq 2\|V^*\|(B_d(x))\cdot d^{-2}\cdot 4\rho^2\leq  8m_0d^{-2}\cdot \rho^2.
\]
By taking sufficiently large $k$, we have
\[   \|\Sigma_k\|(B_{2\rho}(x))< 9m_0d^{-2}\cdot \rho^2.  \]
Then by the slicing theorem \cite{Si}*{\S 28.1}, there exists $\tau_k\in (\rho,2\rho)$ such that
\begin{equation}\label{eq:tauk}
\text{$\Sigma_k$ is transversal to $\partial B_{\tau_k}(x)$, and $\mc H^1(\Sigma_k\cap \partial B_{\tau_k}(x))<9m_0d^{-2} \rho$}.  \end{equation}
Since $\Sigma_k$ is transversal to $\partial B_{\tau_k}(x)$, then $t\mapsto\mc H^1(\Sigma_k\cap \partial B_t(x))$ is continuous at $t=\tau_k$. Hence there exists a small interval $(\sigma_k,s_k)\subset (\rho, 2\rho)$ around $\tau_k$, so that \eqref{eq:tauk} holds for every $\tau\in (\sigma_k,s_k)$.  Now we consider the radial isotopy $\psi\in \mk{Is}\big(B_{s_k}(x)\big)$, so that for some $\eta\ll \sigma_k$ to be specified later,
\[  \psi\big(t, B_{s_k}(x)\big) = B_{s_k}(x) \text{ for all }t\in[0,1], \text{ and }  \psi\big(1, B_{\sigma_k}(x)\big)=B_\eta(x).\] 
By computation, we have
\[ \mc H^2\big(\psi(t,\Sigma_k)\big)\leq \mc H^2(\Sigma_k)+C m_0d^{-2}\rho^2, \]
which implies that 
\begin{equation}\label{eq:psi error}
\mc A^h\big(\psi(t,\Sigma_k, \Omega_k)\big)\leq \mc A^h(\Sigma_k,\Omega_k)+Cm_0d^{-2}\rho^2+ C\rho^3.  
\end{equation}
Here $C$ is a uniform constant which may change from line to line. For simplicity, denote 
\[ \wti \Sigma_k=\psi(1,\Sigma_k), \quad \wti \Omega_k=\psi(1,\Omega_k).\]
Note that any isotopy $\varphi\in \mk{Is}\big(B_\rho(x)\big)$ will correspond to a new isotopy $\wti \varphi:=\psi_1\circ \varphi \circ \psi^{-1}_1$ of $\psi_1(B_\rho(x))\subset B_\eta(x)$, where $\psi_1(\cdot )=\psi(1,\cdot)$. 
A direct computation gives that 
\[ \mc A^h \big(\wti \varphi(t,\wti \Sigma_k,\wti\Omega_k)\big)\leq \mc A^h(\wti \Sigma_k,\wti\Omega_k)+ Cm_1\eta^2\rho^{-2}+C\eta^3, \]
where $m_1:=\sup_t \mc H^2\big(\varphi(t,\Sigma_k)\big)$. By taking sufficiently small $\eta$ (depending on $C,m_1,\rho$), we have 
\begin{equation}\label{eq:wti varphi error}
\mc A^h\big(\wti\varphi(t,\wti \Sigma_k,\wti\Omega_k)\big) \leq \mc A^h(\wti\Sigma_k,\wti \Omega_k)+\rho^3.    \end{equation}
Now let 
\[ \hat  \Sigma_k:=\wti \varphi(1,\wti \Sigma_k)=\psi_1(\varphi(1,\Sigma_k)), \quad \hat \Omega_k:=\wti \varphi(1,\wti \Omega_k)=\psi_1(\varphi(1,\Omega_k)).\]
Finally, we take the $\psi^{-1}$ to deform $(\hat\Sigma_k,\hat \Omega_k)$. Then by the same argument as in \eqref{eq:psi error},
\begin{equation}\label{eq:psi-1 error}
\mc A^h\big(\psi^{-1}(t,\hat \Sigma_k,\hat \Omega_k)\big)\leq \mc A^h\big(\hat \Sigma_k,\hat \Omega_k\big)+Cm_0d^{-2}\rho^2+C\rho^3.
\end{equation}

Now we define an isotopy $\Phi$ by concatenating $\psi$, $\wti \varphi$ and $\psi^{-1}$. Notice that 
\[ \Phi(1,\cdot) = \varphi(1,\cdot).\]
Combining \eqref{eq:psi error}, \eqref{eq:wti varphi error} with \eqref{eq:psi-1 error}, we obtain
\[ \mc A^h\big(\Phi(t,\Sigma_k,\Omega_k)\big)\leq \mc A^h(\Sigma_k,\Omega_k)+ Cm_0d^{-2}\rho^2+C\rho^3.  \]
The lemma follows by taking $\rho_0>0$ so that 
\[C m_0 d^{-2}\rho_0^2+C\rho^3\leq \delta, \quad  \rho_0\leq d/10,\quad \text {and } \rho_0\leq r_1.\]
Here $r_1$ depends only on $M$ and $c$. Hence $\rho_0$ depends only on $M$, $c$, $m_0$, $\delta$ and $\dist_{M}(U',\partial U)$.

\section{Calculations related to construction of Jacobi type fields}\label{appen:jacobi calculations}

In this part, we use the notion of elliptic functionals to present the calculations regarding the PDE satisfied by the height difference of two small graphs surrounding an embedded minimal surface. The setups are as follows.
\begin{itemize}
    \item $\Sigma^n$ denotes a closed embedded minimal hypersurface with a unit normal $\nu$ in a closed Riemannian manifold $(M^{n+1}, g)$. 
    \item $(x, z)\in \Sigma\times (-\delta_0, \delta_0)$ denotes the Fermi coordinates induced by the normal exponential map: $(x, z)\mapsto \exp_x\big(z\nu(x)\big)$. 
    \item Given an open subset $\mc U\subset \Sigma$ and a small $\delta_0>0$, we let $\mc U_{\delta_0} = \mc U\times (-\delta_0, \delta_0)$. 
    \item The area element is a function:
    \[ F: T\Sigma\times (-\delta_0, \delta_0) \to [0, \infty), \]
    defined as follows: for any $(x, z, \bm p)\in \Sigma\times (-\delta_0, \delta_0)\times T_x\Sigma$\footnote{Here we choose to use $(x, z, \bm p)$ to denote a point in $T\Sigma\times(-\delta_0, \delta_0)$ to be coherent with classical notations related to elliptic integrands.}, we use $P_{x, z, \bm p}$ and $P_x$ to denote respectively the $n$-dim parallelograms in $T_{(x, z)}M$ and $T_x\Sigma$ generated by
    \begin{gather*}
    \{e_1, \cdots, e_n\}, \text{ where } e_i = \partial_{x^i} + \lb{\bm p, \partial_{x^i}}_\Sigma \partial_z,\\
    \text{ and }\quad \{\partial_{x^1}, \cdots, \partial_{x^n}\}   
    \end{gather*}
    Then
    \[ F(x, z, \bm p) = \frac{\text{$n$-volume of $P_{x, z, \bm p}$ under $g(x, z)$}}{\text{$n$-volume of $P_x$ under $g(x, 0)$}}.\]
    \item $\sigma(x, z, \bm p)$ denotes the inner product (under $g(x, z)$) between $\partial_z$ and the unit normal of $P_{x, z, \bm p}$  multiplied with $F(x, z, \bm p)$.
\end{itemize}

\begin{lemma}\label{lem:F and sigma are smooth}
    $F(x, z, \bm p)$ and $\sigma(x, z, \bm p)$ are smooth functions over $(x, z, \bm p)$. Moreover, 
    \[ F(x, 0, 0) \equiv 1 \quad \text{ and } \quad \sigma(x, 0, 0) \equiv 1. \]
\end{lemma}

\begin{itemize}
    \item Given a graph $\Graph_u\subset \mc U_{\delta_0}$ of $u\in C^{1,1}(\mc U)$, its $n$-dim area is given by
    \[ \Area(\Graph_u) = \int_{\mc U} F\big(x, u(x), \nabla u(x)\big) \,\mr d\mc H^n(x),\]
    where $d\mc H^n$ is the $n$-dim Hausdorff measure of $\Sigma$. 
    \item Since $\Graph_u$ is a $C^{1,1}$-hypersurface, its generalized mean curvature exists almost everywhere; we denote the one w.r.t. upward unit normal (in the direction of $\partial_z$) by $H_u(x)\in L^\infty(\mc U)$.
\end{itemize}

The first variation formula of $\Area(\Graph_u)$ w.r.t. variations $t\mapsto \Graph_{u+t f}$ for a fixed $f\in C^{1,1}_c(\mc U)$ is given by
\begin{align}
    \delta\Area_u(f) & = \int_{\mc U} \frac{\partial}{\partial \bm p}F\big(x, u(x), \nabla u(x)\big)\cdot \nabla f + \frac{\partial}{\partial z} F\big(x, u(x), \nabla u(x)\big)\cdot f\, \mr d\mc H^n(x)\label{eq:general first variation0}\\ \nonumber
    & = \int_{\mc U} H_u(x)  \cdot \sigma\big(x, u(x), \nabla u(x)\big) \cdot f(x)\, \mr d\mc H^n(x).
\end{align}

We now introduce the following notations:
\begin{align*}
& \bm A(x, z, \bm p) = \Big(\frac{\partial^2}{\partial\bm p^i \partial\bm p^j}F\big(x, z, \bm p\big) \Big)_{1\leq i, j\leq n},\\
& \bm b(x, z, \bm p) = \Big(\frac{\partial^2}{\partial z\partial \bm p^i}F\big(x, z, \bm p\big) \Big)_{1\leq i\leq n}, \\
& d(x, z, \bm p) = \frac{\partial^2}{\partial z^2}F\big(x, z, \bm p\big).
\end{align*}

Given $\varphi\in C^{1,1}(\mc U)$, the second variation formula is given by
\begin{align}
    \delta^2\Area_u(\varphi, f) & := \frac{d}{dt}\Big|_{t=0} \delta\Area_{u+t\varphi}(f)\label{eq:general second variation0} \\ \nonumber
    & = \int_{\mc U}\bm A(x, u, \nabla u )\big(\nabla\varphi, \nabla f\big) + \bm b(x, u, \nabla u)\cdot (\varphi \nabla f  + f \nabla\varphi)\\ \nonumber
    &\quad\quad + d(x, u, \nabla u) \cdot \varphi f\,\mr  d\mc H^n(x). 
\end{align}

Note that since $\Sigma$ is a minimal hypersurface, the above equation \eqref{eq:general second variation0} reduces to the classical second variation formula when $u\equiv 0$:
\begin{equation}\label{eq:second variation when u=0}
\delta^2\Area_0(\varphi, f) = \int_{\mc U} \lb{\nabla\varphi, \nabla f}- \big(\Ric(\nu, \nu)+|A^\Sigma|^2\big) \varphi\cdot f\, \mr d \mc H^n(x).    
\end{equation}

Now let $u^+$ and $u^-$ be two functions in $C^{1,1}(\mc U)$. Then by subtracting \eqref{eq:general first variation0} for $u^+$ and $u^-$, we have
\begin{align}
   \int_{\mc U} \Big[\frac{\partial}{\partial \bm p}F\big(x, u^+, & \nabla u^+\big) - \frac{\partial}{\partial \bm p}F\big(x, u^-, \nabla u^-\big)\Big]\cdot\nabla f \, \mr d\mc H^n(x)\label{eq:subtracting first variation}\\ \nonumber
   & +\int_{\mc U} \Big[\frac{\partial}{\partial z} F\big(x, u^+, \nabla u^+\big) - \frac{\partial}{\partial z} F\big(x, u^-, \nabla u^-\big)\Big] \cdot f\, \mr d\mc H^n(x)\\ \nonumber
   & = \int_{\mc U} \Big[ H_{u^+} \cdot \sigma\big( x, u^+, \nabla u^+\big) - H_{u^-} \cdot \sigma\big(x, u^-, \nabla u^-\big)\Big] \cdot f(x)\, \mr d\mc H^n(x). 
\end{align}
Consider the linear interpolation $u_t:=tu^++(1-t)u^-$, $t\in [0, 1]$. Then we have 
\begin{align*}
\frac{\partial}{\partial \bm p}F\big(x, u^+, \nabla u^+\big) & -\frac{\partial}{\partial \bm p}F\big(x, u^-, \nabla u^-\big)=\int_0^1\frac{d}{dt} \Big[\frac{\partial}{\partial \bm p}F\big(x, u_t, \nabla u_t\big)\Big]\,\mr dt\\
& = \int_0^1 \bm A(x, u_t,\nabla u_t)\cdot \nabla (u^+-u^-)+ \bm b(x, u_t,\nabla u_t)(u^+-u^-) \,\mr dt,
\end{align*}
and
\begin{align*}
\frac{\partial}{\partial z} F\big(x, u^+, \nabla u^+\big) & - \frac{\partial}{\partial z} F\big(x, u^-, \nabla u^-\big)=\int_0^1\frac{d}{dt}\Big[\frac{\partial}{\partial z} F\big(x, u_t,\nabla u_t\big)\Big]\,\mr dt\\
& = \int_0^1 \bm b(x,u_t,\nabla u_t)\cdot \nabla (u^+-u^-) + d(x, u_t,\nabla u_t) \cdot (u^+-u^-)\, \mr dt.
\end{align*}
Let 
\[\wti{\bm A}(x) = \int_0^1 \bm A(x, u_t, \nabla u_t)\, \mr dt,\quad \wti{\bm b}(x) = \int_0^1 \bm b(x, u_t, \nabla u_t)\,\mr dt, \quad \wti{d}(x) = \int_0^1 d(x, u_t, \nabla u_t)\,\mr dt, \]
and
\[ \bm\alpha(x) = \wti{\bm A}(x) - \bm A(x, 0, 0), \quad \bm\beta(x) = \wti{\bm b}(x) - \bm b(x, 0, 0), \quad \zeta(x) = \wti{d}(x) - d(x, 0, 0), \]
and
\[ \sigma^+(x) = \sigma(x, u^+, \nabla u^+), \quad \sigma^-(x) = \sigma(x, u^-, \nabla u^-). \]
Since $F(\cdot, \cdot, \cdot)$ is smooth, we know that $\wti{\bm A}, \wti{\bm b}, \wti{d}, \bm\alpha, \bm\beta, \zeta, \sigma^+, \sigma^-$ are all in $C^{0,1}(\mc U)$. Moreover, using Lemma \ref{lem:F and sigma are smooth}, we know that their norms satisfy:
\begin{gather}
\|\bm\alpha\|_{C^{0,1}(\mc U)} + \|\bm\beta\|_{C^{0,1}(\mc U)} + \|\zeta\|_{C^{0,1}(\mc U)} \leq C \big( \|u^+\|_{C^{1,1}(\mc U)} + \|u^-\|_{C^{1,1}(\mc U)}\big), \label{eq:uniform estimates of coefficients}\\ 
 \|\sigma^+ -1 \|_{C^{0,1}(\mc U)} + \|\sigma^- -1\|_{C^{0,1}(\mc U)} \leq C \big( \|u^+\|_{C^{1,1}(\mc U)} + \|u^-\|_{C^{1,1}(\mc U)}\big). \label{eq:estimates of sigmapm}
\end{gather}
for some uniform constant $C>0$.

Plugging everything back to \eqref{eq:subtracting first variation}, and writing $w = u^+-u^-$ we have
\begin{align}
\int_{\mc U} \wti{\bm A}(x)\big(\nabla w, \nabla f\big) & + \wti{\bm b}(x)\cdot \big(w\nabla f + f\nabla w\big) + \wti{d}(x) w\cdot f \, \mr d\mc H^n(x)\label{eq:equation for height difference} \\ \nonumber
& = \int_{\mc U} \Big(H_{u^+}\cdot\sigma^+( x) - H_{u^-}\cdot\sigma^-(x)\Big) f\, \mr d\mc H^n(x).
\end{align}
Together with \eqref{eq:second variation when u=0}, we can rewrite \eqref{eq:equation for height difference} (by subtracting with \eqref{eq:general second variation0} when $u\equiv 0$ and $\varphi = w$) as 
\begin{align}\label{eq:height diff 2 order pde}
\int_{\mc U}\lb{\nabla w, \nabla f} & - \big(\Ric(\nu, \nu)+|A^\Sigma|^2\big) w\cdot f\, \mr d\mc H^n(x) \\ \nonumber
& = \int_{\mc U} \bm\alpha(x)(\nabla w, \nabla f) + \bm\beta(x)\cdot(w\nabla f + f\nabla w) + \zeta(x) wf\, \mr d\mc H^n(x)\\ \nonumber
& + \int_{\mc U} \big(H_{u^+}\cdot \sigma^+(x) - H_{u^-}\cdot \sigma^-(x) \big)\cdot f\, \mr d\mc H^n(x).
\end{align}

\section{Annuli picking argument}\label{appen:am in a uniform subsequence}
In this appendix, we present a general diagonal argument which has been used several times in this paper. The proof here 
follows that in \cite{Colding-Gabai_Ketover18}*{Lemma A.3}.

Recall that an {\em $L$-admissible collection of annuli} consists of a collection of concentric geodesic annuli 
\[ \mr{An}(p;s_1,r_1),\mr {An}(p;s_2,r_2), \cdots, \mr {An}(p;s_L,r_L), \]
so that $2r_{j+1} < s_j$ for $j=1,\cdots,L-1$. Denote by $\mk A $ the collection of all annuli in $M$.

\begin{proposition}\label{prop:general property}
Let $\mc P_1,\mc P_2,\cdots \subset \mk A$ be countably many sub-collections so that if $\mr{An}\in \mc P_i$, then any sub-concentric annuli of $\mr{An}$ also belongs to $\mc P_i$. Suppose that for any $L$-admissible collection of annuli, $\mc P_k$ contains at least one of them. Then there exists a subsequence (still denoted by $\mc P_k$) so that for each $p\in M$, there exists $r_p>0$ such that for all $s<r<r_p$, $\mr{An}(p;s,r)\in \mc P_k$ for all sufficiently large $k$.
\end{proposition}
\begin{proof}
We first prove a weaker version of the proposition.

\begin{claim}\label{claim:fix p}
Given $p\in M$, there exist a subsequence (still denoted by $\mc P_k$) and $\delta_p>0$ such that for all $s<r<\delta_p$, $\mr{An}(p;s,r)\in \mc P_k$ for all sufficiently large $k$. 
\end{claim}

\begin{proof}[Proof of Claim \ref{claim:fix p}]
Take $0<t_1<r_1$ and a subsequence $\{\mc P_k^1\}\subset \{\mc P_k\}$ so that $\mc P_k^1$ contains $\mr{An}(p;t_1,r_1)$. Such $t_1, r_1$ exist because one can construct an $L$-admissible collection of annuli with outer radius less than any given positive number. We will continue to choose annuli by induction. Suppose that we have chosen $t_j$ and $\{\mc P_k^j\}_k$ such that $\mr{An}(p; t_j, r_1)\in \mc P^j_k$ for all $k$. So long as there exists a subsequence $\{\mc P_k^{j+1}\}\subset \{\mc P_k^j\}$ such that $\mr{An}(p; t_j/2,r_1)\in \mc P_k^{j+1}$ for all $k$, we let $t_{j+1}:=t_j/2$. Otherwise, let $s_1:=t_j/2$. Continuing the argument, we have two possibilities:
\begin{enumerate}
    \item $s_1>0$ and there exists a subsequence $\{\mc P_k'\}\subset \{\mc P_k\}$ so that $\mr{An}(p;s_1,r_1)\notin \mc P_k'$ for all $k$;
    \item there exist $t_1,t_2,\cdots \to 0$ and $\{\mc P_k^j\}\subset \{\mc P_k^{j-1}\}\subset\cdots\subset \{\mc P_k\}$ such that $\mr {An}(p;t_j,r_1)\in \mc P_k^j$ for all $k$.
\end{enumerate}
For the later case, Claim \ref{claim:fix p} follows from a diagonal argument. For the first case, then we can take $r_2<s_1/2$ and use the same argument to find $s_2$. Continuing the argument, we have that either Claim \ref{claim:fix p} holds true or there exist $\mr{An}(p;s_1,r_1)$, $\mr{An}(p;s_2,r_2)$, $\cdots$ and $\{\hat {\mc P}_k^1\}\supset \{\hat{\mc P}_k^2\}\supset\cdots$ such that for each $j\geq 1$, $\mr{An}(p;s_j,r_j)\notin \hat {\mc P}_k^j$ for all $k$. This contradicts the assumption if $j\geq L$. Hence Claim \ref{claim:fix p} is proved.
\end{proof}

Now we will prove the proposition by taking a finite open cover. Given $p\in M$, let $r_p^1$ be the supremum of $r$ so that there exists a subsequence $\{\mc{P}_k'\}\subset \{\mc{P}_k\}$ such that for any $s<r$, $\mr{An}(p;s,r)\in \mc{P}_k'$ for all sufficiently large $k$. By Claim \ref{claim:fix p}, $r_p^1>0$ for each $p\in M$. Let 
\[ \mk r_1:=\frac{1}{2}\sup\left\{r_p^1\,;\,p\in M\right\}.\]
Then we can take $p_1\in M$ so that $r_{p_1}>\mk r_1$, which implies that there exists a subsequence $\{\wti {\mc P}_k^1\}\subset \{\mc P_k\}$  such that for any given $0<s<r<\mk r_1$, $\mr{An}(p_1;s,r)\in\wti {\mc P}_k^1$ for all sufficiently large $k$. Now we define $\mk r_j$, $p_j$ and $\{\wti{\mc P}_k^j\}_k$ inductively. Suppose that we have chosen $\mk r_i$, $p_i$ and $\{\wti{\mc P}_k^i\}_k$ for $i=1,\cdots, j$. Then let $r_p^{j+1}$ be the supremum of $r$ so that there exists a subsequence $\{\mc{P}_k'\}\subset \{\wti{\mc{P}}_k^j\}$ such that for all $0<s<r$, $\mr{An}(p;s,r)\in \mc{P}_k'$ for sufficiently large $k$. Define 
\[
\mk r_{j+1}:=\frac{1}{2}\sup\left\{r_p^{j+1}\,;\,p\in M\setminus \cup_{i=1}^jB_{\mk r_i}(p_i)\right\}.
\]
Then there exist $p_{j+1}$ and $\{\wti {\mc P}_k^{j+1}\}$ so that for any given $0<s<\mk r_{j+1}$, $\mr{An}(p_{j+1};s,\mk r_{j+1})\in \wti{\mc P}_k^{j+1}$ for all $k$ sufficiently large.
Since $\{\wti {\mc P}_k^j\}\subset \{\wti {\mc P}_k^{j-1}\}$, we have that $r_p^{j+1}\leq r_p^j$ for all $p\in M$. It follows that $\mk r_{j+1}\leq \mk r_j$.

Next we will prove that $\{\mk r_j\}$ is finite. Suppose not, observing that $\{B_{\mk r_j/2}(p_j)\}$ are pairwise disjoint balls (since $\mk r_{j+1}\leq \mk r_j$), thus we have $\mk r_j\to 0$ as $j\to \infty$. By the contradiction assumption, there exists $q\in M\setminus \cup_jB_{\mk r_j}(p_{j})$. By the definition of $\mk r_j$, we have $r_q^j\leq 3\mk r_j$. This gives that there exist $\mk s_j>0$ and $k_j>0$ such that $ \mr{An}(q;\mk s_j,3 \mk r_j)\notin \wti {\mc P}_k^j $ for all $k\geq k_j$. Since $\mk r_j\to 0$, then by possibly taking a subsequence of $\{\wti {\mc P}_k^1\}_k$, $\{\wti {\mc P}_k^2\}_k$, $\{\wti {\mc P}_k^3\}_k$, $\cdots$ (not relabelled), we have that $6\mk r_{j+1}<\mk s_j$. Observe that for $k\geq k_L$, $\mr{An}(q;\mk s_j,3\mk r_j)\notin \wti{\mc P}_k^L$ for all $j=1,\cdots ,L$. This leads to a contradiction. Hence $\{\mk r_j\}$ is finite. 

This completes the proof of Proposition \ref{prop:general property}.
\end{proof}

\section{Proof of the Lusternik–Schnirelmann inequality}\label{appen:LS inequality}
\begin{proof}[Proof of Lemma \ref{lem:LS theory}]
We prove the last inequality and the others are similar. 
By the definition of $\mf L(\ms P_4)$, there exist a sequence of $\{\Phi_i:(X_i,Z_i)\to (\oli{\ms X},\ms Y)\}\subset \ms P_4$ such that 
\[   
\mf L(\ms P_4)=\lim_{i\to\infty}\max_{x\in X_i} \Area(\Phi_i(x)).
\]
Denote by $\mc S$ the collection of integral varifolds, with mass equal to $\mf L(\ms P_4)$, whose support are disjoint union of embedded minimal spheres.
Given $\eta_1>0$, define 
\begin{gather*}
    Y_i:=\{x\in X_i:\,\mf F(|\Phi_i(x)|,\mc S)\geq \eta_1\};\quad      K_i:=\oli{X_i\setminus Y_i}.
\end{gather*}
Note that $K_i\subset\interior(X_i)$ for small enough $\eta_1$. Denote by $\iota_1:K_i\to X_i$ and $\iota_2:Y_i\to X_i$ the two natural inclusion maps. Since $\Phi_i\in \ms P_4$, there exists $\bar\lambda\in H^1(\oli{\ms X},\ms Y;\mb Z_2)$ such that $[\Phi_i^*(\bar\lambda)]^4\neq 0\in H^4(X_i,Z_i;\mb Z_2)$. Observe that the following diagram
\[
\begin{tikzcd}
    & H^1(\oli{\ms X},\ms Y;\mb Z_2)\arrow[r,"\wti j^*"] \arrow[d,"\Phi_i^*"]
        & H^1(\oli{\ms X};\mb Z_2)\arrow[d,"\wti{ \Phi_i^*}"]\\
    H^1(X_i,K_i\cup Z_i;\mb Z_2)\arrow[r,"j_1^*"] 
        & H^1(X_i,Z_i;\mb Z_2)\arrow[r,"\iota^*_1"]
            & H^1(K_i\cup Z_i,Z_i;\mb Z_2)\simeq H^1(K_i;\mb Z_2)
\end{tikzcd}
\]
is commutative. Since $\mc S$ is a finite set, one can take $\eta_1$ small enough so that $\wti {\Phi_i^*}(H^1(\oli{\ms X};\mb Z_2))=\{0\}$. To see this, consider the chain of maps
$
\begin{tikzcd}
 K_i \arrow[r, "\iota_1"]
& \interior(X_i) \arrow[r, "\Phi_i"]
& \ms X \arrow[r, "\bm{\iota}"]
&\mc Z_2(S^3; \mb Z_2),
\end{tikzcd}
$
where $\bm{\iota}: \ms X\to \mc Z_2(S^3; \mb Z_2)$ is the natural inclusion map. By the argument in \cite[Section 6]{MN17} and our choice of $K_i$, we know that $(\bm{\iota}\circ\Phi_i\circ\iota_1)^*: H^1(\mc Z_2(S^3; \mb Z_2); \mb Z_2)\to H^1(K_i; \mb Z_2)$ is trivial for small enough $\eta_1$. By the conclusion of the Smale's conjecture \cite{Hat83}, we also know that $\bm{\iota}^*: H^1(\mc Z_2(S^3; \mb Z_2);\mb Z_2) \to H^1(\ms X; \mb Z_2)$ is an isomorphism. All together imply that $(\Phi_i\circ\iota_1)^*: H^1(\ms X;\mb Z_2)\to H^1(K_i;\mb Z_2)$ is trivial, and this implies $\wti {\Phi_i^*}(H^1(\oli{\ms X};\mb Z_2))=\{0\}$.

It then follows that 
\begin{equation*}
    \iota_1^*\circ \Phi_i^*(\bar\lambda)=0.
\end{equation*}
Note that the sequence in the second line is exact. Hence there exists $\alpha\in H^1(X_i,K_i\cup Z_i;\mb Z_2)$ such that $j_1^*(\alpha)=\Phi_i^*(\bar\lambda)$. On the other hand, the following sequence 
\[   
\begin{tikzcd}
H^3(X_i,Y_i;\mb Z_2)\arrow[r,"j_2^*"] 
        &H^3(X_i,Z_i;\mb Z_2)\arrow[r,"\iota^*_2"]
            &H^3(Y_i,Z_i;\mb Z_2)
\end{tikzcd}
\]
is also exact. Since $Y_i\cup K_i=X_i$, we have
\[ j_1^*\big(H^1(X_i,K_i\cup Z_i;\mb Z_2)\big)\cup j_2^*\big(H^3(X_i,Y_i;\mb Z_2)\big)\subset H^4(X_i,X_i;\mb Z_2)=\{0\}.\] 
Together with the fact that 
\[ j_1^*(\alpha)\cup [\Phi_i^*(\bar\lambda)]^3=\Phi_i^*(\bar\lambda)\cup [\Phi_i^*(\bar\lambda)]^3\neq 0,\]
we then conclude that 
\[ [\Phi_i^*(\bar\lambda)]^3\notin \mr{Im} j_2^*=\ker \iota_2^*;\]
that is, $\iota_2^*[\Phi_i^*(\bar\lambda)]^3\neq 0\in H^3(Y_i,Z_i; \mb Z_2)$. Hence we have that $\{\Phi_i:(Y_i,Z_i)\to (\oli{\ms X},\ms Y)\}\subset \ms P_3$. Then by the tightening process (see Section \ref{SS:tightening}), we can derive that
\[ \mf L(\ms P_3)<\mf L(\ms P_4).\]
This completes the proof.
\end{proof}

\bibliographystyle{amsalpha}
\bibliography{minmax}
\end{document}